\documentclass[10pt, letterpaper]{amsart}

\usepackage{bm}
\usepackage{latexsym, amsmath, amstext, amssymb, amsfonts, amscd, bm, array, multirow, amsbsy, mathrsfs}
\usepackage{amsthm}
\usepackage{t1enc}
\usepackage[mathscr]{eucal}
\usepackage{indentfirst}
\usepackage{pb-diagram}
\usepackage{graphicx}
\usepackage{fancyhdr}
\usepackage{fancybox}
\usepackage{enumerate}
\usepackage{color}
\usepackage{tikz-cd}
\usepackage[all]{xy}
\usepackage{hyperref}
\usepackage{tikz}
\usepackage{xparse}
\hypersetup{colorlinks=false,pdfborderstyle={/S/U/W 0}}
\usetikzlibrary{matrix}
\usepackage{upgreek}

\usepackage{url}
\usepackage[sort&compress,comma]{natbib}
\bibpunct{[}{]}{,}{n}{}{,}
\theoremstyle{plain}
\newtheorem{thm}{Theorem}[section]

\newtheorem{prop}[thm]{Proposition}

\newtheorem{lemma}[thm]{Lemma}
\newtheorem{cor}[thm]{Corollary}

\theoremstyle{definition}
\newtheorem{definition}[thm]{Definition}

\theoremstyle{remark}
\newtheorem{remark}[thm]{Remark}
\newtheorem{example}[thm]{Example}

\newtheorem*{ack}{Acknowledgements}

\newcommand{\B}{{\Bbb{B}}}

\newcommand{\End}{\mathrm{End}}
\newcommand{\Aut}{\mathrm{Aut}}

\newcommand{\ev}{\mathrm{ev}}
\newcommand{\eqdef}{\stackrel{{\rm def.}}{=}}

\DeclareFontFamily{U}{rsf}{}
\DeclareFontShape{U}{rsf}{m}{n}{<5> <6> rsfs5 <7> <8> <9> rsfs7 <10-> rsfs10}{}
\DeclareMathAlphabet\Scr{U}{rsf}{m}{n}

\newcommand{\KA}{K\"{a}hler-Atiyah~}

\def\Z{\mathbb{Z}}
\def\C{\mathbb{C}}
\def\R{\mathbb{R}}

\def\H{\mathbb{H}}
\def\K{\mathrm{K}}

\def\rk{{\rm rk}}

\def\Der{{\rm Der}}

\def\GL{\mathrm{GL}}
\def\dd{\mathrm{d}}

\def\AdS{\mathrm{AdS}}

\def\supp{\mathrm{supp}}

\def\Ad{\mathrm{Ad}}

\def\cQ{\mathcal{Q}}

\def\bGamma{\mathbf{\Gamma}}
\def\Id{\mathrm{Id}}
\def\fZ{\mathfrak{Z}}
\def\dfZ{\overset{\cdot}{\fZ}}

\def\dGamma{\overset{\cdot}{\Gamma}}

\newcommand{\be}{\begin{equation*}}
\newcommand{\ee}{\end{equation*}}
\newcommand{\ben}{\begin{equation}}
\newcommand{\een}{\end{equation}}
\newcommand{\beqa}{\begin{eqnarray*}}
	\newcommand{\eeqa}{\end{eqnarray*}}
\newcommand{\beqan}{\begin{eqnarray}}
\newcommand{\eeqan}{\end{eqnarray}}

\newcommand{\nn}{\nonumber}

 \newcommand{\id}{\mathrm{id}}
 \newcommand{\tr}{\mathrm{tr}}
 \newcommand{\diag}{\mathrm{diag}}
\newcommand{\sign}{\mathrm{sign}}

\def\cC{{\mathcal C}}

\def\cB{\Scr B}

\def\cH{\mathcal{H}}
\newcommand{\Sol}{\mathrm{Sol}}
\newcommand{\Conf}{\mathrm{Conf}}

\def\BJ{{\mathbb J}}

\def\cZ{{cal Z}}

\def\Cl{\mathrm{Cl}}

\def\cK{\mathcal{K}}
\def\odd{\mathrm{odd}}
\def\Spin{\mathrm{Spin}}

\def\Pin{\mathrm{Pin}}
\def\Spin{\mathrm{Spin}}

\def\SO{\mathrm{SO}}
\def\O{\mathrm{O}}

\def\cD{\mathcal{D}}

\def\cA{\mathcal{A}}
\def\cE{\mathcal{E}}
\def\hcE{\hat{\cE}}

\def\cP{\mathcal{P}}
\def\cN{\mathcal{N}}

\def\cT{\mathcal{T}}
\def\cF{\mathcal{F}}
\def\cC{\mathcal{C}}

\def\G_2{\mathrm{G_2}}

\def\cS{\mathcal{S}}
\def\cZ{\mathcal{Z}}
\def\cV{\mathcal{V}}

\def\P{\mathbb{P}}

\def\frD{\Psi}
\def\fc{\mathfrak{c}}
\newcommand{\Hom}{{\rm Hom}}

\def\Aut{\mathrm{Aut}}

\def\Im{\mathrm{Im}}

\def\G{\mathrm{G}}

\def\R{\mathbb{R}}

\def\Pic{\mathrm{Pic}}

\def\cW{\mathcal{W}}
\def\cL{\mathcal{L}}
\def\cH{\mathcal{H}}
\def\dd{\mathrm{d}}

\def\bSigma{{\boldsymbol{\Sigma}}}
\def\bS{\mathbf{S}}

\def\AdS{\mathrm{AdS}}
\def\bQ{\mathfrak{Q}}
\def\BY{\mathrm{BY}}
\def\BJ{\mathrm{BJ}}
\def\Span{\mathrm{Span}}
\def\fg{\mathfrak{g}}
\def\B{\mathrm{B}}

\newcommand{\bbGamma}{{\mathpalette\makebbGamma\relax}}
\newcommand{\makebbGamma}[2]{%
  \raisebox{\depth}{\scalebox{1}[-1]{$\mathsurround=0pt#1\mathbb{L}$}}%
}
\def\bGamma{\bbGamma}
\def\Met{\mathrm{Met}}
\def\ss{\mathrm{ss}}
\def\fX{\mathfrak{X}}
\def\cl{\mathrm{cl}}
\def\f{\mathfrak{f}}
\def\fs{\mathfrak{s}}
\def\Ric{\mathrm{Ric}}
\def\grad{\mathrm{grad}}

\begin{document}

\title{Spinors of real type as polyforms and the Generalized Killing equation}

\author[Vicente Cort\'es]{Vicente Cort\'es} \address{Department of
  Mathematics, University of Hamburg, Germany}
\email{vicente.cortes@uni-hamburg.de}
 
\author[Calin Lazaroiu]{Calin Lazaroiu} \address{Center for Geometry
  and Physics, Institute for Basic Science, Pohang, Republic of Korea
  \& Department of Theoretical Physics, Horia Hulubei
  National Institute for Physics and Nuclear Engineering,
  Bucharest, Romania.}  \email{calin@ibs.re.kr, lcalin@theory.nipne.ro}

\author[C. S. Shahbazi]{C. S. Shahbazi} \address{Department of
  Mathematics, University of Hamburg, Germany}
\email{carlos.shahbazi@uni-hamburg.de}


\begin{abstract}
We develop a new framework for the study of generalized Killing
spinors, where generalized Killing spinor equations, possibly with
constraints, can be formulated equivalently as systems of partial
differential equations for a polyform satisfying algebraic relations
in the K\"ahler-Atiyah bundle constructed by quantizing the exterior
algebra bundle of the underlying manifold.  At the core of this
framework lies the characterization, which we develop in detail, of
the image of the spinor squaring map of an irreducible Clifford module
$\Sigma$ of real type as a real algebraic variety in the
K\"ahler-Atiyah algebra, which gives necessary and sufficient
conditions for a polyform to be the square of a real spinor.  We apply
these results to Lorentzian four-manifolds, obtaining a new
description of a real spinor on such a manifold through a certain
distribution of parabolic 2-planes in its cotangent bundle. We use
this result to give global characterizations of real Killing spinors
on Lorentzian four-manifolds and of four-dimensional supersymmetric
configurations of heterotic supergravity. In particular, we find new
families of Einstein and non-Einstein four-dimensional Lorentzian
metrics admitting real Killing spinors, some of which are deformations
of the metric of $\AdS_4$ space-time.
\end{abstract}
 
\maketitle

\setcounter{tocdepth}{1} 
\tableofcontents


\section{Introduction}



\subsection{Background and context}


Let $(M,g)$ be a pseudo-Riemannian manifold of signature $(p,q)$,
equipped with a bundle of irreducible real Clifford modules $S$. If
$(M,g)$ admits a spin structure, then $S$ carries a canonical
connection $\nabla^S$ which lifts the Levi-Civita connection of
$g$. This allows one to define the notions of parallel and Killing
spinors, both of which were studied extensively in the literature
\cite{Hitchin,Wang,MoroianuParallel,Bar,Bohle:2003abk}. Developments
in supergravity and differential geometry (see references cited below)
require the study of more general linear first-order partial
differential equations for spinor fields. It is therefore convenient
to develop a general framework which subsumes all such spinorial
equations as special cases. In order to do this, we assume that $S$ is 
endowed with a fixed connection $\cD\colon \Gamma(S)\to
\Gamma(T^{\ast}M\otimes S)$ (which in practice will depend on various
geometric structures on $(M,g)$ relevant to the specific problem under
consideration) and consider the equation:
\ben
\label{eq:gk}
\cD \epsilon = 0\, 
\een
for a real spinor $\epsilon \in \Gamma(S)$. Solutions to this equation
are called \emph{generalized Killing spinors} with respect to
$\cD$ or simply $\cD$-\emph{parallel spinors} on $(M,g)$. We also
consider {\em linear constraints} of the form:
\ben
\label{eq:lc}
\cQ(\epsilon)=0~~,
\een
where $Q\in \Gamma(\Hom(S,\cW\otimes S))$, with $\cW$ a vector bundle
defined on $M$. Solutions $\epsilon \in \Gamma(M)$ of the system
of equations \eqref{eq:gk} and \eqref{eq:lc} are called
{\em constrained generalized Killing spinors} on $(M,g)$.

The study of generalized Killing spinors can be motivated from various
points of view, such as the theory of spinors on hypersurfaces
\cite{FriedrichHyper,FriedrichKim,BarGM,ContiSalamon,MoroianuSemm} or
Riemannian geometry with torsion
\cite{Friedrich:2001nh,IvanovSpin(7)}. There is nowadays an extensive
literature on the existence and properties of manifolds admitting
generalized Killing spinors for specific connections $\cD$ and in the
presence of various spinorial structures, see for example
\cite{AgricolaFriedrich,AgricolaHoll,FriedrichK,Grunewald,Ikemakhen,Herrera,MoroianuSpinc,
MoroianuSemmI,MoroianuSemmII} and references therein.

Generalized Killing spinors play a fundamental role in supergravity
and string theory \cite{Gibbons:1982fy,Tod:1983pm,Tod:1995jf}. They
occur in these physics theories through the notion of
``supersymmetric configuration'', whose definition involves spinors
parallel under a connection $\cD$ on $S$ which is parameterized by
geometric structures typically defined on fiber bundles, gerbes or
Courant algebroids associated to $(M,g)$
\cite{FreedmanProeyen,Gran:2018ijr,Ortin}. This produces the notion of
\emph{supergravity Killing spinor equations} --- particular instances
of (systems of) constrained generalized Killing spinor equations which
are specific to the physics theory under
consideration. Pseudo-Riemannian manifolds endowed with parameterizing
geometric structures for which such equations admit non-trivial
solutions are called {\em supersymmetric configurations}. They are
called {\em supersymmetric solutions} if they also satisfy the
equations of motion of the given supergravity theory. The study of
supergravity Killing spinor equations was pioneered by P. K. Tod
\cite{Tod:1983pm,Tod:1995jf} and later developed systematically in
several references, including
\cite{Figueroa-OFarrill:2002ecq,Gauntlett:2002nw,Gauntlett:2002fz,Gauntlett:2003wb,
Gauntlett:2003fk,Caldarelli:2003pb,Bellorin:2005hy,Bellorin:2005zc,
LazaroiuB,LazaroiuBII,LazaroiuBC,LBCProc,LBFol1,LBFol2,LBLandscape}. The
study of supersymmetric solutions of supergravity theories provided an
enormous boost to the subject of generalized Killing spinors and to
spinorial geometry as a whole, which resulted in a large body of
literature both in physics and mathematics, the latter of which is
largely dedicated to the case of Euclidean signature in higher
dimensional theories. We refer the reader to
\cite{Friedrich:2001nh,Agricola,Figueroa-OFarrill:2008lbd,Ortin,Garcia-Fernandez:2016azr,
Gran:2018ijr} and references therein for more details and exhaustive
lists of references.

Supergravity Killing spinor equations pose a number of new challenges
when compared to simpler spinorial equations traditionally considered
in the mathematics literature. First, supergravity Killing spinor
equations must be studied for various theories and in various
dimensions and signatures (usually Riemannian and Lorentzian), for
real as well as complex spinors. In particular, this means that every
single case in the modulo eight classification of real Clifford
algebras must be considered, adding a layer of complexity to the
problem. Second, such equations involve spinors parallel under
non-canonical connections coupled to several other objects such as
connections on gerbes, principal bundles or maps from the underlying
manifold into a Riemannian manifold of special type. These objects,
together with the underlying pseudo-Riemannian metric, must be treated
as parameters of the supergravity Killing spinor equations, yielding a
highly nontrivial non-linearly coupled system. Moreover, the
formulation of supergravity theories relies on the
Dirac-Penrose\footnote{When constructing such theories, one views
spinors as sections of given bundles of Clifford modules. The
existence of such bundles on the given space-time is postulated when
writing down the theory, rather than deduced through the associated
bundle construction from a specific classical spinorial structure
assumed on to exist on that spacetime.} rather than on the
Cartan approach to spinors.  As a result, spinors appearing in such
theories need not be associated to a spin structure or other a priory
classical spinorial structure but involve the more general concept of
a (real or complex) Lipschitz structure (see
\cite{FriedrichTrautman,LS2018,Lazaroiu:2016vov,LSProc}). The latter
naturally incorporates the `R-symmetry' group of the theory and is
especially well-adapted for geometric formulations of
supergravity. Third, applications require the study of the moduli
space of supersymmetric solutions of supergravity theories, involving
the metric and all other geometric objects entering their
formulation. This set-up yields remarkably nontrivial moduli problems
for which the automorphism group(oid) of the system is substantially
more complicated than the more familiar infinite-dimensional gauge
group of automorphisms of a principal bundle or the diffeomorphism
group of a compact manifold. Given these aspects, the study of
supergravity Killing spinor equations and of moduli spaces of
supersymmetric solutions of supergravity theories requires methods and
techniques specifically dedicated to their understanding
\cite{Gauntlett:2002nw,Grana:2004bg,Grana:2005jc,Gran:2005wf,Grana:2014rpa,
LazaroiuB,LazaroiuBII,
LazaroiuBC,Garcia-Fernandez:2016azr,Coimbra:2016ydd,
LS2018,Lazaroiu:2016vov,LSProc}. Developing such methods in a
systematic manner is one of the goals of this article.


\subsection{Main results}


One approach to the study of supergravity Killing spinor equations is
the so-called ``method of bilinears''
\cite{Tod:1983pm,Tod:1995jf,Gauntlett:2002nw}, which was successfully
applied in various cases to simplify the local partial differential
equations characterizing certain supersymmetric configurations and
solutions. The idea behind this method is to consider the polyform
constructed by taking the `square' of the Killing spinor (instead of
the spinor itself) and use the corresponding constrained generalized
Killing spinor equations to extract a system of algebraic and partial
differential equations for this polyform, thus producing {\em
necessary} conditions for a constrained generalized Killing spinor to
exist on $(M,g)$. These conditions can also be exploited to obtain
information on the structure of supersymmetric solutions of the
supergravity theory at hand. The main goal of the present work is to
develop a framework inspired by these ideas aimed
at investigating constrained generalized Killing spinors on
pseudo-Riemannian manifolds by constructing a mathematical {\em
equivalence} between real spinors and their polyform squares.

Whereas the fact that the `square of a spinor'
\cite{LawsonMichelsohn,HarveyBook,ACDP} (see Definition
\ref{def:squarespinor} in Section \ref{sec:SpinorsAsPolyforms}) yields
a polyform has been known for a long time (and the square of certain
spinors with particularly nice stabilizers is well-known in specific
-- usually Riemannian -- cases \cite{LawsonMichelsohn}), a proper
mathematical theory to systematically characterize and compute spinor
squares in every dimension and signature has been lacking so far. In
this context, the fundamental questions to be addressed are\footnote{A
systematic approach of this type was first used in references
\cite{LBFol1, LBFol2} for generalized Killing spinor equations in
certain 8-dimensional flux compactifications of M-theory, using the
K\"ahler-Atiyah bundle approach to such problems developed previously
in \cite{LazaroiuBC, LazaroiuB,LazaroiuBII}.}:

\begin{enumerate}
\item {\em What are the {\em necessary and sufficient} conditions for
a polyform to be the square of a spinor, in every dimension and
signature?}
\item {\em Can we (explicitly, if possible) translate constrained
generalized Killing spinor equations into {\em equivalent} algebraic
and partial differential equations for the square polyform?}
\end{enumerate}

\noindent In this work, we solve both questions for irreducible real
spinors when the signature $(p,q)$ of the underlying pseudo-Riemannian
manifold satisfies $p-q \equiv_8 0, 2$, i.e.  when the corresponding
Clifford algebra is simple and of real type. We solve question $(1)$
by fully characterizing the space of polyforms which are the (signed)
square of spinors as the set of solutions of a system of algebraic
equations which define a real affine variety in the space of
polyforms. Every polyform solving this algebraic system can be written
as the square of a real spinor which is determined up to a sign factor
--- and vice-versa. Following \cite{LazaroiuB,LazaroiuBII,LazaroiuBC},
the aforementioned algebraic system can be neatly written using the
\emph{geometric product}. The latter quantizes the wedge product,
thereby deforming the exterior algebra to a unital associative algebra
which is isomorphic to the Clifford algebra. This algebraic system can
be considerably more complicated in indefinite signature than in the
Euclidean case. On the other hand, we solve question $(2)$ in the
affirmative by reformulating constrained generalized Killing spinor
equations on a spacetime $(M,g)$ of such signatures $(p,q)$ as an
equivalent system of algebraic and partial differential equations for
the square polyform. Altogether, this produces an {\em equivalent}
reformulation of the constrained generalized Killing spinor problem as
a more transparent and easier to handle system of partial differential
equations for a polyform satisfying certain algebraic equations in the
K\"ahler-Atiyah bundle of $(M,g)$. We believe that the framework
developed in this paper is especially useful in pseudo-Riemannian
signature and in higher dimensions, where the spin group does not act
transitively on the unit sphere in spinor space and hence
representation theory cannot be easily exploited to understand the
square of a spinor in purely representation theoretic terms. One of
our main results (see Theorem \ref{thm:GCKS} for details and notation)
is:

\begin{thm}
\label{thm:introduction}
Let $(M,g)$ be a connected, oriented and strongly spin
pseudo-Riemannian manifold of signature $(p,q)$ and dimension $d=p+q$,
such that $p-q\equiv_8 0,2$. Let $\cW$ be a vector bundle on $M$ and
$(S,\Gamma,\cB)$ be a paired real spinor bundle on $(M,g)$ whose
admissible pairing has symmetry and adjoint types $\sigma,s\in
\{-1,1\}$. Fix a connection $\cD= \nabla^S - \cA$ on $S$ (where $\cA\in
\Omega^1(M,End(S))$) and a morphism of vector bundles $\cQ\in
\Gamma(End(S)\otimes \cW)$. Then there exists a nontrivial
generalized Killing spinor $\epsilon\in \Gamma(S)$ with respect to the
connection $\cD$ which also satisfies the linear constraint
$\cQ(\epsilon) = 0$ iff there exists a nowhere-vanishing polyform
$\alpha\in \Omega(M)$ which satisfies the following algebraic and
differential equations:
\ben
\label{eq:Fierzglobalintro}
\alpha\diamond \beta\diamond\alpha = 2^{\frac{d}{2}}
(\alpha\diamond\beta)^{(0)}\, \alpha~~,~~
(\pi^{\frac{1-s}{2}}\circ\tau)(\alpha) = \sigma\,\alpha~~,
\een
\ben
\nabla^g\alpha = \hat{\cA}\diamond \alpha + \alpha\diamond
(\pi^{\frac{1-s}{2}}\circ\tau)(\hat{\cA})~~,~~\hat{\cQ}\diamond
\alpha = 0
\een
for every polyform $\beta \in\Omega(M)$, where $\hat{\cA}\in
\Omega^1(M,\wedge T^\ast M)$ and ${\hat \cQ}\in \Gamma(\wedge T^\ast M\otimes
\cW)$ are the symbols of $\cA$ and $\cQ$ while $\pi$, $\tau$ are the
canonical automorphism and anti-automorphism of the K\"ahler-Atiyah
bundle $(\wedge T^{\ast}M,\diamond)$ of $(M,g)$. If $\epsilon$ is
chiral of chirality $\mu\in\{-1,1\}$, then we have to add the
condition:
\be
\ast(\pi\circ\tau)(\alpha) = \mu \, \alpha~~,
\ee
where $\ast$ is the Hodge operator of $(M,g)$. Moreover, any such
polyform $\alpha$ determines a nowhere-vanishing real spinor
$\epsilon\in \Gamma(S)$, which is unique up to a sign and satisfies the
constrained generalized Killing spinor equations with respect to $\cD$
and $\cQ$.
\end{thm}

\noindent When $(M,g)$ is a Lorenzian four-manifold, we say that a
pair of nowhere-vanishing one-forms $(u,l)$ defined on $M$ is {\em
  parabolic} if $u$ and $l$ are mutually orthogonal with $u$ null and
$l$ spacelike of unit norm. Applying the previous result, we obtain
(see Theorem \ref{thm:flagLorentz4d} and Section
\ref{sec:GCKLorentz4d} for detail and terminology):

\begin{thm}
\label{thm:flagLorentz4dintro}
Let $(M,g)$ be a connected and spin Lorentzian four-manifold of
``mostly plus'' signature such that $H^1(M,\Z_2)=0$ and $S$ be a real
spinor bundle associated to the spin structure of $(M,g)$ (which is
unique up to isomorphism). Then there exists a natural bijection
between the set of global smooth sections of the projective bundle
$\P(S)$ and the set of trivializable and co-oriented distributions
$(\Pi,\cH)$ of parabolic 2-planes in $T^\ast M$. Moreover, there exist
natural bijections between the following two sets:
\begin{enumerate}[(a)]
\itemsep 0.0 em
\item The set $\Gamma(\dot{S})/\Z_2$ of sign-equivalence classes of
  nowhere-vanishing real spinors $\epsilon\in \Gamma(S)$.
\item The set of strong equivalence classes of parabolic pairs of
  one-forms $(u,l)\in \cP(M,g)$.
\end{enumerate}
In particular, the sign-equivalence class of a nowhere-vanishing
spinor $\epsilon\in \Gamma(S)$ determines and is determined by a
parabolic pair of one-forms $(u,l)$ considered up to transformations
of the form $(u,l)\rightarrow (-u,l)$ and $l\rightarrow l + c u$ with
$c\in \R$.
\end{thm}

\noindent We use this result to characterize spin Lorentzian
four-manifolds $(M,g)$ with $H^1(M,g)=0$ which admit real Killing
spinors and supersymmetric bosonic heterotic configurations associated
to ``paired principal bundles'' $(P,\fc)$ over such a manifold through
systems of partial differential equations for $u$ and $l$, which we
explore in specific cases. Taking $(M,g)$ to be of signature $(3,1)$,
we prove the following results (see Theorems
\ref{thm:equivalencerealkilling} and \ref{thm:susyheteroticconf}),
where $\ast$ and $\dd^\ast$ denote the Hodge operator and
codifferential of $(M,g)$ while $\nabla^g$ denotes the action of its
Levi-Civita on covariant tensors:

\begin{thm}
\label{thm:equivalencerealkillingintro}
$(M,g)$ admits a nontrivial real Killing spinor with Killing constant
$\frac{\lambda}{2}\in \R$ iff it admits a parabolic pair of one-forms
$(u,l)$ which satisfies:
\be
\nabla^g u = \lambda\, u\wedge l~~,~~ \nabla^g l = \lambda\, (l\otimes
l - g) + \kappa\otimes u
\ee
for some $\kappa\in \Omega^1(M)$. In this case, $u^{\sharp}\in
\fX(M)$ is a Killing vector field with geodesic integral curves.
\end{thm}

\begin{thm}
\label{thm:susyheteroticconfintro}
A bosonic heterotic configuration $(g,\varphi,H,A)$ of $(M,P,\fc)$ is
supersymmetric iff there exists a parabolic pair of one-forms $(u,l)$
which satisfies (here $\rho \eqdef \ast H\in \Omega^1(M)$):
\beqa
& \varphi\wedge u = \ast (\rho\wedge u)~~,~~ \varphi\wedge
u\wedge l = -g^{\ast}(\rho,l) \ast  u~~,~~ -g^{\ast}(\varphi,l)\,
u = \ast (l\wedge u\wedge \rho)~~, \nonumber\\ & g^{\ast}(u,\varphi) =
0~~,~~ g^{\ast}(u,\rho) = 0~~,~~ g^{\ast}(\rho,\varphi) =
0~~,~~ F_A = u\wedge \chi_A~~,\\
& \nabla^g u = \frac{1}{2}
u\wedge\varphi ~~,~~ \nabla^g l = \frac{1}{2} \ast (\rho\wedge
l) + \kappa\otimes u~~,~~ \dd^\ast \rho = 0
\eeqa
for some one-form $\kappa\in \Omega^1(M)$ and some
$\fg_P$-valued one-form $\chi_A \in
\Omega^1(M,\fg_P)$ which is orthogonal to $u$. In this case,
$u^{\sharp} \in \fX(M)$ is a Killing vector field.
\end{thm}

\noindent Let $\H$ be the Poincar\'e half plane with coordinates
$x\in \R$, $y\in \R_{>0}$. Using the last result above, we show (see
Subsection \ref{subsec:SpecialPoincare}) that the following
one-parameter family of metrics defined on $\mathbb{R}^2\times \H$ :
\be
\dd s^2_g = \cF\, (\dd x^v)^2 +\frac{\dd x^v\dd x^u}{y^{2}}   +  
\frac{(\dd x)^2 + (\dd y)^2}{ \lambda^2 y^2} \, \, ~~~~~~~~~~~ \, \, (\lambda\in \R)
\ee
(where $\R^2$ has Cartesian coordinates $x^v,x^u$) admits real Killing
spinors for every $\cF \in
C^{\infty}(\H)$. We also show that these metrics are Einstein with
Einstein constant $\Lambda = -3\lambda^2$ when $\cF$ has the form:
\be
\cF= (a_1 + a_2 x) (a_3 y + \frac{a_4}{y^{2}})~~\mathrm{or}~~\cF =
(a_1 e^{c x} + a_2 e^{-c x}) \left[a_3
\BY(c y) + a_4 \BJ(c y)\right]~~, 
\ee
where $\BY$ and $\BJ$ are the spherical Bessel functions and
$a_1,\ldots , a_4 \in \R$. These give 
deformations of the $\AdS_4$ spacetime of cosmological constant
$\Lambda$, which obtains for $a_1 = a_2 = a_3 = a_4 = 0$.


\subsection{Open problems and further directions}


Theorem \ref{thm:introduction} refers exclusively to real spinors in
signature $p-q \equiv_8 0, 2$, for which the irreducible Clifford
representation map is an isomorphism. It would be interesting to extend
this result to the remaining signatures, which encompass two main
cases, namely real spinors of complex and
quaternionic type (the latter of which can be reducible). This would
yield a rich reformulation of the theory of real spinors through
polyforms subject to algebraic constraints, which could be used to
study generalized Killing spinors in all dimensions and signatures. It
would also be interesting to extend Theorem \ref{thm:introduction} to
other types of \emph{spinorial equations}, such as those
characterizing harmonic or twistor spinors and generalizations
thereof, investigating if it is possible to develop an equivalent
theory exclusively in terms of polyforms.

Some important signatures satisfy the condition
$p-q \equiv_8 0, 2$, most notably signatures $(2,0)$, $(1,1)$, $(3,1)$
and $(9,1)$. The latter two are especially relevant to
supergravity theories and one could apply the formalism developed
in this article to study moduli spaces of supersymmetric solutions
in these cases. Several open problems of analytic, geometric and
topological type exist regarding the heterotic system in four and ten
Lorentzian dimensions, as the mathematical study of its Riemannian
analogue shows \cite{Garcia-Fernandez:2016azr}. Most problems related
to existence, classification, construction of examples and moduli are
open and give rise to interesting analytic and geometric questions on
Lorentzian four-manifolds and ten-manifolds. In this direction, we
hope that Appendix \ref{app:HetSugra} can serve as a brief
introduction to heterotic supergravity in four Lorentzian dimensions
for mathematicians who may be interested in such questions. The local
structure of ten-dimensional supersymmetric solutions to heterotic
supergravity was explored in
\cite{Gran:2007fu,Gran:2007kh,Papadopoulos:2009br}, where that problem
was reduced to a minimal set of partial differential equations on a
local Lorentzian manifold of special type.


\subsection{Outline of the paper}


Section \ref{sec:vectorasendo} gives the
description of rank-one endomorphisms of a vector space which are
(anti-)symmetric with respect to a non-degenerate bilinear pairing
assumed to be symmetric or skew-symmetric. Section
\ref{sec:SpinorsAsPolyforms} develops the algebraic theory of the
square of a spinor culminating in Theorem \ref{thm:reconstruction},
which characterizes it through a system of algebraic
conditions in the K\"ahler-Atiyah algebra of the underlying quadratic
vector space. In Section \ref{sec:GCKS}, we apply this to real spinors
on pseudo-Riemannian manifolds of signature $(p,q)$ satisfying
$p-q\equiv_8 0,2$, obtaining a complete characterization of 
generalized constrained Killing spinors as polyforms satisfying
algebraic and partial differential equations which we list
explicitly. Section \ref{sec:RKSpinors} applies this theory to
real Killing spinors in four Lorentzian dimensions, obtaining a new
global characterization of such. In Section
\ref{sec:Susyheterotic}, we apply the same theory to the study of
supersymmetric configurations of heterotic supergravity, whose
mathematical formulation is explained briefly in an appendix.

\

\noindent{\bf Notations and conventions.}
Throughout the paper, we use Einstein summation over repeated indices. 
We let $R^\times$ denote the group of invertible elements of any
commutative ring $R$. In particular, the multiplicative group of
non-zero real numbers is denoted by $\R^\times=\R$. For any positive
integer $n$, the symbol $\equiv_n$ denotes the equivalence relation of
congruence of integers modulo $n$, while $\Z_n=\Z/n\Z$ denotes the
corresponding congruence group. All manifolds considered in the paper
are assumed to be smooth, connected and paracompact, while all fiber
bundles are smooth. The set of globally-defined smooth sections of any
fiber bundle $F$ defined on a manifold $M$ is denoted by $\Gamma(F)$.
We denote by $G_0$ the connected component of the identity of any Lie
group $G$. Given a vector bundle $S$ on a manifold $M$, the dual
vector bundle is denoted by $S^\ast$ while the bundle of endomorphisms
is denoted by $End(S)\simeq S^\ast\otimes S$. The trivial real line
bundle on $M$ is denoted by $\R_M$. The space of globally-defined
smooth sections of $S$ is denoted by $\Gamma(S)$, while the set of
those globally-defined smooth sections of $S$ which do not vanish
anywhere on $M$ is denoted by $\dGamma(S)$. The complement of the
origin in any $\R$-vector space $\Sigma$ is denoted by $\dot{\Sigma}$
while the complement of the image $0_S$ of the zero section of a
vector bundle $S$ defined on a manifold $M$ is denoted by
$\dot{S}$. The inclusion $\dGamma(S)\subset
\overset{\cdot}{\Gamma(S)}$ is generally strict. If $A$ is any subset
of the total space of $S$, we define:
\ben
\label{GammaA}
\Gamma(A)\eqdef \{s\in \Gamma(S) \,|\, s_m\in A\cap S_m \ \,\,\,\,\,\, \forall m\in
M\}\subset \Gamma(S)~~.
\een
Notice the relation $\dGamma(S)=\Gamma(\dot{S})$. All
pseudo-Riemannian manifolds $(M,g)$ are assumed to have dimension at
least two and signature $(p,q)$ satisfying $p-q\equiv_8 0,2$; in
particular, all Lorentzian four-manifolds have ``mostly plus''
signature $(3,1)$. For any pseudo-Riemannian manifold $(M,g)$, we
denote by $\langle~,~\rangle_g$ the (generally indefinite) metric
induced by $g$ on the total exterior bundle $\Lambda(M)\eqdef \wedge
T^\ast M=\oplus_{k=0}^{\dim M}\wedge^k T^\ast M$. We denote by
$\nabla^g$ the Levi-Civita connection of $g$ and use the same symbol
for its action on tensors.  The equivalence class of an element $\xi$
of an $\R$-vector space $\Sigma$ under the sign action of $\Z_2$ on
$\Sigma$ is denoted by ${\hat \xi}\in \Sigma/\Z_2$ and called the {\em
  sign equivalence class} of $\xi$.

\begin{ack} The work of C. I. L. was supported by grant
IBS-R003-S1. The work of C. S. S. is supported by the Humboldt
Research Fellowship ESP 1186058 HFST-P from the Alexander von Humboldt
Foundation. The research of V.C. and C.S.S. was partially funded by
the Deutsche Forschungsgemeinschaft (DFG, German Research Foundation)
under Germany's Excellence Strategy -- EXC 2121 Quantum Universe --
390833306.
\end{ack}


\section{Representing real vectors as endomorphisms in a paired vector space}
\label{sec:vectorasendo}


Let $\Sigma$ be an $\R$-vector space of positive even dimension $N\geq
2$, equipped with a non-degenerate bilinear pairing $\cB \colon
\Sigma\times \Sigma \to \R$, which we assume to be either symmetric or
skew-symmetric. In this situation, the pair $(\Sigma,\cB)$ is called a {\em
  paired vector space}. We say that $\cB$ (or $(\Sigma,\cB)$) has {\em
  symmetry type} $\sigma\in \{-1,1\}$ if:
\be
\cB(\xi_1,\xi_2) = \sigma \cB(\xi_2,\xi_1) \,\,\,\,\,\, \forall\, \xi_1, \xi_2 \in \Sigma~~.
\ee
Thus $\cB$ is symmetric if it has symmetry type $+1$ and skew-symmetric
if it has symmetry type $-1$.  Let $(\End(\Sigma),\circ)$ be the
unital associative $\R$-algebra of linear endomorphisms of $\Sigma$,
where $\circ$ denotes composition of linear maps. Given
$E\in \End(\Sigma)$, let $E^t\in \End(\Sigma)$ denote the adjoint of
$E$ taken with respect to $\cB$, which is uniquely determined by the
condition:
\be
\cB(\xi_1,E(\xi_2))=\cB(E^t(\xi_1),\xi_2)  \,\,\,\,\,\, \forall \xi_1,\xi_2\in \Sigma~~.
\ee
The map $E\rightarrow E^t$ is a unital anti-automorphism of the
$\R$-algebra $(\End(\Sigma),\circ)$.

\subsection{Tame endomorphisms and the squaring maps}

\begin{definition}
An endomorphism $E\in \End(\Sigma)$ is called {\em tame} if its rank
satisfies $\rk(E)\leq 1$.
\end{definition}

\noindent Thus $E$ is tame iff it vanishes or is of unit rank. Let:
\be
\cT:=\cT(\Sigma)\eqdef \{E\in \End(\Sigma) \,|\, \rk(E)\leq 1\} \subset \End(\Sigma)
\ee
be the real determinantal variety of tame endomorphisms of $\Sigma$ and:
\be
\dot{\cT}:=\dot{\cT}(\Sigma)\eqdef \cT\backslash\left\{ 0\right\}= \{E\in \End(\Sigma) \,|\, \rk(E)=1\}
\ee
be its open subset consisting of endomorphisms of rank one.  We view
$\cT$ as a real affine variety of dimension $2N-1$ in the vector space
$\End(\Sigma)\simeq \R^{N^{2}}$ and $\dot{\cT}$ as a
semi-algebraic variety. Elements of $\cT$ can be written
as:
\be
E = \xi\otimes \beta~~,
\ee
for some $\xi\in\Sigma$ and $\beta\in \Sigma^\ast$, where
$\Sigma^\ast=\Hom(\Sigma,\R)$ denotes the vector space dual to
$\Sigma$. Notice that $\tr(E) = \beta(\xi)$. When $E\in \cT$ is
non-zero, the vector $\xi$ and the linear functional $\beta$ appearing
in the relation above are non-zero and determined by $E$ up to transformations of
the form $ (\xi,\beta)\rightarrow (\lambda \xi, \lambda^{-1}\beta)$
with $\lambda\in \R^\times$. In particular, $\dot{\cT}$ is a manifold
diffeomorphic with the quotient $(\R^N\setminus\{0\})\times
(\R^N\setminus\{0\})/\R^\times$, where $\R^\times$ acts with weights $+1$
and $-1$ on the two copies of $\R^N\setminus\{0\}$.

\begin{definition}
The {\em signed squaring maps} of a paired vector space $(\Sigma,\cB)$
are the quadratic maps $\cE_\pm\colon \Sigma \to \cT$ defined through:
\be
\cE_\pm(\xi)=\pm \xi \otimes \xi^{\ast}  \,\,\,\,\,\, \forall \xi\in \Sigma~~,
\ee
where $\xi^\ast \eqdef \cB(-, \xi)\in \Sigma^\ast$ is the linear map
dual to $\xi$ relative to $\cB$. The map $\cE_+$ is called the {\bf
  positive squaring map} of $(\Sigma,\cB)$, while $\cE_-$ is called
the {\em negative squaring map} of $(\Sigma,\cB)$.
\end{definition}

\noindent Let $\kappa \in \{-1,1\}$ be a sign factor and consider the
open semi-algebraic set $\dot{\Sigma}\eqdef \Sigma\backslash\left\{
0\right\}$. Notice that $\cE_\pm(\xi)=0$ iff $\xi=0$, hence
$\cE_\pm(\dot{\Sigma})\subset \dot{\cT}$. Let $\dot{\cE}_\pm\colon
\dot{\Sigma}\to \dot{\cT}$ be the restrictions of $\cE_\pm$ to
$\dot{\Sigma}$. The proof of the following lemma is immediate.

\begin{lemma}
\label{lemma:2to1E}
For each $\kappa\in \{-1,1\}$, the restricted quadratic map
$\dot{\cE}_{\kappa}\colon \dot{\Sigma}\to \dot{\cT}$ is two-to-one,
namely:
\be
\dot{\cE}_{\kappa}^{-1}(\{\kappa\,\xi\otimes \xi^\ast\}) = 
\left\{-\xi,\xi\right\}  \,\,\,\,\,\, \forall \xi \in \dot{\Sigma}~~.
\ee
Moreover, we have $\cE_{\kappa}(\xi) = 0$ iff $\xi = 0$ and hence
$\cE_\kappa$ is a real branched double cover of its image, which is
ramified at the origin.
\end{lemma}

\subsection{Admissible endomorphisms}

The maps $\cE_\pm$ need not be surjective. To characterize their
images, we introduce the notion of \emph{admissible endomorphism}. Let
$(\Sigma,\cB)$ be a paired vector space of type $\sigma$.

\begin{definition}
An endomorphism $E$ of $\Sigma$ is called {\em $\cB$-admissible} if it
satisfies the conditions:
\be
E\circ E = \tr(E) E~~\mathrm{and}~~ E^{t} = \sigma E~~.
\ee
\end{definition}
\noindent Let:
\be
\cC:=\cC(\Sigma,\cB) \eqdef \left\{ E\in\End(\Sigma)\,\, |\,\, E\circ E = \tr(E) E~~, \,\, E^{t} = \sigma E \right\}
\ee
denote the real cone of $\cB$-admissible endomorphisms of $\Sigma$.

\begin{remark} 
Tame endomorphisms are not related to admissible endomorphisms in any
simple way. A tame endomorphism need not be admissible, since it need
not be (anti-)symmetric with respect to $\cB$. An admissible
endomorphism need not be tame, since it can have rank larger than one
(as shown by a quick inspection of explicit examples in four
dimensions).
\end{remark}

\noindent
Let:
\ben
\label{cZdef}
\cZ :=\cZ(\Sigma,\cB)\eqdef\cT(\Sigma) \cap \cC(\Sigma,\cB)
\een
denote the real cone of those endomorphisms of $E$ which are both tame
and $\cB$-admissible and consider the open set $\dot{\cZ} \eqdef
\cZ\backslash\left\{ 0\right\}$.

\begin{lemma}
\label{lemma:EcE}
We have:
\be
\cZ = \Im(\cE_{+}) \cup \Im(\cE_{-})~~\mathrm{and}~~\Im(\cE_{+}) \cap \Im(\cE_{-})=\{0\}~~.
\ee
Hence an endomorphism $E\in\End(\Sigma)$ belongs to $\dot{\cZ}$ iff
there exists a non-zero vector $\xi \in \dot{\Sigma}$ and a sign
factor $\kappa\in \{-1,1\}$ such that:
\be
E = \cE_{\kappa}(\xi)~~.
\ee
Moreover, $\kappa$ is uniquely determined by $E$ through this
equation while $\xi$ is determined up to sign.
\end{lemma}

\begin{proof}
Let $E\in \dot{\cZ}$. Since $E$ has unit rank, there exists a non-zero
vector $\xi\in\Sigma$ and a non-zero linear functional $\beta\in
\Sigma^\ast$ such that $E = \xi\otimes \beta$.  Since $\cB$ is
non-degenerate, there exists a unique non-zero $\xi_0\in \Sigma$ such
that $\beta = \cB(-,\xi_0)=\xi_0^\ast$. The condition $E^t = \sigma E$
amounts to:
\be
\cB(-,\xi_0) \xi = \cB(-,\xi) \xi_0~~.
\ee
Since $\cB$ is non-degenerate, there exists $\chi\in \Sigma$
such that $\cB(\chi,\xi)\neq 0$, which by the previous equations also
satisfies $\cB(\chi,\xi_0)\neq 0$. Hence:
\be
\xi_0 = \frac{\cB(\chi,\xi_0)}{\cB(\chi,\xi)} \xi=\frac{\cB(\xi_0, \chi)}{\cB(\xi, \chi)} \xi
\ee
and:
\be
E = \frac{\cB(\xi_0, \chi)}{\cB(\xi, \chi)} \xi\otimes\xi^\ast~~.
\ee
Using the rescaling $\xi\mapsto \xi'\eqdef \bigg\vert\frac{\cB(\xi_0,
  \chi)}{\cB(\xi, \chi)}\bigg\vert^{\frac{1}{2}} \xi$, the previous
relation gives $E=\kappa\, \xi'\otimes (\xi')^\ast \in
\Im(\cE_\kappa)$, where $\kappa\eqdef\sign\left(\frac{\cB(\xi_0,
  \chi)}{\cB(\xi, \chi)}\right)$. This implies the inclusion $\cZ
\subseteq \Im(\cE_{+}) \cup \Im(\cE_{-}) $. Lemma \ref{lemma:2to1E}
now shows that $\xi'$ is unique up to sign. The inclusion
$\Im(\cE_{+})\cup \Im(\cE_{-})\subseteq \cZ$ follows by direct
computation using the explicit form $E = \kappa\, \xi \otimes
\xi^{\ast}$ of an endomorphism $E\in \Im(\cE_{\kappa})$, which
implies:
\be
E\circ E = \cB(\xi,\xi) E~~,~~ E^t = \sigma E~~,~~ \tr(E) = \cB(\xi,\xi)~~.
\ee
Combining the two inclusions above gives $\cZ = \Im(\cE_{+}) \, \cup\,
\Im(\cE_{-})$. The relation $\Im(\cE_{+}) \, \cap\,
\Im(\cE_{-})=\{0\}$ follows immediately from Lemma \ref{lemma:2to1E}.
\end{proof}

\begin{definition}
The {\em signature} $\kappa_E\in \{-1,1\}$ of an element $E\in
\dot{\cZ}$ with respect to $\cB$ is the sign factor $\kappa$
determined by $E$ as in Lemma \ref{lemma:EcE}. When $E=0$, we set
$\kappa_E=0$.
\end{definition}

\begin{remark}
Notice that $\kappa_{-E}=-\kappa_{E}$ for all $E\in \cZ$. 
\end{remark}

\noindent In view of the above, define:
\ben
\label{cZpmdef}
\cZ_\pm:=\cZ_\pm(\Sigma,\cB)\eqdef \Im(\cE_\pm)~~.
\een
Then:
\be
\cZ_-=-\cZ_+~~,~~\cZ=\cZ_+\cup \cZ_-~\mathrm{and}~~\cZ_+\cap \cZ_-=\{0\}~~.
\ee
Let $\Z_2\simeq \{-1,1\}$ act on $\Sigma$ and on $\cZ\subset \End(E)$
by sign multiplication. Then $\cE_+$ and $\cE_-$ induce the same map
between the quotients (which is a bijection by Lemma
\ref{lemma:EcE}). We denote this map by:
\ben
\label{eq:hcE}
{\hat \cE}:\Sigma/\Z_2\stackrel{\sim}{\rightarrow} \cZ/\Z_2~~.
\een

\begin{definition}
The bijection \eqref{eq:hcE} is called the {\em class squaring map} of $(\Sigma,\cB)$.
\end{definition}

\subsection{The manifold $\dot{\cZ}$ and the projective squaring map}

Given any endomorphism $A\in \End(\Sigma)$, define a (possibly
degenerate) bilinear pairing $\cB_A$ on $\Sigma$ through:
\ben
\label{eq:cBA}
\cB_A(\xi_1,\xi_2) \eqdef \cB(\xi_1, A(\xi_2))  \,\,\,\,\,\, \forall \xi_1 , \xi_2 \in \Sigma~~.
\een
Notice that $\cB_A$ is symmetric iff $A^t = \sigma A$ and
skew-symmetric iff $A^t = -\sigma A$.

\begin{prop}
The open set $\dot{\cZ}$ has two connected components, which are given
by:
\be
\dot{\cZ}_{+} \eqdef \Im(\dot{\cE}_{+})=\Im(\cE_+)\setminus\{0\}~~\mathrm{and}~~\dot{\cZ}_{-} \eqdef \Im(\dot{\cE}_{-})=\Im(\cE_-)\setminus\{0\}~~,
\ee
and satisfy:
\be
\dot{\cZ}_{+}=\{E\in \cZ|\kappa_E=+1\} \subset \left\{ E\in \cZ\,\,
\vert\,\, \cB_E\geq 0 \right\}~~,~~ \dot{\cZ}_{-} =\{E\in
\cZ|\kappa_E=-1\}\subset \left\{ E\in \cZ\,\, \vert\,\, \cB_E\leq 0
\right\}~~.
\ee
Moreover, the maps $\dot{\cE}_\pm  \colon \dot{\Sigma}\to
\dot{\cZ}_\pm$ define principal $\Z_2$-bundles over
$\dot{\cZ}_\pm$.
\end{prop}

\begin{proof}
By Lemma \ref{lemma:EcE}, we know that $\dot{\cZ} = \dot{\cZ}_{+} \cup
\dot{\cZ}_{-}$ and $\dot{\cZ}_{+}\cap \dot{\cZ}_{-} = \emptyset$. The
open set $\dot{\Sigma}$ is connected because $N=\dim\Sigma\geq 2$.
Fix $\kappa\in \{-1,1\}$. Since the continuous map $\cE_{\kappa}$ surjects onto
$\dot{\cZ}_{\kappa}$, it follows that $\dot{\cZ}_{\kappa}$ is
connected. Let $E\in \dot{\cZ}_{\kappa}$. The pairing $\cB_E$ is symmetric
since $E^t=\sigma E$. By Lemma \ref{lemma:EcE}, we have $E = \kappa\,
\cB(-,\xi_0) \xi_0$ for some non-zero $\xi_0\in \Sigma$ and hence:
\be
\cB_E(\xi,\xi) = \cB(\xi , E(\xi)) = \kappa \vert
\cB(\xi,\xi_0)\vert^2 \,\,\,\,\,\, \forall\,\, \xi \in \Sigma~~.
\ee
Since $\xi_0\neq 0$ and $\cB$ is non-degenerate, this shows that $\cB_E$ is
nontrivial and that it is positive-semidefinite when restricted to $\dot{\cZ}_{+}$ and
negative-semidefinite when restricted to $\dot{\cZ}_{-}$. The remaining
statement follows from Lemmas \ref{lemma:2to1E} and \ref{lemma:EcE}.
\end{proof}

\begin{prop}
\label{prop:admissiblendo}
$\dot{\cZ}(\Sigma,\cB)$ is a manifold diffeomorphic to $\R^\times\times
\R\P^{N-1}$ (where $N=\dim \Sigma$). 
\end{prop}

\begin{proof}
Let $\vert\cdot\vert^2_0$ denotes the norm induced by any
scalar product on $\Sigma$. Then the diffeomorphism:
\be
\dot{\cZ} \xrightarrow{\sim}\R^\times \times \R\P^{N-1}~~,~~ \kappa\,
\xi\otimes\xi^{\ast} \mapsto (\kappa\,\vert\xi\vert_0^2,[\xi])
\ee
satisfies the desired properties.
\end{proof}

\noindent
The maps $\dot{\cE}_\pm\colon \dot{\Sigma}\to \End(\Sigma)\setminus \{0\}$
induce the same map:
\be
\P\cE\colon \P(\Sigma) \to \P(\End(\Sigma))~~,~~ [\xi]\mapsto  [\xi\otimes\xi^{\ast}]~~,
\ee
between the projectivizations $\P(\Sigma)$ and $\P(\End(\Sigma))$ of
the real vector spaces $\Sigma$ and $\End(\Sigma)$. Setting
$\P\cZ(\Sigma,\cB)\eqdef \dot{\cZ}(\Sigma,\cB)/\R^\times\subset \P\End(\Sigma)$, Proposition
\ref{prop:admissiblendo} gives:

\begin{prop}
\label{prop:projectivediff}
The map $\P\cE:\P(\Sigma) \rightarrow \P\cZ(\Sigma,\cB)$ is a diffeomorphism.
\end{prop}

\begin{definition}
$\P\cE\colon \P(\Sigma) \stackrel{\sim}{\to} \P\cZ(\Sigma,\cB)$
is called the {\em projective squaring map} of $(\Sigma,\cB)$.
\end{definition}

\subsection{Tamings of $\cB$}

The bilinear form $\cB_A$ defined in equation \eqref{eq:cBA} is
non-degenerate iff the endomorphism $A\in \End(\Sigma)$ is
invertible. The following result is immediate:

\begin{prop}
\label{prop:operator}
Let $\cB'$ be a non-degenerate symmetric pairing on $\Sigma$.  Then
there exists a unique endomorphism $A\in \GL(E)$ (called the {\bf
  operator of $\cB'$ with respect to $\cB$}) such that
  $\cB'=\cB_A$. Moreover, $A$ is invertible and satisfies $A^t=\sigma
  A$. Furthermore, the transpose $E^T$ of any endomorphism
  $E\in \End(\Sigma)$ with respect to $\cB'$ is given by:
\ben
\label{eq:ET}
E^T=(A^{-1})^t E^t A^t=A^{-1} E^t A
\een
and in particular we have $A^T=A^t=\sigma A$.
\end{prop}

\begin{remark}
Replacing $\xi_2$ by $A^{-1}\xi_2$ in the relation
$\cB'(\xi_1,\xi_2)=\cB(\xi_1,A\xi_2)$ gives:
\ben
\label{eq:tamingmatrix}
\cB(\xi_1,\xi_2)=\cB'(\xi_1,A^{-1}\xi_2)~~,
\een
showing that $A^{-1}$ is the operator of $\cB$ with respect to $\cB'$.
\end{remark}

\begin{definition}
We say that $A\in \End(\Sigma)$ is a {\em taming} of $\cB$ if $\cB_A$
is a scalar product.
\end{definition}

\noindent Let $A\in \End(\Sigma)$ be a taming of $\cB$ and denote by
$(-,-)\eqdef \cB_A$ the corresponding scalar product. Relation
\eqref{eq:tamingmatrix} shows that the matrix ${\hat B}$ of $\cB$ with
respect to a $(~,~)$-orthonormal basis $\{e_1,\ldots, e_N\}$ of
$\Sigma$ is the inverse of the matrix ${\hat A}$ of $A$ in the the
same basis. Distinguish the cases:
\begin{enumerate}[1.]
\itemsep 0.0em
\item When $\cB$ is symmetric, its operator $B$ with respect to
  $(~,~)$ and the taming $A=B^{-1}$ of $\cB$ are $(~,~)$-symmetric and
  can be diagonalized by a $(-,-)$-orthogonal linear transformation of
  $\Sigma$.  If $(p,q)$ is the signature of $\cB$, Sylvester's theorem
  shows that we can choose the basis $\{e_i\}$ such that:
\be
{\hat A}=\diag(+1,\ldots, +1,-1,\ldots, -1)~~,
\ee
with $p$ positive and $q$ negative entries. With this choice, we have $A^2=\Id$.
\item When $\cB$ is skew-symmetric, we have:
\be
\cB(\xi_1,\xi_2)=-(\xi_1,J(\xi_2))~~,~~\mathrm{i.e.}~~(\xi_1,\xi_2)=\cB(\xi_1,J(\xi_2))~~,
\ee
where $J$ is a $(-,-)$-compatible complex structure on $\Sigma$. This
gives $A^{-1}=-J$ and hence $A=J$, which is antisymmetric with respect
to both $(-,-)$ and $\cB$. Setting $N=2n$, we can choose
$\{e_1,\ldots, e_n\}$ to be a basis of $\Sigma$ over $\C$ which is
orthonormal with respect to the Hermitian scalar product defined by
$(-,-)$ and $J$ and take $e_{n+i}=Je_i$ for all $i=1,\ldots, n$. Then
the basis $\{e_1,\ldots, e_N\}$ over $\R$ is $(-,-)$-orthonormal while
being a Darboux basis for $\cB$ and we have:
\be
{\hat A}={\hat J}=\left[\begin{array}{cc} 0 & I_n\\ -I_n & 0 \end{array}\right]~~,
\ee
where $I_n$ is the identity matrix of size $n$.
\end{enumerate}

\begin{prop}
\label{prop:Aadm}
Let $\cB'$ be a non-degenerate symmetric pairing on $\Sigma$, $A$ be
its operator with respect to $\cB$ and $E\in \End(\Sigma)$ be an
endomorphism of $\Sigma$. Then the endomorphism $E_A\eqdef E\circ
A\in \End(\Sigma)$ is $\cB'$-admissible iff the following relations
hold:
\ben
\label{eq:Aadm}
E^t=\sigma E~~\mathrm{and}~~E\circ A\circ E=\tr(E\circ A) E~~.
\een
\end{prop}

\begin{proof}
Let $~^T$ denote transposition of endomorphisms with respect to $\cB'$.
By definition, $E_A$ is $\cB'$-admissible if:
\ben
\label{Eadm}
E_A^T=E_A~~\mathrm{and}~~E_A^2=\tr(E_A)E_A~~.
\een
Since $A$ is invertible, the second of these conditions amounts to the
second relation in \eqref{eq:Aadm}. On the other hand, we have:
\be
E_A^T=(EA)^T=A^TE^T=A^t=A^t (A^{-1})^t E^t A^t=E^t A^t=\sigma E^t A~~,
\ee
where we used Proposition \ref{prop:operator}. Hence the first
condition in \eqref{Eadm} is equivalent with $\sigma E^t A=E A$,
which in turn amounts to $E^t=\sigma E$ since $A$ is invertible.
\end{proof}

\subsection{Characterizations of tame admissible endomorphisms}
  
\noindent The following gives an open subset of the cone $\cC$
of admissible endomorphisms which consists of rank one elements.

\begin{prop}
\label{prop:tametracenon0}
Let $E\in \cC(\Sigma,\cB)$ be a $\cB$-admissible endomorphism of $\Sigma$. If
$\tr(E)\neq 0$, then $E$ is of rank one.
\end{prop}

\begin{proof}
Define $P\eqdef \frac{E}{\tr(E)}$. Then $P^2 = P$ (which implies
$\rk(P)=\tr(P)$) and $\tr(P) = 1$, whence $\rk(E) = \rk(P) =
\tr(P)=1$.
\end{proof}

\noindent Define
\be
\cK_0\eqdef \left\{ \xi \in \Sigma\,\, \vert\,\, \cB(
\xi,\xi) = 0 \right\}~~,~~ \cK_{\mu}\eqdef \left\{ \xi \in \Sigma\,\, \vert\,\, \cB(
\xi,\xi) = \mu \right\}~~,
\ee
where $\mu\in \left\{ -1, +1\right\}$. When $\cB$ is symmetric, the
set $\cK_0\subset \Sigma$ is the isotropic cone of $\cB$ and
$\cK_{\mu}$ are the positive and negative unit ``pseudo-spheres''
defined by $\cB$. When $\cB$ is skew-symmetric, we have $\cK_0 =
\Sigma$ and $\cK_{\mu} = \emptyset$. Lemma \ref{lemma:EcE} and
Proposition \ref{prop:tametracenon0} imply:

\begin{cor}
Assume that $\cB$ is symmetric (i.e. $\sigma=+1$). For any $\mu\in
\{-1,1\}$, the set $\cE_+(\cK_{\mu})\cup \cE_-(\cK_{\mu})$ is the
real algebraic submanifold of $\End(\Sigma)$ given by:
\be
\cE_+(\cK_{\mu})\cup \cE_-(\cK_{\mu}) = \left\{ E\in \End(\Sigma) \,\, \vert\,\, E\circ E =
\mu\,E~~, \,\, E^{t} = E~~, \,\, \tr(E) = \mu \right\}~~.
\ee
\end{cor}

\begin{prop}
\label{prop:tamescalarprod}
If $\cB$ is a scalar product, then every non-zero $\cB$-admissible
endomorphism $E\in \cC\setminus \{0\}$ is tame, whence $\cZ = \cC$. In
this case, the signature of $E$ with respect to $\cB$ is given by:
\ben
\label{eq:kappascalarprod}
\kappa_E=\sign(\tr(E))~~.
\een
\end{prop}

\begin{proof}
Let $E\in\cC$. By Proposition \ref{prop:tametracenon0}, the first
statement follows if we can show that $\tr(E)\neq 0$ when $E\neq
0$. Since $E$ is admissible, it is symmetric with respect to the
scalar product $\cB$ and hence diagonalizable with eigenvalues
$\lambda_1,\hdots , \lambda_N\in \R$. Taking the trace of equation
$E^2 = \tr(E)\, E$ gives:
\be
\tr(E)^2 = \sum_{i=1}^{N} \lambda_i^2~~.
\ee
Since the right-hand side is a sum of squares, it vanishes iff
$\lambda_1 = \hdots = \lambda_N = 0$, i.e. iff $E=0$. This proves the
first statement.  To prove the second statement, recall from Lemma
\ref{lemma:EcE} that any non-zero tame admissible endomorphism $E$ has
the form $E=\kappa_E \xi\otimes \xi^\ast$  for some $\xi\in
\dot{\Sigma}$. Taking the trace of this relation gives:
\be
\tr(E)=\kappa_E \xi^\ast(\xi)=\kappa_E \cB(\xi,\xi)~~,
\ee
which implies \eqref{eq:kappascalarprod} since $\cB(\xi,\xi)>0$. 
\end{proof}

\noindent A quick inspection of examples shows that there exist
nontrivial admissible endomorphisms which are not tame (and thus satisfy $\tr(E)
= 0$) as soon as there exists a totally isotropic subspace of $\Sigma$
of dimension at least two. In these cases we need to impose further
conditions on the elements of $\cC$ in order to guarantee tameness. To
describe such conditions, we consider the more general equation:
\be
E\circ A\circ E = \tr(A\circ E) E  \,\,\,\,\,\, \forall\,\, A\in \End(\Sigma)~~,
\ee
which is automatically satisfied by every $E\in \Im(\cE_+)\cup \Im(\cE_-)$.

\begin{prop}
\label{prop:characterizationtamecone}
The following statements are equivalent for any endomorphism $E\in
\cC$ which satisfies the condition $E^t=\sigma E$:
\begin{enumerate}[(a)]
\item $E$ is $\cB$-admissible and $\rk(E) = 1$.
\item We have $E^2=\tr(E) E$ and there exists an endomorphism
  $A\in \End(\Sigma)$ satisfying:
\ben
\label{eq:conditionstame0}
E\circ A\circ E = \tr(E\circ A) E~~\mathrm{and}~~ \tr(E\circ A)\neq 0~~.
\een
\item $E\neq 0$ and the relation:
\ben
\label{eq:conditionstame1}    
E\circ A\circ E = \tr(E\circ A) E~~,  
\een
is holds for every endomorphism $A\in \End(\Sigma)$.
\end{enumerate}
\end{prop}

\begin{proof}
We first prove the implication $(b)\Rightarrow (a)$. By Proposition
\ref{prop:tametracenon0}, it suffices to consider the case $\tr(E) =
0$. Assume $A\in\End(\Sigma)$ satisfies
\eqref{eq:conditionstame0}. Define:
\be
A_\epsilon = \Id +\frac{\epsilon}{\tr(E\circ A)} A~~,
\ee
where $\epsilon \in \R_{>0}$ is a positive constant. For $\epsilon>0$
small enough, $A_{\epsilon}$ is invertible and the endomorphism
$E_{\epsilon} \eqdef E\circ A_{\epsilon}$ has non-vanishing trace
given by $\tr(E_{\epsilon}) = \epsilon$.  The first relation in
\eqref{eq:conditionstame0} gives:
\be
E_{\epsilon}\circ E_{\epsilon} = \epsilon E_{\epsilon}~~.
\ee
Hence $P\eqdef \frac{1}{\epsilon} E_{\epsilon}$ satisfies $P^2 = P$
and $\tr(P) = 1$, whence $\rk(E_\epsilon)=\rk(P)=1$. Since
$A_{\epsilon}$ is invertible, this implies $\rk(E) = 1$ and hence $(a)$
holds.

The implication $(a)\Rightarrow (c)$ follows directly from Lemma
\ref{lemma:EcE}, which shows that $E\in \Im(\cE_\kappa)$ for some sign
factor $\kappa$. For the implication $(c)\Rightarrow (b)$, notice
first that setting $A=\Id$ in \eqref{eq:conditionstame1} gives
$E^2=\tr(E)$.  Non-degeneracy of the bilinear form induced by the
trace on the space $\End(\Sigma)$ now shows that we can choose $A$ in
equation \eqref{eq:conditionstame1} such that $\tr(E\circ A)\neq 0$.
\end{proof}

\begin{prop}
\label{prop:tameconeA}
Let $A$ be a taming of $\cB$ and let $\cB'\eqdef \cB_A$ be the
corresponding scalar product on $\Sigma$. Then the following
statements are equivalent for any endomorphism $E\in \End(\Sigma)$:
\begin{enumerate}[(a)]
\itemsep 0.0em
\item $E$ is $\cB$-admissible and $\rk(E) = 1$.
\item $E$ satisfies $E^t=\sigma E$ as well as conditions
  \eqref{eq:conditionstame0} with respect to $A$.
\end{enumerate}
In this case, there exists a non-zero vector $\xi\in \Sigma$ such
that:
\be
E=\kappa \, \xi\otimes \xi^\ast=\kappa'\, \xi\otimes \xi^\vee~~,
\ee
where $\xi^\vee=\sigma \xi^\ast\circ A$ denotes the linear functional
dual to $\xi$ with respect $\cB'$ while:
\be
\kappa'=\sign(\tr(E\circ A)) \in \{-1,1\}
\ee
is the signature of $E\circ A$ with respect to $\cB'$ and:
\ben
\label{eq:kappa}
\kappa=\sigma \, \kappa'=\sigma\, \sign(\tr(E\circ A)) \in \{-1,1\}
\een
is the signature of $E$ with respect to $\cB$.
\end{prop}

\begin{proof}
Set $E_A\eqdef E\circ A$. To prove the implication $(a)\Rightarrow
(b)$, assume that $E$ is $\cB$-admissible and of rank one.  Then
Proposition \ref{prop:characterizationtamecone} shows that we have
$E\circ A\circ E=\tr(E\circ A) E$. Since $A$ is $\cB$-admissible, we
also have $E^t=\sigma E$. It follows that $E_A$ is $\cB'$-admissible
by Proposition \ref{prop:Aadm}. Since $\cB'$ is a scalar product,
Proposition \ref{prop:tamescalarprod} implies that $E_A$ is tame and
hence $\tr(E_A)\neq 0$ since $E_A$ is non-zero. Thus $\tr(E\circ
A)\neq 0$. Combining everything, this shows that (b) holds.

\

\noindent To prove the implication $(b)\Rightarrow (a)$, assume that
$E$ satisfies $E^t=\sigma E$ as well as \eqref{eq:conditionstame0}. By
Proposition \ref{prop:Aadm}, this implies that $E_A$ is
$\cB'$-admissible. Since $\tr(E_A)=\tr(E\circ A)\neq 0$, Proposition
\ref{prop:tametracenon0} implies $\rk (E_A)=1$.  Hence $E_A$ is
nonzero, tame and admissible with respect to the scalar product
$\cB'$. We thus have $E_A=\kappa' \xi\otimes \xi^\vee$, where
$\kappa'=\sign(\tr(E_A))=\sign(\tr(E\circ A))$ (see Proposition
\ref{prop:tamescalarprod}) and $\xi\in \Sigma$ is a non-zero
vector. Here $\xi^\vee \in \Sigma^\ast$ is the dual of $\xi$ with
respect to $\cB'$, which is given by: \be \xi^\vee(\xi')\eqdef
\cB'(\xi',\xi)=\cB(\xi',A\xi)=\cB(A^t\xi',\xi)=\xi^\ast(A^t\xi')=\sigma
\xi^\ast(A\xi')~~, \ee i.e. $\xi^\vee=\sigma \xi^\ast\circ A$. Thus
$E\circ A=E_A=\sigma \kappa_A \xi\otimes \xi^\ast\circ A$ and hence
$E=\kappa \xi\otimes \xi^\ast$ (with $\kappa\eqdef \sigma\, \kappa'$)
because $A$ is invertible. Thus $E$ belongs to $\Im(\cE_{\kappa})$ and
hence is $\cB$-admissible and of rank one.
\end{proof}

\begin{remark}
When $A$ is a taming of $\cB$, Proposition \ref{prop:tameconeA} shows
that conditions \eqref{eq:conditionstame0} and the condition
$E^t=\sigma E$ automatically imply $E^2=\tr(E) A$ and hence tameness
of $E$. In this case, Proposition \ref{prop:characterizationtamecone}
shows that $E$ also satisfies $E\circ B\circ E=\tr(E\circ B) E$ for {\em
  any} $B\in \End(\Sigma)$.
\end{remark}

\noindent Proposition \ref{prop:tameconeA} gives the following
characterization of $\P\cZ$.

\begin{cor} 
Let $A$ be a taming of $\cB$. Then the image of the projective
squaring map $\P\cE$ is the real algebraic submanifold of
$\P(\End(\Sigma))$ given by:
\be
\P\cZ(\Sigma,\cB) = \left\{ E\in \P(\End(\Sigma))  \,\,\vert  \,\, E\circ E =
\tr(E) E\ , \, E^{t} =\sigma E~~, \, E\circ A\circ E = \tr(E\circ A) E\right\}~~.
\ee
\end{cor}

\subsection{Two-dimensional examples}

Let $\Sigma$ be a two-dimensional $\R$-vector space with
basis $\Delta\eqdef \left\{e_i\right\}_{i=1,2}$. Any vector $\xi\in
\Sigma$ expands as:
\be
\xi = \xi_1 e_1 + \xi_2 e_2~~\mathrm{with}~~\xi_1, \xi_2 \in \R~~.
\ee
Let $E_\xi\eqdef \cE_\kappa(\xi)\eqdef\kappa \xi\otimes
\xi^{\ast}\in \End(\Sigma)$, where $\kappa\in \{-1,1\}$. For any
$S\in \End(\Sigma)$, let ${\hat S}$ denote the matrix of $S$ in the
basis $\Delta$.

\begin{example}
\label{ep:2dEuclidean}
Let $\cB$ be a scalar product on $\Sigma$ having $\Delta$ as an
orthonormal basis. Then:
\be
{\hat E}_\xi = \kappa \left( \begin{array}{ccc} \xi_1^2 &  \xi_1 \xi_2
  \\\xi_1 \xi_2 &  \xi_2^2 \end{array} \right)
\ee
and the relations $E^2_{\xi} = \tr(E_{\xi}) E_{\xi}$ and $E_{\xi}^t =
E_{\xi}$ follow from this form. Let $E\in \End(\Sigma)$
satisfy $E^2 = \tr(E) E$ and $E^t = E$. The second of these conditions
implies:
\be
{\hat E} = \left( \begin{array}{ccc} k_1 & b \\ b &
  k_2 \end{array}\right)~~(\mathrm{with}~b,k_1,k_2\in \R)~~.
\ee
Condition $E^2 = \tr(E) E$ amounts to $b^2 = k_1 k_2$, implying that
$k_1$ and $k_2$ have the same sign unless at least one of them
vanishes (in which case $b$ must also vanish). Since $E$ is
$\cB$-symmetric (and hence diagonalizable), its trace $\tr(E)=k_1+k_2$
vanishes iff $E=0$. Assume $E\neq 0$ and set:
\be
\kappa\eqdef \sign(\tr(E))=\sign(k_1+k_2)~~,~~\xi_1\eqdef
\sqrt{|k_1|}~~,~~\xi_2\eqdef \sign(b) \kappa \sqrt{|k_2|}~~.
\ee
Then $k_1 = \kappa \xi_1^2$, $k_2 = \kappa \xi_2^2$ and $b=\kappa
\xi_1\xi_2$, showing that $E=E_\xi$ for
some $\xi\in \Sigma\setminus \{0\}$. Hence conditions $E^2 = \tr(E) E$
and $E = E^{t}$ characterize endomorphisms of the form $E_\xi$.
\end{example}

\begin{example}
\label{ep:2dsplitsig}
Let $\cB$ be a split signature inner product on $\Sigma$ having
$\Delta$ as an orthonormal basis:
\be
\cB(e_1,e_1) = 1~~,~~ \cB(e_2,e_2) = - 1~~,~~ \cB(e_1,e_2) = \cB(e_2,e_1) = 0~~.
\ee
Then $\cB$ is tamed by the operator $A$ with matrix ${\hat A}=\diag(+1,-1)$,
which corresponds to the scalar product $(-,-)$ defined on $\Sigma$
through:
\be
(e_1,e_1) = (e_2,e_2) =  1~~,~~ (e_1,e_2) = (e_2,e_1) = 0~~.
\ee
We have:
\be
{\hat E}_\xi = \kappa \left(\begin{array}{ccc} \xi_1^2 &
  -\xi_1 \xi_2 \\ \xi_1 \xi_2 & -\xi_2^2 \end{array} \right)
\ee
and the relations $E^2_{\xi} = \tr(E_{\xi}) E_{\xi}$ and $E_{\xi}^t =
E_{\xi}$ follow directly from this form, where $~^t$
denotes the adjoint taken with respect to $\cB$. 
Let $E\in \End(\Sigma)$ satisfy $E^t = E$. Then:
\be
{\hat E} = \left( \begin{array}{ccc} k_1 & -b \\ b & k_{2} \end{array}\right)\,~~\mathrm{and}
~~{\hat E}{\hat A} = \left( \begin{array}{ccc} k_1 & b \\ b & -k_{2} \end{array}\right)\,~~(\mathrm{with}~b,k_1,k_2\in \R)~~.
\ee
Direct computation shows that the conditions $E^2 = \tr(E) E$ and $E\circ A\circ
E=\tr(E\circ A) E$ are equivalent to each other in this
two-dimensional example and amount to the relation $b^2 = -k_1 k_2$,
which implies that $E$ vanishes iff $k_1=k_2$. Let us assume that
$E\neq 0$ and set:
\be
\kappa\eqdef \tr(E\circ A)=\sign(k_1-k_2)~~,~~\xi_1\eqdef  \sqrt{|k_1|}~~,~~\xi_2\eqdef \sign(b) \kappa \sqrt{|k_2|}~~,
\ee
where $\sign(b)\eqdef 0$ if $b=0$. Then it is easy to see that $k_1 =
\kappa \xi_1^2$, $k_2 =-\kappa \xi_2^2$ and $b= \kappa \xi_1\xi_2$,
which implies $E=E_\xi$. In this example endomorphisms $E$ that
can be written in the form $E_\xi$ are characterized by the condition $E^t=E$,
together with either of the two equivalent conditions $E^2=\tr(E) E$ or $E\circ
A\circ E=\tr(E\circ A) E$.  Notice that $\tr(E)=k_1+k_2$ can vanish in
this case. However, in this low-dimensional example, the conditions
$E\circ E = \tr(E) E$ and $E^{t} = E$ suffice to characterize
endomorphisms of the form $E_\xi$, including those which satisfy
$\tr(E) = 0$.
\end{example}

\begin{example}
\label{ep:2dskew}
Let $\cB$ a symplectic pairing on $\Sigma$ having $\Delta$ as a Darboux
basis:
\be
\cB(e_1,e_1) = \cB(e_2,e_2) =  0~~,~~ \cB(e_1,e_2) =
- \cB(e_2,e_1) = 1~~.
\ee
The complex structure $A$ of $\Sigma$ with matrix given by:
\be
{\hat A}=\left( \begin{array}{ccc} 0 & 1 \\ -1
  & 0 \end{array} \right)
\ee
tames $\cB$ to the scalar product $(-,-)$ defined through:
\be
(e_1,e_1) = (e_2,e_2) =  1~~,~~ (e_1,e_2) =(e_2,e_1) = 0~~.
\ee
We have:
\ben
\label{eq:matrix22skew}
{\hat E}_{\xi} = \kappa \left( \begin{array}{ccc} \xi_1 \xi_2 & -\xi^2_1 \\ \xi_2^2
  & - \xi_1 \xi_2 \end{array} \right)~~,
\een
which implies $E^2_{\xi} = 0$ and $E_{\xi}^t = -
E_{\xi}$, where $^t$ denotes transposition with respect to $\cB$. Let
$E\in \End(\Sigma)$ be an endomorphism satisfying $E^t = - E$. This
condition implies:
\be
{\hat E} = \left( \begin{array}{ccc} k & -b \\ c & - k \end{array}
\right)~~,~~{\hat E}{\hat A}= \left( \begin{array}{ccc} b & k \\ k &
  c \end{array} \right)~~(\mathrm{with}~k,b,c\in \R)~~.
\ee
Notice that $\tr(E)=0$. Direct computation shows
that the conditions $E^2 = 0$ and $E\circ A \circ E=\tr(E\circ A) E$
are equivalent to each other in this two-dimensional example and amount
to the relation $k^2 = b c$, which in particular shows that $E$
vanishes iff $b=-c$. Assume that $E\neq 0$ and set:
\be
\kappa\eqdef \tr(E\circ A)=\sign(b+c)~~,~~\xi_1\eqdef
\sqrt{|b|}~~,~~\xi_2\eqdef \sign(k)\kappa\sqrt{|c|}~~,
\ee
where $\sign(k)\eqdef 0$ if $k=0$. Then it is easy to see that
$b=\kappa \xi^2_1$, $c = \kappa \xi^2_2$ and $k = \kappa \xi_1 \xi_2$,
which shows that $E=E_\xi$. Hence endomorphism which can be written
in this form are characterized by the condition $E^t=-E$ together with
either of the conditions $E^2=0$ or $E\circ A \circ E=\tr(E\circ A)
E$, which are equivalent to each other in this low-dimensional example.
\end{example}

\subsection{Including linear constraints}

\noindent The following result will be used in Sections
\ref{sec:SpinorsAsPolyforms} and \ref{sec:GCKS}.

\begin{prop}
\label{prop:constraintendo}
Let $Q\in \End(\Sigma)$ and $\kappa\in \{-1,1\}$ be a fixed sign
factor. A real spinor $\xi\in \Sigma$ satisfies $Q(\xi) = 0$ if and
only if $Q\circ\cE_{\kappa}(\xi) = 0$ or, equivalently,
$\cE_{\kappa}(\xi)\circ Q^{t} = 0$, where $Q^{t}$ is the adjoint of
$Q$ with respect to $\cB$.
\end{prop}

\begin{proof}
Take $\xi\in \Sigma$ and assume $Q(\xi) = 0$. Then:
\be
(Q\circ\cE_{\kappa}(\xi))(\chi) = \kappa\, Q(\xi)\,\xi^{\ast}(\chi) = 0  \,\,\,\,\,\, \forall \chi\in \Sigma
\ee
and hence $Q\circ\cE_{\kappa}(\xi) = 0$. Conversely, assume that
$Q\circ\cE_{\kappa}(\xi) = 0$ and pick $\chi\in\Sigma$ such that
$\xi^{\ast}(\chi) \neq 0$ (which is possible since $\cB$ is
non-degenerate). Then the same calculation as before gives:
\be
Q(\xi)\,\xi^{\ast}(\chi) = 0~~, 
\ee
implying $Q(\xi)=0$. The statement for $Q^{t}$ follows from the fact
that $\cB$-transposition is an anti-automorphism of the $\R$-algebra
$(\End(\Sigma),\circ)$, upon noticing that the relation
$\cE_\kappa(\xi)^t=\sigma \cE_\kappa(\xi)$ implies $(Q\circ
\cE_{\kappa}(\xi))^{t}=\sigma \cE_{\kappa}(\xi)\circ Q^{t}$.
\end{proof}

\begin{example}
Let $(\Sigma,\cB)$ be a two-dimensional Euclidean vector space with
orthonormal basis $\Delta$ as in Example \ref{ep:2dEuclidean}. Let
$Q\in\End(\Sigma)$ have matrix:
\be
{\hat Q} = \left( \begin{array}{ccc} q & 0 \\ 0 & 0 \end{array}\right)~~(\mathrm{with}~q\in \R^\times)
\ee
in this basis. Given $\xi\in \Sigma$,  Example \ref{ep:2dEuclidean} gives:
\be
{\hat E}_{\xi} = \kappa \left( \begin{array}{ccc} \xi_1^2 & \xi_1 \xi_2
  \\ \xi_1 \xi_2 & \xi_2^2 \end{array} \right)~~,~~
{\hat Q}{\hat E}_{\xi} = \kappa \left( \begin{array}{ccc} \xi_1^2 q & q\xi_1\xi_2 \\ 0
  & 0 \end{array} \right)~~.
\ee	
Thus $Q\circ E_{\xi}$ vanishes iff $\xi_1 = 0$, i.e. iff $Q(\xi) = 0$.
\end{example}


\section{From real spinors to polyforms}
\label{sec:SpinorsAsPolyforms}


\subsection{Admissible pairings for irreducible real Clifford modules}

Let $V$ be an oriented $d$-dimensional $\R$-vector space equipped with
a non-degenerate metric $h$ of signature $p-q\equiv_8 0, 2$ (hence the
dimension $d=p+q$ of $V$ is even) and let $(V^{\ast},h^{\ast})$ be the
quadratic space dual to $(V,h)$, where $h^{\ast}$ denotes the metric
dual to $h$. Let $\Cl(V^{\ast},h^{\ast})$ be the real Clifford algebra
of this dual quadratic space, viewed as a $\Z_2$-graded associative
algebra with decomposition:
\be
\Cl(V^{\ast},h^{\ast}) = \Cl^{\ev}(V^{\ast},h^{\ast}) \oplus \Cl^{\odd}(V^{\ast},h^{\ast})~~.
\ee
In our conventions, the Clifford algebra satisfies (notice the sign !):
\ben
\label{Crel}
\theta^2 =+h^\ast(\theta,\theta) \,\,\,\,\,\, \forall \theta\in V^\ast~~.
\een
Let $\pi$ denote the standard automorphism of
$\Cl(V^{\ast},h^{\ast})$, which acts as minus the identity on
$V^{\ast}\subset \Cl(V^{\ast},h^{\ast})$ and $\tau$ denote the
standard anti-automorphism, which acts as the identity on
$V^{\ast}\subset \Cl(V^{\ast},h^{\ast})$. These two commute and their
composition is an anti-automorphism denoted by $\hat{\tau} =
\pi\circ\tau=\tau\circ \pi$. Let $\Cl^\times(V^{\ast},h^{\ast})$
denote the group of units $\Cl(V^{\ast},h^{\ast})$. Its \emph{twisted
  adjoint action} is the morphism of groups $\widehat{\Ad}\colon
\Cl^{\times}(V^{\ast},h^{\ast}) \to \Aut(\Cl(V^{\ast},h^{\ast}))$
defined through:
\be
\widehat{\Ad}_x(y) = \pi(x)\, y\, x^{-1} \,\,\,\,\,\, \forall\, x , y 
\in \Cl^{\times}(V^{\ast},h^{\ast})~~.
\ee
We denote by $\bGamma(V^{\ast},h^{\ast})\subset \Cl(V^{\ast},h^\ast)$
the Clifford group, which is defined as follows:
\be
\bGamma(V^{\ast},h^{\ast}) \eqdef \left\{ x \in 
\Cl^{\times}(V^{\ast},h^{\ast})\,\, \vert\,\, \widehat{\Ad}_x(V^{\ast}) = V^{\ast} \right\}~~,
\ee
This fits into the short exact sequence:
\ben
\label{eq:GammaSeq}
1 \to \R^\times \hookrightarrow  \bGamma(V^{\ast},h^{\ast}) \xrightarrow{\widehat{\Ad}} \O(V^{\ast},h^{\ast})\to 1~~,
\een
where $\O(V^{\ast},h^{\ast})$ is the orthogonal group of the quadratic
space $(V^{\ast},h^{\ast})$. Recall that the pin and spin groups of
$(V^{\ast},h^{\ast})$ are the subgroups of
$\bGamma(V^{\ast},h^{\ast})$ defined through:
\be
\Pin(V^{\ast},h^{\ast}) \eqdef \left\{x\in \bGamma(V^{\ast},h^{\ast})\,\, \vert\,\, N(x)^2 = 1 \right\}~~,~~
\Spin(V^{\ast},h^{\ast}) \eqdef \Pin(V^{\ast},h^{\ast})\cap \Cl^{\ev}(V^{\ast},h^{\ast})~~,
\ee
where $N\colon \bGamma(V^{\ast},h^{\ast}) \to \R^\times$ is the
Clifford norm morphism, which is given by:
\be
N(x) \eqdef \hat{\tau}(x)\, x \,\,\,\,\,\, \forall x\in \bGamma(V^{\ast},h^{\ast})~~.
\ee
We have $N(x)^2 = N(\pi(x))^2$ for all
$x\in\bGamma(V^{\ast},h^{\ast})$. For $p q \neq 0$, the groups
$\SO(V^\ast,h^\ast)$, $\Spin(V^{\ast},h^{\ast})$ and
$\Pin(V^{\ast},h^{\ast})$ are disconnected; the first
have two connected components while the last has four. The connected
components of the identity in $\Spin(V^{\ast},h^{\ast})$ and
$\Pin(V^\ast,h^\ast)$ coincide, being given by:
\be
\Spin_0(V^{\ast},h^{\ast}) = \left\{x\in \bGamma(V^{\ast},h^{\ast})\,\, 
\vert\,\, N(x) = 1 \right\}
\ee
and we have $\Spin(V^\ast,h^\ast)/\Spin_0(V^\ast,h^\ast)\simeq \Z_2$
and $\Pin(V^{\ast},h^{\ast})/\Spin_0(V^\ast,h^\ast)\simeq \Z_2\times \Z_2$.

Let $\Sigma$ be a finite-dimensional $\R$-vector space and
$\gamma\colon \Cl(V^{\ast},h^{\ast})\to \End(\Sigma)$ a Clifford
representation. Then $\Spin(V^{\ast},h^{\ast})$ acts on $\Sigma$
through the restriction of $\gamma$ and \eqref{eq:GammaSeq}
induces the short exact sequence:
\ben
\label{eq:SpinSeq}
1\to \Z_2 \to \Spin(V^{\ast},h^{\ast}) \xrightarrow{\widehat{\Ad}} \SO(V^{\ast},h^{\ast}) \to 1~~,
\een
which in turn gives the exact sequence:
\be
1\to \Z_2 \to \Spin_0(V^{\ast},h^{\ast}) 
\xrightarrow{\widehat{\Ad}} \SO_0(V^{\ast},h^{\ast}) \to 1~~.
\ee
Here $\SO_0(V^{\ast},h^{\ast})$ denotes the connected component of the
identity of the special orthogonal group $\SO(V^{\ast},h^{\ast})$. In
signatures $p-q\equiv_8 0, 2$ (the ``real/normal simple case'' of
\cite{LazaroiuBC}), the algebra $\Cl(V^{\ast},h^{\ast})$ is simple and
isomorphic (as a unital associative $\R$-algebra) to the algebra of
square real matrices of size $N=2^{\frac{d}{2}}$. In such signatures
$\Cl(V^{\ast},h^{\ast})$ admits a unique {\em irreducible} real left
module $\Sigma$, which has dimension $N$. This irreducible
representation is faithful and surjective, hence in such signatures
the representation map $\gamma\colon
\Cl(V^{\ast},h^{\ast})\xrightarrow{\simeq} \End(\Sigma)$ is an {\em
  isomorphism} of unital $\R$-algebras.

We will equip $\Sigma$ with a non-degenerate bilinear pairing which is
\emph{compatible} with Clifford multiplication. Ideally, such
compatibility should translate into invariance under the natural
action of the pin group. However, this condition cannot be satisfied
when if $p q\neq 0$. Instead, we consider the weaker notion of
\emph{admissible bilinear pairing} introduced in \cite{AC,ACDP} (see
\cite{LazaroiuB,LazaroiuBC,LazaroiuBII} for applications to
supergravity), which encodes the \emph{best} compatibility condition
with Clifford multiplication that can be imposed on a bilinear pairing
on $\Sigma$ in arbitrary dimension and signature. The following result
of \cite{HarveyBook} summarizes the main properties of admissible
bilinear pairings.

\begin{thm}{\rm \cite[Theorem 13.17]{HarveyBook}}
\label{thm:admissiblepairings}
Suppose that $h$ has signature $p-q\equiv_8 0,2$. Then the irreducible
real Clifford module $\Sigma$ admits two non-degenerate bilinear
pairings $\cB_{+}\colon \Sigma\times\Sigma\to \R$ and $\cB_{-}\colon
\Sigma\times\Sigma\to \R$ (each determined up to multiplication by a
non-zero real number) such that:
\ben
\label{eq:admissiblepairins}
\cB_{+}(\gamma(x)(\xi_1),\xi_2) = \cB_{+}(\xi_1, \gamma(\tau(x))(\xi_2))~~, 
~~ \cB_{-}(\gamma(x)(\xi_1),\xi_2) = \cB_{-}(\xi_1, \gamma(\hat{\tau}(x))(\xi_2))~~, 
\een
for all $x\in \Cl(V^{\ast},h^{\ast})$ and $\xi_1, \xi_2 \in
\Sigma$. The symmetry properties of $\cB_{+}$ and $\cB_{-}$ are as
follows in terms of the modulo $4$ reduction of $k \eqdef \frac{d}{2}$:
\begin{center}
\begin{tabular}{ | l | p{3cm} | p{3cm} | p{3cm} | p{3cm} |}
\hline
$k\,mod$ 4 & 0 & 1 & 2 & 3 \\ \hline
$\cB_{+}$ & Symmetric & Symmetric & Skew-symmetric & Skew-symmetric   \\ \hline
$\cB_{-}$ & Symmetric & Skew-symmetric & Skew-symmetric & Symmetric  \\ \hline
\hline
\end{tabular}
\end{center}
In addition, if $\cB_s$ (with $s\in \{-1,1\}$) is symmetric, then it
is of split signature unless $pq=0$, in which case $\cB_s$ is
definite. 
\end{thm}

\begin{definition}
The sign factor $s$ appearing in the previous theorem is called the
{\em adjoint type} of $\cB_s$, hence $\cB_{+}$ is of positive adjoint
type ($s=+1$) and $\cB_{-}$ is of negative adjoint type
($s=-1$). 
\end{definition}

\noindent Relations \eqref{eq:admissiblepairins} can be written as:
\ben
\label{gammat}
\gamma(x)^t=\gamma((\pi^{\frac{1-s}{2}}\circ \tau)(x)) \,\,\,\,\,\, \forall x\in \Cl(V^{\ast},h^{\ast})~~,
\een
where $~^t$ denotes the $\cB_s$-adjoint. The symmetry
type of an admissible bilinear form $\cB$ will be denoted by
$\sigma\in \{-1,1\}$. If $\sigma = +1$ then $\cB$ is symmetric whereas
if $\sigma = -1$ then $\cB$ is skew-symmetric. Notice that $\sigma$
depends both on $s$ and on the mod $4$ reduction of $\frac{d}{2}$.

\begin{definition}
A (real) {\em paired simple Clifford module} for $(V^\ast,h^\ast)$ is
a triplet $\bSigma=(\Sigma,\gamma,\cB)$, where $(\Sigma,\gamma)$ is a
simple $\Cl(V^\ast, h^\ast)$-module and $\cB$ is an admissible pairing
on $(\Sigma,\gamma)$. We say that $\bSigma$ has adjoint type $s\in
\{-1,1\}$ and symmetry type $\sigma\{-1,1\}$ if $\cB$ has these
adjoint and symmetry types.
\end{definition}

\begin{remark}
\label{rem:cBrelation} 
Admissible bilinear pairings of positive and negative adjoint types
are related through the pseudo-Riemannian volume form $\nu$ of
$(V^{\ast},h^{\ast})$:
\ben
\label{eq:cBpm}
\cB_{+} = C\,\cB_{-}\circ (\gamma(\nu)\otimes \Id)~~,
\een
for an appropriate non-zero real constant $C$. For specific
applications, we will choose to work with $\cB_{+}$ or with $\cB_{-}$
depending on which admissible pairing yields the computationally
simplest polyform associated to a given spinor $\xi\in \Sigma$. When
$pq = 0$, we will take $\cB_s$ to be positive-definite (which we can
always achieve by rescaling it with a non-zero constant of appropriate
sign). See \cite{LazaroiuBC} for a useful discussion of properties of
admissible pairings in various dimensions and signatures.
\end{remark}

\begin{remark}
Directly from their definition, the pairings $\cB_{s}$ satisfy:
\be
\cB_{s}(\gamma(\pi^{\frac{1+s}{2}}(x))(\xi_1),\gamma(x)(\xi_2)) = N(x)
\cB_{s}( \xi_1, \xi_2) \,\,\,\,\,\, \forall\, x\in\Cl(V^{\ast},h^{\ast})~~
\forall\, \xi_1, \xi_2 \in \Sigma~~.
\ee 
This relation yields:
\be
\cB_{s}(\gamma(x)(\xi_1),\xi_2) + \cB_{s}(\xi_1,\gamma(x)(\xi_2)) =
0 \,\,\,\,\,\, \forall\,\,\xi_1 , \xi_2 \in \Sigma
\ee
for all $x = \theta_1\cdot\theta_2$ with $h^\ast$-orthogonal
$\theta_1, \theta_2 \in V^{\ast}$. This implies that $\cB_{s}$ is
invariant under the action of $\Spin(V^{\ast},h^{\ast})$ and hence
also under $\Spin_0(V^{\ast},h^{\ast})$. If $h$ is positive-definite,
then $\cB_{+}$ is $\Pin(V^{\ast},h^{\ast})$-invariant, since it
satisfies:
\be
\cB_{+}(\gamma(\theta)(\xi_1),\gamma(\theta)(\xi_2)) =
\cB_{+}(\xi_1,\xi_2) \,\,\,\,\,\, \forall \xi_1, \xi_2 \in \Sigma
\ee
for all $\theta\in V^{\ast}$ of unit norm. If $h$ is
negative-definite, then $\cB_{-}$ is
$\Pin(V^{\ast},h^{\ast})$-invariant.
\end{remark}

\noindent
For completeness, let us give an explicit construction of $\cB_{+}$
and $\cB_{-}$. Pick an $h^\ast$-orthonormal basis $\left\{
e^{i}\right\}_{i=1,\ldots, d}$ of $V^\ast$ and
let:
\be
\mathrm{K}(\left\{ e^{i}\right\}) \eqdef \{1\}\cup \left\{ \pm
e^{i_1}\cdot \ldots \cdot e^{i_k}\, | \,\, 1\leq i_1<\ldots <i_k\leq d
\, , \, 1\leq k\leq d\right\}
\ee 
be the finite multiplicative subgroup of $\Cl(V^{\ast},h^{\ast})$
generated by the elements $\pm e^i$. Averaging over
$\mathrm{K}(\left\{ e^{i}\right\})$, we construct an auxiliary
positive-definite inner product $(-,-)$ on $\Sigma$ which is invariant
under the action of this group. This product satisfies:
\be
(\gamma(x)(\xi_1), \gamma(x)(\xi_2)) = (\xi_1, \xi_2) \,\,\,\,\,\,
\forall x\in \mathrm{K}(\left\{e^{i}\right\})~~\forall \, \xi_1,
\xi_2 \in \Sigma~~.
\ee
Write $V^{\ast} = V^{\ast}_{+} \oplus V^{\ast}_{-}$, where
$V^{\ast}_{+}$ is a $p$-dimensional subspace of $V^{\ast}$ on which
$h^{\ast}$ is positive definite and $V^{\ast}_{-}$ is a
$q$-dimensional subspace of $V^{\ast}$ on which $h^{\ast}$ is
negative-definite. Fix an orientation on $V^{\ast}_{+}$ (which induces
a unique orientation on $V^{\ast}_{-}$ compatible with the orientation
of $V^{\ast}$ induced from that of $V$) and denote by $\nu_{+}$ and
$\nu_{-}$ the corresponding pseudo-Riemannian volume forms. We have
$\nu = \nu_{+} \wedge \nu_{-}$.  For $p$ (and hence $q$) odd, define:
\ben
\label{eq:Bpmodd}
\cB_{\pm}(\xi_1 , \xi_2) = (\gamma(\nu_{\pm})(\xi_1),\xi_2) \,\,\,\,\,\, \forall \xi_1, \xi_2 \in \Sigma~~,
\een
whereas for $p$ (and hence $q$) even, set:
\ben
\label{eq:Bpmeven}
\cB_{\pm}(\xi_1 , \xi_2) = (\gamma(\nu_{\mp})(\xi_1),\xi_2) \,\,\,\,\,\, \forall \xi_1, \xi_2 \in \Sigma~~.
\een
Then $\cB_{\pm}$ are admissible pairings in the sense of Theorem
\ref{thm:admissiblepairings}. Direct computation using equations
\eqref{eq:Bpmodd} and \eqref{eq:Bpmeven} gives the following result,
which fixes the constant $C$ appearing in Remark \ref{rem:cBrelation}.

\begin{prop}
The admissible pairings $\cB_{+}$ and $\cB_{-}$ constructed above are
related as follows:
\ben
\label{eq:cB_pm}
\cB_{+} = (-1)^{[\frac{q}{2}]} \cB_{-}(\gamma(\nu)\otimes \Id)~~.
\een
Thus we can normalize $\cB_\pm$ such that the constant in
\eqref{eq:cBpm} is given by $C= (-1)^{[\frac{q}{2}]}$.
\end{prop}

\begin{prop}
\label{prop:cBinvar}
Let $\cB$ be an admissible pairing on the real simple
$\Cl(V^\ast,h^\ast)$-module $(\Sigma,\gamma)$.  Then $\cB$ is
invariant under the action of the group $\Spin_0(V^\ast,h^\ast)$ on
$\Sigma$ obtained by restricting $\gamma$.
\end{prop}

\begin{proof}
We have to show the relation:
\ben
\label{gammagammat}
\gamma(x)^t\circ \gamma(x)=\Id \,\,\,\,\, \forall x\in
\Spin_0(V^\ast,h^\ast)~~.
\een
Consider orthonormal basis $\{e^i\}_{i=1,\ldots, d}$ of $(V^\ast,
h^\ast)$ such that $h^\ast(e^i,e^i)=+1$ for $i=1,\ldots,p$ and
$h^\ast(e^i,e^i)=-1$ for $i=p+1,\ldots, d$. A simple computation using
relation \eqref{gammat} shows that \eqref{gammagammat} holds for $x$ of the
form $e^{i_1}\cdot \ldots \cdot e^{i_{2k}}\cdot e^{j_1}\cdot \ldots \cdot
e^{j_{2l}}$, where $1\leq i_1\leq \ldots\leq i_{2k}\leq p$ and
$p+1\leq j_1\leq \ldots\leq j_{2l}\leq d$ with $0\leq 2k\leq p$ and
$0\leq 2l\leq q$. Since such elements generate
$\Spin_0(V^\ast,h^\ast)$, we conclude.
\end{proof}

\subsection{The K\"ahler-Atiyah model of $\Cl(V^\ast,h^\ast)$}

To identify spinors with polyforms, we will use an explicit
realization of $\Cl(V^{\ast},h^{\ast})$ as a deformation of the
exterior $\wedge V^{\ast}$. This model (which goes back to the work of
Chevalley and Riesz \cite{Chevalley1,Chevalley2,Riesz}) has an
interpretation as a deformation quantization of the odd symplectic
vector space obtained by parity change from the quadratic space
$(V,h)$ (see \cite{Berezin,Voronov}). It can be constructed using the
\emph{symbol map} and its inverse, the \emph{quantization
  map}. Consider first the linear map $\f\colon V^{\ast}
\to \End(\wedge V^{\ast})$ given by:
\be
\f(\theta)(\alpha) = \theta\wedge \alpha + \iota_{\theta^{\sharp}} \alpha \,\,\,\,\,\, \forall \theta\in V^{\ast}~~\forall \alpha \in \wedge V^{\ast}~~.
\ee
We have:
\be
\f(\theta)\circ \f(\theta) = h^{\ast}(\theta,\theta) \,\,\,\,\,\,  \forall \theta \in V^{\ast}~~.
\ee
By the universal property of Clifford algebras, it follows that
$\f$ extends uniquely to a morphism $\f\colon
\Cl(V^{\ast},h^{\ast}) \to \End(\wedge V^{\ast})$ of unital
associative algebras such that $\f\circ i = \f$,
where $i\colon V^{\ast}\hookrightarrow \Cl(V^{\ast},h^{\ast})$ is the
canonical inclusion of $V$ in $\Cl(V^{\ast},h^{\ast})$.

\begin{definition}
The {\em symbol (or Chevalley-Riesz) map} is the linear map
$\mathfrak{l} \colon \Cl(V^{\ast},h^{\ast}) \to \wedge V^{\ast}$
defined through:
\be
\mathfrak{l}(x) = \f(x)(1)\,\,\,\,\,\, \forall x\in \Cl(V^\ast,h^\ast)~~,
\ee
where $1\in \R$ is viewed as an element of $\wedge^0(V^{\ast}) = \R$.
\end{definition}

\noindent
The symbol map is an isomorphism of filtered vector spaces. We have:
\be
\mathfrak{l}(1) = 1~~,~~ \mathfrak{l}(\theta) = \theta~~,~~
\mathfrak{l}(\theta_1\cdot \theta_2) = \theta_1\wedge \theta_2 +
h^{\ast}(\theta_1,\theta_2) \,\,\,\,\,\, \forall \theta, \theta_1, \theta_2
\in V^{\ast}~~.
\ee
As expected, $l$ is not a morphism of algebras. The inverse:
\be
\Psi\eqdef \mathfrak{l}^{-1}\colon \wedge V^{\ast} \to \Cl(V^{\ast},h^{\ast})~~.
\ee
of $\mathfrak{l}$ (called the {\em quantization map})
allows one to view  $\Cl(V^{\ast},h^{\ast})$ as a
deformation of the exterior algebra $(\wedge V^\ast,\wedge)$
(see \cite{Berezin,Voronov}). Using $\mathfrak{l}$ and $\Psi$, we 
transport the algebra product of $\Cl(V^{\ast},h^{\ast})$ to an
$h$-dependent unital associative product defined on $\wedge
V^{\ast}$, which deforms the wedge product.

\begin{definition}
The {\em geometric product} $\diamond:\wedge V^{\ast}\times \wedge
V^{\ast}\rightarrow \wedge V^{\ast}$ is defined through:
\be
\alpha_1 \diamond \alpha_2 \eqdef \mathfrak{l}(\Psi(\alpha_1)\cdot \Psi(\alpha_2)) \,\,\,\,\, \forall \alpha_1,\alpha_2\in \wedge V^\ast~~,
\ee
where $\cdot$ denotes multiplication in $\Cl(V^\ast,h^\ast)$.
\end{definition}

\noindent
By definition, the map $\Psi$ is an isomorphism of unital associative
$\R$-algebras\footnote{Notice that the geometric product is not
  compatible with the grading of $\wedge V^\ast$ given by form rank,
  but only with its mod 2 reduction, because the quantization map does
  not preserve $\Z$-gradings. Hence the K\"ahler-Atiyah algebra is
  {\em not} isomorphic with $\Cl(V^{\ast},h^{\ast})$ in the category
  of Clifford algebras defined in \cite{Lazaroiu:2016vov}. As such,
  it provides a different viewpoint on spin geometry, which is
  particularly useful for our purpose (see
  \cite{LazaroiuBC,LazaroiuB,LazaroiuBII,LBCProc}).} from $(\wedge
V^{\ast},\diamond)$ to $\Cl(V^{\ast},h^{\ast})$. Through this
isomorphism, the inclusion $V^{\ast}\hookrightarrow
\Cl(V^{\ast},h^{\ast})$ corresponds to the natural inclusion
$V^{\ast}\hookrightarrow \wedge V^{\ast}$. We shall refer to $(\wedge
V^{\ast}, \diamond)$ as the \emph{K\"ahler-Atiyah algebra} of the
quadratic space $(V,h)$ (see \cite{Kahler,Graf}).
It is easy to see that the geometric product satisfies:
\be
\theta\diamond \alpha = \theta \wedge \alpha +
\iota_{\theta^{\sharp}} \alpha \,\,\,\,\,\, \forall \theta\in
V^{\ast}~~\forall \alpha\in \wedge V^{\ast}~~.
\ee
Also notice the relation:
\be
\theta\diamond\theta = h^{\ast}(\theta,\theta) \,\,\,\,\,\, \forall \theta\in V^\ast~~.
\ee
The maps $\pi$ and $\tau$ transfer through $\Psi$ to the
K\"ahler-Atiyah algebra, producing unital (anti)-automorphisms of the
latter which we denote by the same symbols. With this notation, we have:
\ben
\label{pitauPsi}
\pi\circ \Psi=\Psi\circ \pi~~,~~\tau\circ \Psi=\Psi\circ \tau~~.
\een
For any orthonormal basis $\left\{ e^i \right\}_{i=1,\ldots,d}$ of
$V^{\ast}$ and any $k\in \{1,\ldots,d\}$, we have $e^{1}\diamond
\dots\diamond e^{k}=e^{1}\wedge\dots\wedge e^{k}$ and:
\be
\pi(e^{1}\wedge\dots\wedge e^{k}) = (-1)^{k} e^{1}\wedge\dots\wedge e^{k}~~ ,
~~ \tau(e^{1}\wedge\dots\wedge e^{k}) = e^{k}\wedge\dots\wedge e^{1}~~. 
\ee

\noindent
Let $\cT(V^{\ast})$ denote the tensor algebra of the (parity change
of) $V^{\ast}$, viewed as a $\Z$-graded associative
superalgebra whose $\Z_2$-grading is the reduction of the natural
$\Z$-grading; thus elements of $V$ have integer degree one and they
are odd. Let:
\be
\Der(\cT(V^{\ast})) \eqdef \bigoplus_{k\in \Z} \Der^k(\cT(V^{\ast}))
\ee
denote the $\Z$-graded Lie superalgebra of all superderivations. The
minus one integer degree component $\Der^{-1}(\cT(V^{\ast}))$ is linearly
isomorphic with the space $\Hom(V^{\ast},\R) = V$ acting by
contractions:
\be
\iota_v (\theta_1\otimes \dots \otimes \theta_k) = \sum_{i=1}^{k}
(-1)^{i-1} \theta_1\otimes \dots \otimes \iota_v\theta_i\otimes \dots
\otimes \theta_k \,\,\,\,\,\, \forall v\in V~~\forall \theta_1,\hdots
, \theta_k\in V^{\ast}~~,
\ee
while the zero integer degree component $\Der^0(\cT(V^{\ast}))
= \End(V^{\ast}) = \mathfrak{gl}(V^{\ast})$ acts through:
\be
L_A(\theta_1\otimes \dots \otimes \theta_k) = \sum_{i=1}^{k}
\theta_1\otimes \dots \otimes A(\theta_i)\otimes \dots \otimes
\theta_k \,\,\,\,\,\, \forall A\in \mathfrak{gl}(V^{\ast})~~.
\ee
We have an isomorphism of super-Lie algebras:
\be
\Der^{-1}(\cT(V^{\ast})) \oplus \Der^0(\cT(V^{\ast})) \simeq 
V\rtimes \mathfrak{gl}(V^{\ast})~~.
\ee
The action of this super Lie algebra preserves the ideal used to
define the exterior algebra as a quotient of $\cT(V^{\ast})$ and hence
descends to a morphism of super Lie algebras $\cL_{\Lambda}\colon
V\rtimes \mathfrak{gl}(V^{\ast}) \to \Der(\wedge V^{\ast},
\wedge)$. Contractions also preserve the ideal used to define the
Clifford algebra as a quotient of $\cT(V^{\ast})$. On the other hand,
endomorphisms $A$ of $V^{\ast}$ preserve that ideal iff
$A\in\mathfrak{so}(V^{\ast},h^{\ast})$. Together with
contractions, they induce a morphism of super Lie algebras $\cL_{\Cl}
\colon V\rtimes \mathfrak{so}(V^{\ast},h^{\ast}) \to
\Der(\Cl(V^{\ast},h^{\ast}))$. The following result states that
$\cL_{\Lambda}$ and $\cL_{\Cl}$ are compatible with
$\mathfrak{l}$ and $\Psi$.

\begin{prop}{\rm \cite[Proposition 2.11]{Meinrenken}}
\label{prop:equivariancePsi}
The quantization and symbol maps intertwine the actions of $V\rtimes
\mathfrak{so}(V^{\ast},h^{\ast})$ on $\Cl(V^{\ast},h^{\ast})$ and
$\wedge V^{\ast}$:
\be
\Psi(\cL_{\Lambda}(\varphi)(\alpha)) = \cL_{\Cl}(\varphi)(\Psi(\alpha)) ~~ , 
~~ \mathfrak{l}(\cL_{\Cl}(\varphi)(x)) = \cL_{\Lambda}(\varphi)(\mathfrak{l}(x))~~,
\ee
for all $\varphi \in V\rtimes \mathfrak{so}(V^{\ast},h^{\ast})$, $\alpha 
\in \wedge V^{\ast}$ and $x\in \Cl(V^{\ast},h^{\ast})$.
\end{prop}

\noindent
This proposition shows that quantization is equivariant with respect
to affine orthogonal transformations of $(V^\ast, h^\ast)$. In
signatures $p-q\equiv_8 0, 2$, composing $\Psi$ with the irreducible
representation $\gamma\colon \Cl(V^{\ast},h^{\ast}) \to \End(\Sigma)$
(which in such signatures is a unital isomorphism of algebras) gives
an isomorphism of unital associative $\R$-algebras\footnote{This
  isomorphism identifies the deformation quantization $(\wedge
  V^{\ast},\diamond)$ of the exterior algebra $(\wedge
  V^{\ast},\wedge)$ with the operator quantization
  $(\End(\Sigma),\circ)$ of the latter.}:
\ben
\label{Psigamma}
\Psi_{\gamma}\eqdef \gamma\circ\Psi\colon (\wedge V^{\ast}, \diamond)\ \stackrel{\sim}{\to} (\End(\Sigma),\circ)~~.
\een
Since $\Psi_{\gamma}$ is an isomorphism of algebras and $(\wedge
V^{\ast},\diamond)$ is generated by $V^{\ast}$, the
identity together with the elements $\Psi_\gamma(e^{i_1}\wedge \ldots \wedge
e^{i_k})=\gamma(e^{i_1})\circ \hdots \circ \gamma(e^{i_k})$ for $1\leq
i_1<\hdots <i_k\leq d$ and $k=1,\ldots, d$ form a basis of
$\End(\Sigma)$.

\begin{remark}
We sometimes denote the action of a polyform $\alpha \in \wedge
V^{\ast}$ as an endomorphism on $\Sigma$ by a dot (this
corresponds to Clifford multiplication through the isomorphism
$\Psi_\gamma$):
\be
\alpha\cdot\xi \eqdef \Psi_{\gamma}(\alpha)(\xi) \,\,\,\,\,\, \forall \alpha\in \wedge V^{\ast}~~\forall \xi\in \Sigma~~.
\ee 
\end{remark}

\noindent The trace on $\End(\Sigma)$ transfers to the K\"ahler-Atiyah
algebra through the map $\Psi_\gamma$ (see \cite{LazaroiuBC}):

\begin{definition}
The {\em K\"ahler-Atiyah trace} is the linear functional:
\be
\cS\colon \wedge V^{\ast}  \to \R~~,~~ \alpha\mapsto \tr(\frD_{\gamma}(\alpha))~~.
\ee
\end{definition}

\noindent
We will see in a moment that $\cS$ does not depend on $\gamma$ or $h$.
Since $\frD_{\gamma}$ is a unital morphism of algebras, we have:
\be
\cS(1) = \dim(\Sigma)=N=2^{\frac{d}{2}}~~\mathrm{and}~~\cS(\alpha_1\diamond \alpha_2) = 
\cS(\alpha_2\diamond \alpha_1) \,\,\,\,\,\, \forall \alpha_1,\alpha_2\in \wedge V^\ast~~,
\ee
where $1\in \R=\wedge^0 V^\ast$ is the unit element of the field $\R$ of real numbers.

\begin{lemma}
\label{lemma:tracev}
For any $0< k \leq d$, we have:
\be
\cS\vert_{\wedge^{k} V^{\ast}} = 0~~.
\ee
\end{lemma}

\begin{proof}
Let $\left\{ e^i\right\}_{i=1,\ldots,d}$ be an orthonormal basis of
$(V^{\ast},h^{\ast})$. For $i\neq j$ we have $e^i\diamond e^j = -
e^j\diamond e^i$ and hence $(e^i)^{-1}\diamond e^j\diamond e^i = -
e^j$. Let $0 < k \leq d$ and $1\leq i_1 < \cdots < i_k \leq d$. If $k$
is even, then:
\be
\cS(e^{i_1}\diamond \dots \diamond e^{i_{k}}) = \cS(e^{i_k}\diamond
e^{i_1}\diamond \dots \diamond e^{i_{k-1}}) = (-1)^{k-1}
\cS(e^{i_1}\diamond \dots \diamond e^{i_{k}})~~,
\ee
and hence $\cS(e^{i_1}\diamond \dots \diamond e^{i_{k}}) =
0$. Here we used cyclicity of the K\"ahler-Atiyah trace and the
fact that $e^{i_{k}}$ anticommutes with $e^{i_1}, \dots,
e^{i_{k-1}}$. If $k$ is odd, let $j\in \left\{1,\dots,d\right\}$ be
such that $j\not \in \left\{i_1,\dots,i_k\right\}$ (such a $j$ exists
since $k<d$). We have:
\be
\cS(e^{i_1}\diamond \dots \diamond e^{i_{k}}) = \cS(( e^j)^{-1}
\diamond e^{i_1}\diamond \dots 
\diamond e^{i_{k}}\diamond e^j) = - 
\cS( e^{i_1}\diamond \dots \diamond e^{i_{k}}) = 0
\ee
and we conclude.
\end{proof}

\noindent Let $\alpha^{(k)} \in \wedge^k V^{\ast}$ denote the degree
$k$ component of $\alpha\in \wedge V^\ast$. Lemma \ref{lemma:tracev}
implies:

\begin{prop}
\label{prop:cS}
The K\"ahler-Atiyah trace is given by:
\be
\cS(\alpha) = \dim (\Sigma)\, \alpha^{(0)} =2^{\frac{d}{2}}
\alpha^{(0)} \,\,\,\,\,\, \forall \alpha\in \wedge V^\ast~~.
\ee
In particular, $\cS$ does not depend on the irreducible representation
$\gamma$ of $\Cl(V^\ast,h^\ast)$ or on $h$.
\end{prop}

\begin{lemma}
\label{lemma:adjointpoly}
Let $\alpha\in \wedge V^{\ast}$ and $\cB$ be an admissible bilinear
pairing of $(\Sigma,\gamma)$ of adjoint type $s \in \{-1,1\}$. Then
the following equation holds:
\ben
\label{eq:traceD}
\Psi_{\gamma}(\alpha)^{t} = \Psi_{\gamma}(\alpha^t)~~,
\een
where $\Psi_{\gamma}(\alpha)^{t}$ is the $\cB$-adjoint of
$\Psi_{\gamma}(\alpha)$ and we defined the {\em $s$-transpose} of
$\alpha$ through:
\be
\alpha^t\eqdef (\pi^{\frac{1-s}{2}}\circ\tau)(\alpha)~~.
\ee
\end{lemma}

\begin{proof}
Follows immediately from \eqref{gammat} and relations \eqref{pitauPsi}.
\end{proof}

\subsection{Spinor squaring maps}

\begin{definition}
\label{def:squarespinor}
Let $\bSigma=(\Sigma,\gamma,\cB)$ be a paired simple Clifford module
for $(V^\ast, h^\ast)$. The {\em signed spinor squaring maps} of
$\bSigma$ are the quadratic maps:
\be
\label{eq:spinorsquaremap}
\cE_\bSigma^\pm \eqdef \Psi_{\gamma}^{-1} \circ \cE_\pm: \Sigma \to \wedge V^{\ast}~~,
\ee
where $\cE_\pm:\Sigma\to \End(\Sigma)$ are the signed squaring maps of
the paired vector space $(\Sigma,\cB)$ which were defined in Section
\ref{sec:vectorasendo}. Given a spinor $\xi\in \Sigma$, the polyforms
$\cE_\bSigma^+(\xi)$ and $\cE_\bSigma^-(\xi)=-\cE_\bSigma^+(\xi)$ are
called the positive and negative {\em squares} of $\xi$ relative to
the admissible pairing $\cB$. A polyform $\alpha\in \wedge V^\ast$ is
called a {\em signed square} of $\xi\in \Sigma$ if
$\alpha=\cE_\bSigma^+(\xi)$ or $\alpha=\cE_\bSigma^-(\xi)$.
\end{definition}

\noindent Consider the following subsets of $\wedge V^\ast$:
\be
Z:=Z(\bSigma)\eqdef \Psi_\gamma^{-1}(\cZ(\Sigma,\cB)) ~~ , ~~
Z_\pm:=Z_\pm(\bSigma)\eqdef \Psi_\gamma^{-1}(\cZ_\pm(\Sigma,\cB))~~.
\ee
Since $\Psi_{\gamma}$ is a linear isomorphism, Section
\ref{sec:vectorasendo} implies that $\cE_\bSigma^\pm$ are
two-to-one except at $0\in \Sigma$ and:
\be
Z_-=-Z_+ ~~ , ~~ Z_+\cap Z_-=\{0\} ~~ , ~~ Z=Z_+\cup Z_-~~.
\ee
Moreover, $\cE_\bSigma^\pm$ induce the same bijective map:
\ben
\label{eq:hcEspinor}
\hcE_\bSigma:\Sigma/\Z_2 \stackrel{\sim}{\rightarrow} Z(\bSigma)/\Z_2~~.
\een
Notice that $Z$ is a cone in $\wedge V^\ast$, which is the union of
the opposite half cones $Z_\pm$.

\begin{definition}
The bijection \eqref{eq:hcEspinor} is called the {\em class spinor
  squaring map} of the paired simple Clifford module
$\bSigma=(\Sigma,\gamma,\cB)$.
\end{definition}

\noindent  We will sometimes denote by $\alpha_{\xi}\eqdef
\cE^+_\bSigma(\xi)\in Z_+(\bSigma)$ the positive polyform square of
$\xi\in \Sigma$.

\begin{remark}
The representation map $\gamma$ is an isomorphism when $p-q\equiv_8 0,
2$. This does not hold in other signatures, for which the construction
of spinor squaring maps is more delicate (see \cite{LazaroiuBC}).
\end{remark}

\noindent
The following result is a direct consequence of Proposition
\ref{prop:equivariancePsi}.

\begin{prop}
The quadratic maps $\cE_\bSigma^\pm\colon \Sigma \to \wedge V^{\ast}$ are
$\Spin_0(V^{\ast},h^{\ast})$-equivariant:
\be
\cE_\bSigma^\pm(u\,\xi) = \Ad_u(\cE_\bSigma^\pm(\xi)) \,\,\,\,\,\, \forall
u\in \Spin_0(V^{\ast},h^{\ast})~~\forall \xi\in\Sigma~~,
\ee
where the right hand side denotes the natural action of
$\Ad_u\in\O(V^{\ast},h^{\ast})$ on $\wedge V^{\ast}$.
\end{prop}

\noindent We are ready to give the algebraic characterization of
spinors in terms of polyforms.

\begin{thm}
\label{thm:reconstruction} 
Let $\bSigma=(\Sigma,\gamma,\cB)$ be a paired simple Clifford module for
$(V^\ast, h^\ast)$ of symmetry type $\sigma$ and adjoint type
$s$. Then the following statements are equivalent for a polyform
$\alpha\in \wedge V^{\ast}$:
\begin{enumerate}[(a)]
\itemsep 0.0em
\item $\alpha$ is a signed square of some spinor $\xi\in \Sigma$,
  i.e. it lies in the set $Z(\bSigma)$.
\item $\alpha$ satisfies the following relations:
\ben
\label{eq:thmdefequationsequiv}
\alpha\diamond\alpha =\cS(\alpha) \, \alpha ~~ , ~~
(\pi^{\frac{1-s}{2}}\circ\tau)(\alpha) = \sigma\,\alpha ~~ , ~~
\alpha\diamond \beta\diamond\alpha =
\cS(\alpha\diamond\beta)\, \alpha
\een
for a fixed polyform $\beta \in \wedge V^{\ast}$ which satisfies
$\cS(\alpha\diamond\beta) \neq 0$.
\item The following relations hold:
\ben
\label{eq:thmdefequations}
(\pi^{\frac{1-s}{2}}\circ\tau)(\alpha) = \sigma\,\alpha ~~ ,~~
\alpha\diamond \beta\diamond\alpha = \cS(\alpha\diamond\beta)\, \alpha
\een
for any polyform $\beta\in \wedge V^\ast$.
\end{enumerate}
In particular, the set $Z(\bSigma)$ depends only on $\sigma$, $s$ and $(V^\ast, h^\ast)$. 
\end{thm}

\noindent In view of this result, we will also denote 
$Z(\bSigma)$ by $Z_{\sigma,s}(V^\ast,h^\ast)$.

\begin{proof}
Since $\Psi\colon \Cl(V^{\ast},h^{\ast}) \to \End(\Sigma)$ is a unital
isomorphism algebras, $\alpha$ satisfies \eqref{eq:thmdefequations} iff:
\ben
\label{eq:Esquarethm}
E^t = \sigma\,E~~,~~E\circ A\circ E = \tr(E\circ A) E \,\,\,\,\,\,
\forall\, A\in \End(\Sigma)~~,
\een
where $E \eqdef \Psi^{-1}_\gamma(\alpha)$, $A\eqdef \Psi^{-1}_\gamma(\beta)$ and we
used Lemma \ref{lemma:adjointpoly} and the definition and properties
of the \KA trace. The conclusion now follows from
Proposition \ref{prop:characterizationtamecone}.
\end{proof}

\noindent
The second equation in \eqref{eq:thmdefequations} implies:

\begin{cor}
\label{cor:signcriteria}
Let $\alpha \in Z_{\sigma,s}(V^\ast,h^\ast)$. If $k\in \{1,\ldots,
d\}$ satisfies:
\be
(-1)^{k \frac{1-s}{2}} (-1)^{\frac{k(k-1)}{2}} = -\sigma~~,
\ee
then $\alpha^{(k)} = 0$.
\end{cor}

\noindent
Polyform $\alpha\in Z_\pm(\bSigma)$ admits an explicit presentation
which first appeared in \cite{LazaroiuB, LazaroiuBII, LazaroiuBC}.

\begin{prop}
Let $\left\{ e^i\right\}_{i=1,\ldots, d}$ be an orthonormal basis of
$(V^{\ast},h^{\ast})$ and $\kappa\in \{-1,1\}$. Then every polyform
$\alpha\in Z_\kappa(\bSigma)$ can be written as:
\ben
\label{eq:bilinears}
\alpha = \frac{\kappa}{2^{\frac{d}{2}}} \sum_{k=0}^{d} 
\,\sum_{i_1 < \dots < i_k} \cB((\gamma(e^{i_k})^{-1} \circ \dots 
\circ \gamma(e^{i_1})^{-1})(\xi),\xi)\, e^{i_1}\wedge \hdots \wedge e^{i_k}~~,
\een
where the spinor $\xi\in \Sigma$ is determined by $\alpha$ up to sign.
\end{prop}

\begin{remark}
\label{rem:Dirac}
We have:
\be
\gamma(e^i)^{-1} = h^{\ast}(e^i,e^i) \gamma(e^i) = h(e_{i},e_{i}) \gamma(e^i)~~,
\ee
where $\left\{ e_i\right\}_{i=1,\ldots,d}$ is the contragradient
orthonormal basis of $(V,h)$. For simplicity, set:
\be
\gamma^i \eqdef \gamma(e^i)~~ \mathrm{and} ~~ \gamma_i \eqdef
h(e_i,e_i)\gamma(e^i)~~,
\ee
so that $(\gamma^i)^{-1}=\gamma_i$. Then the degree one component in
\eqref{eq:bilinears} reads:
\be
\alpha^{(1)}=\frac{\kappa}{2^{\frac{d}{2}}} \cB(\gamma_i(\xi),\xi)e^i
\ee
and its dual vector field
$(\alpha^{(1)})^\sharp=\frac{\kappa}{2^{\frac{d}{2}}}
\cB(\gamma_i(\xi),\xi)e_i$ is called the (signed) {\em Dirac vector}
of $\xi$ relative to $\cB$. For spinors on a manifold (see Section
\ref{sec:GCKS}), this vector globalizes to the Dirac current.
\end{remark}

\begin{proof}
It is easy to see that the set:
\be
P \eqdef \left\{\Id\right\}\cup \left\{ \gamma^1 \circ \cdots \circ \gamma^{i_1} \circ \cdots
\gamma^{i_k}\circ \cdots \circ \gamma^{d} \,|\,1\leq i_1 < \cdots < i_k \leq d, k=1,\ldots, d \right\}~~,
\ee
gives an orthogonal basis of $\End(\Sigma)$ with respect to the
nondegenerate and symmetric bilinear pairing induced by the trace:
\be
\End(\Sigma)\times \End(\Sigma) \to \R~~,~~ (A_1,A_2) 
\mapsto \tr(A_1 A_2)~~.
\ee
In particular, the endomorphism $E \eqdef \Psi_{\gamma}(\alpha)\in Z_\kappa(\bSigma)$ expands as:
\beqa
E = \frac{1}{2^{\frac{d}{2}}}\sum_{k=0}^{d} \sum_{i_1 < \dots < i_k}
\tr((\gamma^{i_1}\circ \cdots \circ \gamma^{i_k})^{-1}\circ E)\,
\gamma^{i_1}\circ \cdots \circ \gamma^{i_k} \\ =
\frac{\kappa}{2^{\frac{d}{2}}}\sum_{k=0}^{d} \sum_{i_1 < \dots < i_k}
\cB((\gamma^{i_1}\circ \cdots \circ \gamma^{i_k})^{-1} (\xi), \xi)\,
\gamma^{i_1}\circ \cdots \circ \gamma^{i_k} ~~,
\eeqa
where $\xi\in \Sigma$ is a spinor such that $E=\cE_\kappa(\xi)$ and we
noticed that $\tr(B\circ \cE_\kappa(\xi))=\kappa\, \tr(B(\xi)\otimes
\xi^\ast)=\kappa\, \xi^\ast(B(\xi))=\kappa\, \cB(B\xi,\xi)$ for all
$B\in \End(\Sigma)$.  The conclusion follows by applying the
isomorphism algebras
$\Psi_{\gamma}^{-1}:(\End(\Sigma),\circ)\rightarrow (\wedge V^\ast,
\diamond)$ to the previous equation.
\end{proof}

\begin{lemma}
\label{lemma:actionnu}
The following identities hold for all $\alpha \in \wedge V^{\ast}$:
\ben
\label{eq:nuaction}
\alpha \diamond\nu  = \ast\, \tau(\alpha) ~~, ~~ \nu \diamond \alpha = 
\ast\, (\pi\circ\tau) (\alpha)~~.
\een
\end{lemma}

\begin{proof}
Since multiplication by $\nu$ is $\R$-linear, it suffices to consider
homogeneous elements $\alpha=e^{i_1}\wedge \cdots \wedge e^{i_k}$ with
$1\leq i_1 < \cdots < i_k \leq d$, where $\left\{ e^{i}
\right\}_{i=1,\ldots, d}$ is an orthonormal basis of
$(V^{\ast},h^{\ast})$. We have:
\begin{eqnarray*}
& e^{i_1}\wedge \cdots \wedge e^{i_k} \diamond \nu = 
e^{i_1}\diamond \cdots \diamond e^{i_k} \diamond e^1\diamond\cdots \diamond e^d \\ 
& = (-1)^{i_1 + \cdots + i_k} (-1)^k e^1\diamond\cdots \diamond (e^{i_1})^2
\diamond e^{i_1+1}\diamond \cdots \diamond (e^{i_k})^2\diamond e^{i_k + 1}
\diamond \cdots \diamond e^d\\
& = h^{\ast}(e^{i_1},e^{i_1})\cdots h^{\ast}(e^{i_k},e^{i_k})\,(-1)^{i_1 + \cdots + i_k}
 (-1)^k e^1\diamond \cdots \diamond e^{i_1-1}\diamond e^{i_1+1}\cdots \diamond e^{i_k-1}
\diamond e^{i_k+1}\diamond \cdots\diamond e^d\\
& = (-1)^{\frac{k(k - 1)}{2}}(-1)^{2(i_1 + \cdots + i_k)} (-1)^{2k} \ast(e^{i_1}\wedge 
\cdots \wedge e^{i_k} ) = \ast \tau(\alpha)~~,
\end{eqnarray*}
which implies $ \alpha \diamond \nu = \ast\, \tau (\alpha)$. Using the
obvious relation $\alpha \diamond \nu = (\nu\diamond\pi)(\alpha)$, we
conclude.
\end{proof}

\noindent The following shows that the choice of admissible pairing
used to construct the spinor square map is a matter of taste,
see also Remark \ref{rem:cBrelation}.

\begin{prop} 
Let $\xi\in \Sigma$ and denote by $\alpha_{\xi}^\pm \in Z_+$ the {\em
  positive} polyform squares of $\xi$ relative to the admissible
pairings $\cB_+$ and $\cB_-$ of $(\Sigma,\gamma)$, which we assume to
be normalized such that they are related through
\eqref{eq:cB_pm}. Then the following relation holds:
\be
\ast \,\alpha_{\xi}^{+} = (-1)^{[\frac{q+1}{2}]+ p(q+1) } (-1)^{d}
c(\alpha_{\xi}^{-})~~.
\ee
where $c\colon \wedge V^{\ast} \to \wedge V^{\ast}$ is the linear map
which acts as multiplication by $\frac{k!}{ (d - k)!} $
in each degree $k$.
\end{prop}

\begin{proof}
We compute:
\begin{eqnarray*}
& \ast (\alpha_{\xi}^{+})^{(k)} = \frac{1}{2^{\frac{d}{2}}} 
\cB_{+}((\gamma_{i_k}\circ \hdots \circ \gamma_{i_1})(\xi), \xi)   \ast (e^{i_1}\wedge 
\hdots \wedge e^{i_k}) \\ & =  (-1)^{[\frac{q+1}{2}]+ pq }  
(-1)^{\frac{k (k-1)}{2}}\frac{\sqrt{\vert h\vert }}{2^{\frac{d}{2}} 
(d - k)!} \cB_{-}((\gamma(\nu)\circ\gamma_{i_1}\circ \hdots \circ \gamma_{i_k})(\xi), \xi) 
\epsilon^{i_{i_1}\hdots i_k a_{k+1} \hdots a_d} e_{a_{k+1}}\wedge \hdots 
\wedge e_{a_d} \\ & =  (-1)^{[\frac{q+1}{2}]+ pq } (-1)^{\frac{k (k-1)}{2}} 
(-1)^{k(d-k)}  \frac{k!}{2^{\frac{d}{2}} (d - k)!} \cB_{-}(\gamma(\nu)
\gamma(\ast (e^{a_{k+1}}\wedge \hdots \wedge e^{d}))(\xi), \xi)  e_{a_{k+1}}
\wedge \hdots \wedge e_{d} \\ & =  (-1)^{[\frac{q+1}{2}]+ pq } 
(-1)^{\frac{k (k-1)}{2}} (-1)^{\frac{(d-k)(d+k+1)}{2}} 
 \frac{k!}{2^{\frac{d}{2}} (d - k)!} \cB_{-}((\gamma(\nu)^2 \circ \gamma^{a_{k+1}}
 \circ \hdots  \circ \gamma^{a_d})(\xi) , \xi)  e_{a_{k+1}}\wedge \hdots \wedge e_{d} 
\\ & = (-1)^{[\frac{q+1}{2}]+ p(q+1) } (-1)^{k} \frac{k!}{ (d - k)!} 
 (\alpha_{\xi}^{-})^{(d-k)} = (-1)^{[\frac{q+1}{2}]+ p(q+1) } (-1)^{d}
 \frac{k!}{ (d - k)!}  \pi(\alpha_{\xi}^{-})^{(d-k)}~~,
\end{eqnarray*}
where we used the identity $\nu\diamond \alpha = \ast (\pi \circ
\tau)(\alpha)$ proved in Lemma \ref{lemma:actionnu}.
\end{proof}

\subsection{Linear constraints}
The following result will be used in Sections \ref{sec:SpinorsAsPolyforms}
and \ref{sec:GCKS}.

\begin{prop}
\label{prop:constraintendopoly} 
A spinor $\xi\in \Sigma$ lies in the kernel of an endomorphism $Q\in \End(\Sigma)$
iff:
\be
\hat{Q} \diamond \alpha_{\xi} = 0~~,
\ee
where $\alpha_{\xi}\eqdef \cE_\bSigma^+(\xi)$ is the positive polyform
square of $\xi$ and:
\be
\hat{Q} \eqdef \Psi_{\gamma}^{-1}(Q) \in \wedge V^{\ast}
\ee
is the {\em dequantization} of $Q$.
\end{prop}

\begin{remark}
Taking the $s$-transpose shows that equation $\hat{Q} \diamond
\alpha_{\xi} = 0$ is equivalent to:
\be
\alpha_{\xi}\diamond (\pi^{\frac{1-s}{2}}\circ\tau)(\hat{Q}) = 0~~.
\ee
\end{remark}

\begin{proof} 
Follows immediately from Proposition \ref{prop:constraintendo}, using
the fact that $\Psi_{\gamma}\colon (\wedge
V^{\ast},\diamond)\to \End(\Sigma)$ is an isomorphism of unital
associative algebras.
\end{proof}


\subsection{Real chiral spinors}


Theorem \ref{thm:reconstruction} can be refined for chiral spinors of
real type, which exist in signature $p-q\equiv_8 0$. In this case, the
Clifford volume form $\nu\in \Cl(V^{\ast},h^{\ast})$ squares to $1$
and lies in the center of $\Cl^{\ev}(V^{\ast},h^{\ast})$, giving
a decomposition as a direct sum of simple associative
algebras:
\be
\Cl^{\ev}(V^{\ast},h^{\ast}) = \Cl^{\ev}_{+}(V^{\ast},h^{\ast}) \oplus
\Cl^{\ev}_{-}(V^{\ast},h^{\ast})~~,
\ee
where we defined:
\be
\Cl^{\ev}_{\pm}(V^{\ast},h^{\ast}) \eqdef\frac{1}{2}(1 \pm \nu )\,
\Cl(V^{\ast},h^{\ast})~~.
\ee
We decompose $\Sigma$ accordingly:
\be
\Sigma = \Sigma^{(+)} \oplus \Sigma^{(-)}~~,~~\mathrm{where}~~ \Sigma^{(\pm)} 
\eqdef \frac{1}{2} (\Id \pm \gamma(\nu ))(\Sigma)~~.
\ee
The subspaces $\Sigma^{(\pm)}\subset\Sigma$ are preserved by the
restriction of $\gamma$ to $\Cl^{\ev}(V^{\ast},h^{\ast})$, which
therefore decomposes as a sum of two irreducible representations:
\be
\gamma^{(+)} \colon \Cl^{\ev}(V,h)\to \End(\Sigma^{(+)})~~ \mathrm{and} ~~ \gamma^{(-)} 
\colon \Cl^{\ev}(V,h)\to \End(\Sigma^{(-)})~~
\ee
distinguished by the value which they take on the volume form $\nu\in
\Cl^{\ev}(V^{\ast},h^{\ast})$:
\be
\gamma^{(+)} (\nu ) = \Id~~,~~ \gamma^{(-)} (\nu ) = -\Id~~.
\ee
A spinor $\xi \in \Sigma$ is called chiral of chirality
$\mu\in\{-1,1\}$ if it belongs to $\Sigma^{(\mu)}$. Setting
$\alpha_{\xi}\eqdef \cE_\bSigma^+(\xi)$, Proposition
\ref{prop:constraintendopoly} shows that this amounts to the
condition:
\be
\nu \diamond \alpha_{\xi} = \mu\, \alpha_{\xi}~~.
\ee

\noindent
For any $\mu\in \{-1,1\}$ and $\kappa\in \{-1,1\}$, define:
\be
Z^{(\mu)}_\kappa:=Z^{(\mu)}_\kappa(\bSigma)\eqdef
\cE_\bSigma^\kappa(\Sigma^{(\mu)})~~,~~Z^{(\mu)}:=Z^{(\mu)}(\bSigma)\eqdef
Z^{(\mu)}_+(\bSigma)\cup Z^{(\mu)}_-(\bSigma)~~.
\ee
We have $Z^{(\mu)}_-(\bSigma)=-Z^{(\mu)}_+(\bSigma)$ and
$Z^{(\mu)}_+(\bSigma)\cap Z^{(\mu)}_-(\bSigma)=\{0\}$. Moreover,
$\cE_\bSigma^\kappa$ restrict to surjections
$\cE_\bSigma^{(\mu),\kappa}:\Sigma^{(\mu)}\rightarrow
Z^{(\mu)}_\kappa(\bSigma)$ (which are two to one except at the
origin). In turn, the latter induce bijections
$\hcE_{\bSigma}^{(\mu)}:\Sigma^{(\mu)}/\Z_2\stackrel{\sim}{\rightarrow}
Z^{(\mu)}(\bSigma)/\Z_2$. Theorem \ref{thm:reconstruction},
Proposition \ref{prop:constraintendopoly} and Lemma
\ref{lemma:actionnu} give:

\begin{cor}
\label{cor:reconstructionchiral}
Let $\bSigma$ be a paired simple $\Cl(V^\ast,h^\ast)$-module of
symmetry type $\sigma$ and adjoint type $s$. The following
statements are equivalent for $\alpha\in \wedge V^{\ast}$,
where $\mu\in \{-1,1\}$ is a fixed chirality type:
\begin{enumerate}[(a)]
\itemsep 0.0em
\item $\alpha$ lies in the set $Z^{(\mu)}(\bSigma)$, i.e. it is a
  signed square of a chiral spinor of chirality $\mu$.
\item The following conditions are satisfied:
\ben
\label{eq:reconstructionchiral0II}
(\pi^{\frac{1-s}{2}}\circ\tau)(\alpha) = \sigma\, \alpha ~~ , ~~
\ast\, (\pi\circ \tau)(\alpha) = \mu\, \alpha ~~,~~\alpha\diamond
\alpha =\cS(\alpha) \, \alpha ~~ , ~~ \alpha\diamond \beta\diamond
\alpha = \cS(\alpha\diamond\beta)\, \alpha
\een
for a fixed polyform $\beta \in \wedge V^{\ast}$ which satisfies
$\cS(\alpha\diamond\beta) \neq 0$.
\item The following conditions are satisfied:
\ben
\label{eq:reconstructionchiral0}
(\pi^{\frac{1-s}{2}}\circ\tau)(\alpha) = \sigma \alpha ~~ , ~~ \ast\,
(\pi\circ \tau)(\alpha) = \mu\, \alpha~~,~~\alpha\diamond
\beta\diamond \alpha =\cS(\alpha\diamond\beta)\, \alpha
\een
for every polyform $\beta \in \wedge V^{\ast}$.
\end{enumerate}
In this case, the real chiral spinor of chirality $\mu$ which
corresponds to $\alpha$ through the either of the maps
$\cE_\bSigma^{(\mu),+}$ or $\cE_\bSigma^{(\mu),-}$ is unique up to
sign and vanishes iff $\alpha = 0$.
\end{cor}

\noindent In particular, $Z^{(\mu)}(\bSigma)$ depends only on $\sigma,s$
and $(V^\ast, h^\ast)$ and will also be denoted by
$Z^{(\mu)}_{\sigma,s}(V^\ast, h^\ast)$.

\begin{cor}
Let $\alpha\in Z^{(+)}_{\sigma,s}(V^\ast, h^\ast)\cup
Z^{(-)}_{\sigma,s}(V^\ast, h^\ast)$. If $k\in \{1,\ldots, d\}$ satisfies:
\be
- (-1)^{k \frac{s-1}{2}} (-1)^{\frac{k(k-1)}{2}} = \sigma~~,
\ee
then we have $\alpha^{(k)} = 0$ and $\alpha^{(d-k)} = 0$. 
\end{cor}

\begin{proof}
Follows immediately from Corollary \ref{cor:signcriteria} and the
second relation in \eqref{eq:reconstructionchiral0II}.
\end{proof}

\subsection{Low-dimensional examples}

\noindent Let us describe $Z$ and $Z^{(\mu)}$ for
some low-dimensional cases. 


\subsubsection{Signature $(2,0)$}
\label{sec:Riemanexample}


Let $(V^{\ast},h^{\ast})$ be a two-dimensional $\R$-vector space with
a scalar product $h^{\ast}$. Its irreducible Clifford module
$(\Sigma,\gamma)$ is two-dimensional with an admissible pairing $\cB$
which is a scalar product. Theorem \ref{thm:reconstruction} with
$\beta =1$ shows that $w\in \wedge V^{\ast}$ is a signed square
of $\xi\in\Sigma$ iff:
\ben
\label{eq:2dEuclideanpolyform}
\alpha \diamond \alpha = 2\,\alpha^{(0)}\,\alpha ~~ , ~~ \tau(\alpha) = \alpha~~.
\een
Writing $\alpha = \alpha^{(0)}\oplus \alpha^{(1)} \oplus
\alpha^{(2)}$, the second of these relations reads:
\be
\alpha^{(0)} + \alpha^{(1)} - \alpha^{(2)} = \alpha^{(0)} + \alpha^{(1)} + \alpha^{(2)}~~.
\ee
This gives $\alpha^{(2)} = 0$, whence the first equation in
\eqref{eq:2dEuclideanpolyform} becomes $(\alpha^{(0)})^2 =
h^{\ast}(\alpha^{(1)},\alpha^{(1)})$. Hence $\alpha$ is a signed
square of a spinor iff:
\be
\alpha = \pm h^{\ast}(\alpha^{(1)},\alpha^{(1)})^{\frac{1}{2}} \oplus
\alpha^{(1)}~~\mathrm{with}~~\alpha^{(1)}\in V^{\ast}~~.
\ee
Let $\left\{e^i\right\}_{i=1,2}$ be an orthonormal basis of
$(V^{\ast},h^{\ast})$ and $\alpha = \cE_\bSigma^+(\xi)$ for some
$\xi\in\Sigma$. Then:
\be
2\,\alpha = \cB(\xi,\xi) + \cB(\gamma_i(\xi),\xi)\, e^i~~.
\ee
Thus:
\be
4\,h^{\ast}(\alpha^{(1)},\alpha^{(1)}) = \cB(\xi,\xi)^2
\ee
and hence the norm of $\xi$ determines the norm of one-form $\alpha^{(1)}\in V^{\ast}$.


\subsubsection{Signature $(1,1)$}
\label{sec:LorentzExample}


Let $(V^{\ast},h^{\ast})$ be a two-dimensional vector space $V^{\ast}$
equipped with a Lorentzian metric $h^{\ast}$. Its irreducible Clifford
module $(\Sigma,\gamma)$ is two-dimensional with a symmetric
admissible bilinear pairing $\cB$ of split signature and positive
adjoint type (see Theorem \ref{thm:admissiblepairings}). To guarantee
that $\alpha\in \wedge V^\ast$ belongs to $Z$, we should in principle
consider the first equation in \eqref{eq:thmdefequations} of Theorem
\ref{thm:reconstruction} for all $\beta \in \wedge V^{\ast}$. However,
$V^{\ast}$ is two-dimensional and Example \ref{ep:2dsplitsig} shows
that it suffice to take $\beta=1$. Thus $\alpha$ belongs to the set
$Z_{+,+}(V^\ast,h^\ast)$ iff:
\ben
\label{eq:2dLorentzpolyform}
\alpha \diamond \alpha = 2 \, \alpha^{(0)}\,\alpha ~~ , ~~ \tau(\alpha) =
\alpha~~.
\een
Writing $\alpha = \alpha^{0}\oplus \alpha^{(1)} \oplus \alpha^{(2)}$,
the second condition gives $\alpha^{(2)} = 0$, while the first
condition becomes:
\be
(\alpha^{(0)})^2 = h^{\ast}(\alpha^{(1)},\alpha^{(1)})~~.
\ee
In particular, $\alpha^{(1)}$ is space-like or null. Hence $\alpha$ is
a signed square of a spinor iff:
\ben
\label{eq:alpha11}
\alpha = \pm h^{\ast}(\alpha^{(1)},\alpha^{(1)})^{\frac{1}{2}} +
\alpha^{(1)}
\een
for a one-form $\alpha^{(1)}\in V^{\ast}$.  As in the Euclidean case,
we have:
\be
2\,\alpha = \cB(\xi,\xi) + \cB(\gamma_i(\xi),\xi) \, e^i~~,
\ee
whence:
\be
4\,h^{\ast}(\alpha^{(1)},\alpha^{(1)}) = \cB(\xi,\xi)^2~~.
\ee
Thus $\alpha^{(1)}$ is null iff $\cB(\xi,\xi) = 0$. In this signature
the volume form squares to $1$ and we have chiral spinors. Fix $\mu\in
\{-1,1\}$. By Corollary \ref{cor:reconstructionchiral}, $\alpha$ lies
in the set $Z^{(\mu)}_{+,+}(V^\ast,h^\ast)$ iff it has the form
\eqref{eq:alpha11} and satisfies the supplementary condition:
\be
\ast\,(\pi\circ\tau)(\alpha) = \mu \, \alpha~~.
\ee
This amounts to the following system, where $\nu_h$ is the volume form of
$(V^{\ast},h^{\ast})$:
\be
\pm h^{\ast}(\alpha^{(1)},\alpha^{(1)})^{\frac{1}{2}} \nu_h - \ast \alpha^{(1)} = 
\pm \mu\, h^{\ast}(\alpha^{(1)},\alpha^{(1)})^{\frac{1}{2}} + \mu\, \alpha^{(1)}~~.
\ee 
Thus $h^{\ast}(\alpha^{(1)},\alpha^{(1)}) = 0$
and $\ast \alpha^{(1)} = -\mu\, \alpha^{(1)}$. Hence a signed polyform
square of a chiral spinor of chirality $\mu$ is a null one-form which
is anti-self-dual when $\mu = +1$ and self-dual when $\mu = -
1$. Notice that the nullity condition on $\alpha^{(1)}$ is 
equivalent with (anti-)selfduality.


\subsubsection{Signature $(3,1)$}
\label{sec:4dLorentzexample}


This case is relevant for supergravity applications and will arise in
Sections \ref{sec:RKSpinors} and \ref{sec:Susyheterotic}. Let
$(V^{\ast},h^{\ast})$ be a Minkowski space of ``mostly plus'' signature $(3,1)$. Its
irreducible Clifford module $(\Sigma,\gamma)$ is four-dimensional and
both admissible pairings $\cB_\pm$ are skew-symmetric. We work with
the admissible pairing $\cB=\cB_-$ of negative adjoint type.

\begin{definition}
\label{def:parabolic}
A {\em parabolic pair} of one-forms is ordered pair $(u,l)\in
V^\ast\times V^\ast$ such that $u\neq 0$ and:
\ben
\label{eq:ulcond}
h^\ast (u,u) = 0~~,~~ h^\ast (l,l) = 1~~,~~ h^{\ast}(u,l) = 0~~,
\een
i.e. $u$ is nonzero and null, $l$ is spacelike of unit norm and $u$ is
orthogonal to $l$. Two parabolic pairs of one forms $(u,l)$ and
$(u',l')$ are called:
\begin{itemize}
\itemsep 0.0em
\item {\em weakly equivalent}, if there exist $b\in \R^\times$, $c\in
  \R$ and $\eta\in \{-1,1\}$ such that:
\ben
\label{weakpairequiv}
u'=b u~~\mathrm{and}~~l'=\eta l+cu~~.
\een
\item {\em equivalent} (and we write $(u,l)\equiv (u',l')$)
if there exist $b\in \R^\times$ and $c\in \R$
such that:
\ben
\label{pairequiv}
u'=b u~~\mathrm{and}~~l'=l+cu~~.
\een
\item {\em strongly equivalent} (and we write $(u,l)\sim
(u',l')$) if there exist $\zeta\in \{-1,1\}$ and $c\in \R$ such that:
\ben
\label{strongpairequiv}
u'=\zeta u~~\mathrm{and}~~l'=l+cu~~.
\een
\end{itemize}
\end{definition}

\noindent Let $\cP(V^\ast,h^\ast)$ denote the set of parabolic pairs
of one-forms. The binary relations defined above are equivalence
relations on this set; moreover, strong equivalence implies
equivalence, which in turn implies weak equivalence.

\

\noindent Recall that a 2-plane $\Pi$ in $V^\ast$ is called {\em
  parabolic} (with respect to $h^\ast$) if the restriction
$h^\ast_\Pi$ of $h^\ast$ to $\Pi$ has one-dimensional kernel. This
happens iff $\Pi$ is tangent to the light cone of the Minkowski space
$(V^\ast, h^\ast)$ along a null line. This line coincides with
$K_h(\Pi)\eqdef \ker(h^\ast_\Pi)$ and is called the {\em null line} of
$\Pi$.  If $\Pi\subset V^\ast$ is a parabolic 2-plane, then any
element of $\Pi$ which does not belong to $\K_h(\Pi)$ is
spacelike. The two connected components of the complement
$\Pi\setminus \K_h(\Pi)$ are the {\em spacelike half-planes} of
$\Pi$. An orientation of the null line $\K_h(\Pi)$ is called a {\em
  time orientation} of $\Pi$, while an orientation of the quotient
line $\Pi/\K_h(\Pi)$ is called a {\em co-orientation} of $\Pi$. Notice
that a co-orientation of $\Pi$ amounts to a choice $\cH$ of one of the
spacelike half-spaces of $\Pi$.  A {\em co-oriented parabolic 2-plane}
in $V^\ast$ is a pair $(\Pi,\cH)$, where $\Pi$ is a parabolic
two-plane in $V^\ast$ and $\cH$ is a co-orientation of $\Pi$. The set
of spacelike unit norm elements of $\Pi$ has two connected components,
each of which is an affine line parallel to $\K_h(\Pi)$. These two
affine lines are related by the inversion of $\Pi$ with respect to the
origin. Notice that a co-orientation $\cH$ of $\Pi$ amounts to a
choice $L$ of one of these two affine lines. Namely, we associate to
$L$ that spacelike half-plane $\cH_L$ of $\Pi$ which contains
$L$. Given $u\in \K_h(\Pi)\setminus \{0\}$, a unit norm spacelike
element $l\in \Pi$ is determined up to transformations of the form
$l\rightarrow \zeta l+c u$, where $\zeta\in \{-1,1\}$ and $c\in \R$.

\begin{remark}
Parabolic 2-planes correspond to degenerate complete flags in
$(V^\ast,h^\ast)$ (see Appendix \ref{app:flags}). Notice that a
parabolic 2-plane $\Pi$ determines a short exact sequence of vector
spaces:
\be
0\rightarrow K\rightarrow \Pi\rightarrow N\rightarrow 0
\ee
with $K=\K_h(\Pi)$ and $N=\Pi/K$ and induces a scalar product on the
quotient line $N$. Conversely, giving a ``parabolic'' metric on a
2-plane $\Pi$ amounts to giving a short exact sequence of this form
together with a scalar product on $N$. A time orientation of $\Pi$ is
orientation of $K$ while a co-orientation is an orientation of
$N$. Since the determinant line of $\Pi$ is given by
$\det(\Pi)=\wedge^2 \Pi=K\otimes L$, a time orientation and a
co-orientation taken together determine an orientation of $\Pi$.
\end{remark}

\noindent A basis of a parabolic 2-plane $\Pi\subset V^\ast$ is
called {\em parabolic} if its two elements form a parabolic pair. By
Sylvester's theorem, any parabolic plane $\Pi$ admits parabolic bases.

\begin{prop}
\label{prop:pairs2planes}
The map $(u,l)\rightarrow \Span_\R(u,l)$ induces a bijection between
the set of weak equivalence classes of parabolic pairs of one-forms
and the set of all parabolic 2-planes in $(V^\ast,h^\ast)$.
\end{prop}

\begin{proof}
If $(u,v)$ is a parabolic pair, then
$\Span_\R(u,v)$ is a parabolic 2-plane, which depends only on the weak
equivalence class of $(u,v)$. Conversely, it is easy to see that any
two parabolic bases of a parabolic 2-plane $\Pi$ in $V^\ast$ are
weakly-equivalent as parabolic pairs.
\end{proof}

\noindent Proposition
\ref{prop:pairs2planes} implies:

\begin{cor}
\label{cor:pairsframed2planes}
The map $(u,l)\rightarrow (\Span_\R(u,l),\cH_l)$ induces a bijection
between the set $\cP(V^\ast,h^\ast)/_{\equiv}$ of equivalence classes
of parabolic pairs of one-forms and the set of all co-oriented
parabolic 2-planes in $(V^\ast,h^\ast)$, where $\cH_l$ is the unique
spacelike half-plane of the parabolic 2-plane $\Span_\R(u,l)$ which
contains the vector $l$.
\end{cor}

\begin{thm}
\label{thm:squarespinorMink}
A polyform $\alpha\in \wedge V^{\ast}$ is a signed square of a nonzero
spinor (i.e. it belongs to the set $Z_{-,-}(V^\ast,h^\ast)$) iff it
has the form:
\ben
\label{eq:alphaul}
\alpha = u + u\wedge l
\een
for a parabolic pair of one-forms $(u,l)\in \cP(V^\ast,h^\ast)$. In
this case, $u$ is uniquely determined by $\alpha$ while $l$ is
determined by $\alpha$ up to transformations of the form:
\ben
\label{eq:lgauge}
l\rightarrow l+cu~~,
\een
where $c\in \R$ is arbitrary. Moreover, $(u,l)$ is determined by the
sign equivalence class of $\alpha$ up to strong equivalence of
parabolic pairs. This gives a natural bijection between
the sets $Z_{-,-}(V^\ast,h^\ast)/\Z_2$ and $\cP(V^\ast,h^\ast)/_\sim$.
\end{thm}

\begin{proof}
Let:
\be
\alpha = \sum_{k=0}^{4} \alpha^{(k)}\in \wedge
V^\ast~~\mathrm{where}~~\alpha^{(k)}\in \wedge^k V^\ast \,\,\,\,\,\, \forall
k=0,\ldots, 4.
\ee
Fix an orthonormal basis $\left\{e^a\right\}_{a =0, \hdots, 3}$ with
$e^0$ time-like. By Theorem \ref{thm:reconstruction}, $\alpha$ lies in
$Z_{-,-}(V^\ast,h^\ast)$ iff the following relations hold for $\beta
=1 $ and for a polyform $\beta$ such that $(\beta\diamond\alpha)^{(0)}
\neq 0$:
\ben
\label{eq:eqs3,1}
\alpha\diamond\beta\diamond \alpha = 4\, (\beta\diamond\alpha)^{(0)}\,
\alpha ~~ , ~~ (\pi\circ\tau)(\alpha) = -\alpha~~.
\een
The condition $(\pi\circ\tau)(\alpha) = -\alpha$ gives $\alpha^{(0)} =
\alpha^{(3)} =\alpha^{(4)} = 0$.  Thus $\alpha = u + \omega$, where $u
\eqdef \alpha^{(1)}\in \wedge^1 V^\ast$ and $\omega \eqdef
\alpha^{(2)}\in \wedge^2 V^\ast$. For $\beta =1$, the first condition
in \eqref{eq:eqs3,1} gives $(u + \omega)\diamond (u + \omega) = 0$,
which reduces to the following relations upon expanding the geometric
product:
\ben
\label{eq:eqseqiv3,1}
h^{\ast}(u,u) = \langle \omega,\omega\rangle_h ~~ , ~~ \omega\wedge u = 0~~.
\een
Here $\langle~,~\rangle_h$ is the metric induced by $h$ on $\wedge
V^\ast$. The second condition in \eqref{eq:eqseqiv3,1} amounts to
$\omega = u\wedge l$ for some $l\in V^{\ast}$ determined up
to the transformations \eqref{eq:lgauge}. Using this in
\eqref{eq:eqseqiv3,1} gives the condition:
\ben
\label{eq:norms}
(h^\ast(l,l)-1)\, h^\ast (u,u) = h^{\ast}(u,l)^2~~,
\een
which is invariant under the transformations \eqref{eq:lgauge}. For
$\beta = u$, the first equation in \eqref{eq:eqs3,1} amounts to
$h^\ast(u,u) = 0$, whence $h^{\ast}(u,l)= 0$ by \eqref{eq:norms}. It
remains to show that $h^\ast(l,l)=1$. Since $u$ is non-zero and null,
there exists a non-zero null one-form $v\in V^{\ast}$ such that
$h^\ast(v,u)= 1$. Then $(v\diamond v)^{(0)}=(v\diamond (u + u\wedge
l))^{(0)} = h^\ast(v,u)=1$.  For $\beta=v$, the first condition in
\eqref{eq:eqs3,1} reduces to:
\be
(u + u\wedge l)\diamond v \diamond (u + u\wedge l) =  4\,(u + u\wedge l)~~.
\ee
Direct computation shows that this equation amounts to $h^\ast(l,l) = 1$
and we conclude.
\end{proof}

\begin{remark} 
Given $v\in V^{\ast}$ with $h^\ast(v,v)=-1$, denote by
$P_v\colon V^{\ast}\to \R v$ the orthogonal projection onto
the line $\R v=\Span_\R(v)$. A canonical choice of $l$ is obtained by
imposing the condition:
\be
P_v(l) = 0~~.
\ee
Given $l\in V^\ast$ of unit norm and orthogonal to $u$, there
exists a unique $c\in \R$ such that $P_v(l+c\, u) = 0$.
This ``choice of gauge'' could be useful for spinors on time-oriented
Lorentzian four-manifolds.
\end{remark}

\noindent Corollary \ref{cor:pairsframed2planes} and Theorem
\ref{thm:squarespinorMink} imply the following result.

\begin{cor}
\label{cor:MinkHyp}
The projective spinor squaring map $\P\cE_\bSigma$ induces a bijection between
$\P(\Sigma)$ and the set of all co-oriented parabolic 2-planes $(\Pi,\cH)$
in $(V^\ast,h^\ast)$.  Moreover, there exist natural bijections
between the following three sets:
\begin{itemize}
\itemsep 0.0em
\item The set $\dot{\Sigma}/\Z_2$ of sign equivalence classes of
  nonzero spinors.
\item The set $\cP(V^\ast,h^\ast){/}_\sim$ of {\em strong} equivalence
  classes of parabolic pairs of one-forms.
\item The set of triples $(\Pi,\cH,{\hat u})$, where $(\Pi,\cH)$ is a
  co-oriented parabolic 2-plane in $(V^\ast, h^\ast)$ and ${\hat u}$ is the
  sign equivalence class of a non-zero element $u\in \K_h(\Pi)$.
\end{itemize}
\end{cor}

\begin{proof}
A line $\R\xi\in \P(\Sigma)$ corresponds to the line $\R \alpha\in
\P(Z_{-,-}(V^\ast,h^\ast))\subset \P(\wedge V^\ast)$. By Theorem
\ref{thm:squarespinorMink}, $\alpha$ determines a null one-form
$u=\alpha^{(1)}$ and a line $L$ of spacelike vectors of unit norm
which is parallel to the line $\Span_\R(u)=\R u$. Pick any $l\in L$
and set $\Pi\eqdef\Span(u,l)=\R u\oplus L$. Then $\Pi$ is a parabolic
2-plane in $V^\ast$ depending only $\R u$ and $L$ and we have
$\K_h(\Pi)=\R u$. Rescaling $\alpha$ by a non-zero real number
corresponds to rescaling $u=\alpha^{(1)}$ by the same. Hence
$\R\alpha$ determines the line $\R u=K_h(\Pi)$ and 
relation \eqref{eq:alphaul} shows that $\R\alpha$ determines and is
determined by the co-oriented parabolic 2-plane $(\Pi,\cH_L)$. This
proves the first statement.

Now recall that the sign-equivalence class of a non-zero spinor $\xi$
determines and is determined by the sign equivalence class of its
square polyform $\alpha$ through the map ${\hat \cE}_\bSigma$. By
\eqref{eq:alphaul}, the sign change $\alpha\rightarrow -\alpha$
corresponds to $u\rightarrow -u$ and $l\rightarrow l$. Thus the sign
equivalence class of $u$ is uniquely determined by that of $\alpha$
and hence by that of $\xi$.  This establishes the bijection between
the three sets in the second statement.
\end{proof}

\begin{remark}
A parabolic pair $(u,l)$ and a polyform square $\alpha$ are recovered from the
triplet $(\Pi,\cH,{\hat u})$ by taking $u$ to be any representative of
the sign equivalence class ${\hat u}$ and $l$ to be any vector lying
on the unit norm affine line $L$ contained in $\cH$ and setting
$\alpha=u+u\wedge l$.
\end{remark}

\noindent
Let us count the degrees of freedom encoded in $\alpha = u + u\wedge
l$. Apriori, the null one-form $u$ has three degrees of freedom while
the space-like one-form $l$ has four, which are reduced to two by the
requirement that $l$ has unit norm and is orthogonal to $u$. Since $l$
is defined only up to $l\mapsto l + c\, u$ ($c \in \R$), its degrees
of freedom further reduce from two to one. This gives a total of four
degrees of freedom, matching those of a real spinor in
four-dimensional Lorentzian signature.


\subsubsection{Signature $(2,2)$}
\label{sec:(2,2)}


Let $(V^{\ast},h^{\ast})$ be four-dimensional with metric $h^{\ast}$
of split signature. Its Clifford module $(\Sigma,\gamma)$ is
four-dimensional and has a skew-symmetric admissible pairing $\cB$ of
positive adjoint type (see Theorem \ref{thm:admissiblepairings}). This
dimension and signature admits chiral spinors. Let:
\be
\alpha = \sum_{k=0}^{4} \alpha^{(k)}\in \wedge
V^{\ast}~~\mathrm{with}~~\alpha^{(k)}\in \wedge^k V^\ast \,\,\,\,\,\,
\forall k=1,\ldots 4~~.
\ee
Fixing an orthonormal basis $\left\{e^a\right\}_{a =1, \hdots, 4}$ of
$(V^{\ast},h^{\ast})$ with $e^1, e^2$ timelike, define timelike and
spacelike volume forms through $\nu_{-} = e^1\wedge e^2$ and $\nu_{+}
= e^3\wedge e^4$. By Corollary \ref{cor:reconstructionchiral}, we have
$\alpha\in Z^{(\mu)}_{-,+}(V^\ast,h^\ast)$ iff:
\ben
\label{eq:eqs2,2}
\alpha\diamond \alpha = 0 ~~ , ~~ \tau(\alpha) = -\alpha~~, 
~~ \ast \pi(\tau(\alpha)) = \mu\, \alpha~~,~~   
\alpha\diamond \beta \diamond \alpha = 4\, (\beta\diamond\alpha)^{(0)}\, \alpha
\een
for a polyform $\beta\in \wedge V^{\ast}$ such that
$(\beta\diamond\alpha)^{(0)}\neq 0$. Here we used skew-symmetry of
$\cB$, which implies $ \alpha^{(0)} = 0$. The condition $\tau(\alpha)
= -\alpha$ amounts to:
\be
\alpha^{(0)} = \alpha^{(1)} = \alpha^{(4)} = 0~~,
\ee
whereas the condition $\ast \pi(\tau(\alpha)) = \mu\, \alpha$ is
equivalent with:
\be
\ast\alpha^{(2)} = - \mu\, \alpha^{(2)}~~,~~ \alpha^{(3)} = 0~~.
\ee
Thus it suffices to consider $\alpha = \omega$, where
$\omega$ is selfdual if $\mu = -1$ and anti-selfdual if $\mu =
1$. In signature $(2,2)$, the Hodge star operator squares
to the identity and yields a decomposition:
\be
\wedge^2 V^{\ast} = \wedge^2_{+} V^{\ast} \oplus \wedge^2_{-} V^{\ast}~~,
\ee
into self-dual and anti-selfdual two-forms. This corresponds to the
decomposition $\mathfrak{so}(2,2) = \mathfrak{sl}(2) \oplus \mathfrak{sl}(2)$
of the Lie algebra $\mathfrak{so}(2,2) = \wedge^2 V^{\ast}$.
Expanding the geometric product shows that the first equation in
\eqref{eq:eqs2,2} reduces to the following condition for a selfdual
or anti-selfdual two-form $\alpha = \omega$:
\be
\langle \omega,\omega\rangle_h = 0~~.
\ee
For simplicity of exposition we set $\mu = -1$ in what follows,
in which case $\omega$ is self-dual (analogous results hold for
$\mu =1$). Consider the basis $\left\{u_a\right\}_{a = 1, 2, 3}$ of
$\wedge^2_+V^\ast$ given by:
\be
u_1 \eqdef e^1\wedge e^2 + e^3\wedge e^4~~,~~  u_2 \eqdef e^1\wedge e^3 + 
e^2\wedge e^4~~,~~ u_3 \eqdef e^1\wedge e^4 - e^2\wedge e^3~~,
\ee
and expand:
\be
\omega = \sum k^a  u_a~~.
\ee
We have:
\be
\nu_{-}\diamond u_1 = u_1\diamond \nu_{-} = -1 + \nu_h~~,~~
\nu_{-}\diamond u_2 = - u_2\diamond\nu_{-} = -u_3~~,~~
\nu_{-}\diamond u_3 = -u_3\diamond\nu_{-} = u_2~~,
\ee
which gives:
\be
(\nu_{-}\diamond\omega)^{(0)} = - k^1~~.
\ee
Furthermore, we compute:
\begin{eqnarray*}
& u_1\diamond u_3 = - u_3 \diamond u_1 = 2\, u_2 ~~,~~
  u_1\diamond u_2 = -u_2\diamond u_1 = -2\, u_3~~,~~ u_2\diamond
  u_3 = - u_3\diamond u_2 = 2\, u_1~~,\\ & u_1\diamond u_1 = - u_2
  \diamond u_2 = - u_3 \diamond u_2 = -2 + 2 \, \nu_h~~.
\end{eqnarray*}
These products realize the Lie algebra $\mathfrak{sl}(2,\R)$
upon defining a Lie bracket by the commutator:
\be
[u_1 , u_2] = u_1 \diamond u_2 - u_2 \diamond u_1= - 4\, u_3~~,~~
[u_1 , u_3] = u_1 \diamond u_3 - u_3 \diamond u_1= 4\, u_2~~,~~
[u_2 , u_3] = u_2 \diamond u_3 - u_3 \diamond u_2= 4\, u_1~~.
\ee
Since $\wedge^2_{+} V^{\ast} = \mathfrak{sl}(2,\R)$, the Killing form
$\mathfrak{B}$ of $\mathfrak{sl}(2,\R)$ gives a symmetric
non-degenerate pairing of signature $(1,2)$ on $\wedge^2_{+}
V^{\ast}$, which can be rescaled to coincide with that induced induced
by $h$. Then:
\be
\mathfrak{B}(\omega,\omega) = \langle \omega,\omega\rangle_h^2 = 2
\left[(k^ 1)^2 - (k^2)^2 - (k^3)^2\right] \,\,\,\,\,\, \forall \omega
\in \wedge^2_{+} V^{\ast}~~.
\ee

\begin{prop}
\label{prop:(2,2)}
A polyform $\alpha \in \wedge V^{\ast}$ is a signed square of a real
chiral spinor $\xi \in \Sigma^{(-)}$ of negative chirality iff
$\alpha$ is a self-dual two-form of zero norm.
\end{prop}

\begin{proof} 
It suffices to consider the case $\alpha\neq 0$. By the discussion
above, a non-zero polyform $\alpha\neq 0$ belongs to the set
$Z^{(-)}_{-,+}(V^\ast,h^\ast)$ only if $\alpha=\omega$ is self-dual
and of zero norm (which is equivalent to the first three equations in
\eqref{eq:eqs2,2}). Once these conditions are satisfied, the only
equation that remains to be solved is the fourth equation in
\eqref{eq:eqs2,2}. To solve it, we take $\beta = \nu_{-}$. Since
$(\nu_{-}\diamond\omega)^{(0)} = -4\, k^1$ (as remarked above), we
conclude that $(\nu_{-}\diamond\omega)^{(0)} \neq 0$ iff $\omega \neq
0$, whence taking $\beta = \nu_{-}$ is a valid choice. A computation
shows that this equation is automatically satisfied and thus we
conclude.
\end{proof}

\begin{remark} 
Subsection \ref{sec:LorentzExample} together with Proposition
\ref{prop:(2,2)} show that the square of a chiral spinor in signatures
$(1,1)$ and $(2,2)$ is given by an (anti-)self-dual form of zero norm
in middle degree. The reader can verify, through a computation similar
to the one presented in this subsection, that the same statement holds
in signature $(3,3)$. It is tempting to conjecture that the
square of a chiral spinor in general split signature $(p,p)$
corresponds to an (anti-)self-dual $p$-form of zero norm, the latter
condition being automatically implied when $p$ is odd. Verifying this
conjecture would be useful in the study of manifolds of split
signature which admit parallel chiral spinors \cite{Dunajski}.
\end{remark}


\section{Constrained Generalized Killing spinors of real type}
\label{sec:GCKS}


To study constrained generalized Killing spinors of real type, we will
extend the theory of Section \ref{sec:SpinorsAsPolyforms} to bundles
of real irreducible Clifford modules equipped with an arbitrary
connection. Throughout this section, let $(M,g)$ denote a 
  connected pseudo-Riemannian manifold of signature $(p,q)$ and even
dimension $d=p+q\geq 2$, where $p-q\equiv_8 0,2$. Since $M$ is
connected, the pseudo-Euclidean vector bundle $(TM,g)$ is modeled on a
fixed quadratic vector space denoted by $(V,h)$. For any point $m\in
M$, we thus have an isomorphism of quadratic spaces $(T_mM,g_m)\simeq
(V,h)$. Accordingly, the cotangent bundle $T^\ast M$ (endowed
with the dual metric $g^\ast$) is modeled on the dual quadratic space
$(V^{\ast},h^{\ast})$. We denote by $\Cl(M,g)$ the bundle of real
Clifford algebras of the {\em cotangent} bundle $(T^\ast M,g^\ast)$,
which is modeled on the real Clifford algebra
$\Cl(V^{\ast},h^{\ast})$. Let $\pi$ and $\tau$ be the canonical
automorphism and anti-automorphism of the Clifford bundle, given by
fiberwise extension of the corresponding objects defined in Section
\ref{sec:SpinorsAsPolyforms} and set ${\hat \pi}=\pi\circ \tau$.  We
denote by $(\Lambda(M),\diamond)$ the exterior bundle
$\Lambda(M)=\oplus_{j=0}^d \wedge^j T^\ast M$, equipped with the
pointwise extension $\diamond$ of the geometric product of Section
\ref{sec:SpinorsAsPolyforms} (which depends on the metric $g$). This
bundle of unital associative algebras is called the
\emph{K\"ahler-Atiyah bundle} of $(M,g)$ (see
\cite{LazaroiuB,LazaroiuBC}).  The map $\Psi$ of Section
\ref{sec:SpinorsAsPolyforms} extends to a unital
isomorphism of bundles of algebras:
\be
\Psi\colon (\Lambda(M),\diamond) \stackrel{\sim}{\rightarrow} \Cl(M,g)~~,
\ee
which allows us to view the K\"ahler-Atiyah bundle as a model for the
Clifford bundle. We again denote by $\pi$, $\tau$ and ${\hat
  \pi}=\pi\circ \tau$ the (anti-)automorphisms of the \KA bundle
obtained by transporting the corresponding objects from the Clifford
bundle through $\Psi$. The \KA trace of Section
\ref{sec:SpinorsAsPolyforms} extends to a morphism of vector bundles:
\be
\cS:\Lambda(M)\rightarrow \R_M
\ee
whose induced map on smooth sections satisfies:
\be
\cS(1_M)=N=2^{\frac{d}{2}} 1_M~~ \mathrm{and} ~~
\cS(\omega_1\diamond\omega_2)=\cS(\omega_2\diamond
\omega_1)\,\,\,\,\,\, \forall \omega_1,\omega_2\in \Omega^\ast(M)~~,
\ee
where $1_M\in \Gamma(\R_M)=\Omega^0(M)$ is the unit function defined
on $M$. By Proposition \ref{prop:cS}, we have:
\be
\cS(\omega)=2^{\frac{d}{2}}\omega^{(0)}\,\,\,\,\,\, \forall \omega \in \Omega(M)~~.
\ee
In particular, $\cS$ does not depend on the metric $g$.  The
following encodes a well-know property of the Clifford bundle, which
also follows from the definition of $\diamond$ (cf.
\cite{LazaroiuB,LazaroiuBC}). 

\begin{prop}
\label{prop:LCderClifford}
The canonical extension to $\Lambda(M)$ of the Levi-Civita connection $\nabla^g$ of
$(M,g)$ to $\Lambda(M)$ (which we again denote by $\nabla^g$) acts by
derivations of the geometric product:
\be
\nabla^g(\alpha\diamond \beta) = (\nabla^g\alpha)\diamond \beta + 
\alpha\diamond (\nabla^g\beta) \,\,\,\,\,\, \forall \alpha, \beta\in \Omega(M)~~.
\ee
\end{prop}

\noindent

\subsection{Bundles of real simple Clifford modules}

\begin{definition}
A {\em bundle of (real) Clifford modules} on $(M,g)$ is a pair
$(S,\Gamma)$, where $S$ is a real vector bundle on $M$ and
$\Gamma:\Cl(M,g)\rightarrow \End(S)$ is a unital morphism of bundles
of algebras (which we call the {\em structure map}).
\end{definition}

\noindent Since $M$ is connected, any bundle of Clifford modules
$(S,\Gamma)$ on $(M,g)$ is modeled on a Clifford representation
$\gamma\colon \Cl(V^{\ast},h^{\ast})\to \End(\Sigma)$ (called its {\em
  model representation}), where $\Sigma$ is a vector space isomorphic
to the fiber of $S$. For every point $m\in M$, the unital morphism of
associative algebras $\Gamma_m:\Cl(T_m^\ast M, g_m^\ast) \to
(\End(S_m),\circ)$ identifies with the representation morphism
$\gamma$ upon composing appropriately with the unital algebra
isomorphisms $\Cl(T_m^\ast, g_m^\ast)\simeq \Cl(V^\ast,h^\ast)$ and
$\End(S_m)\simeq \End(\Sigma)$ (the latter of which is induced by the
linear isomorphism $S_m\simeq \Sigma$).

\begin{definition}
We say that $(S,\Gamma)$ is a bundle of simple real Clifford modules
(or a {\em real spinor bundle}) if its model representation $\gamma$
is irreducible. In this case, a global section $\epsilon\in \Gamma(S)$
is called\footnote{Since $S$ need not be associated to a spin
  structure on $(M,g)$, this generalizes the traditional notion of
  spinor. In signatures $p-q\equiv_8 0,2$, $S$ is associated to an
  untwisted $\Pin$ structure (see \cite{Lazaroiu:2016vov}) so its
  sections could also be called ``pinors''.} a {\em spinor} on
$(M,g)$.
\end{definition}

\noindent In the signatures $p-q\equiv_8 0,2$ considered in this
paper, a simple bundle of Clifford modules satisfies $\rk\, S=\dim
V=2^{\frac{d}{2}}$, where $d$ is the dimension of $M$. Reference
\cite{Lazaroiu:2016vov} proves that $(M,g)$ admits a bundle of simple
real Clifford modules iff it admits a {\em real Lipschitz structure}
of type $\gamma$. In signatures $p - q\equiv_8 0,2$, the latter
corresponds to an adjoint-equivariant (a.k.a. ``untwisted'')
$\Pin(V^{\ast},h^{\ast})$-structure $\bQ$ on $(M,g)$ and $(S,\Gamma)$
is isomorphic (as a unital bundle of algebras) with the bundle of real
Clifford modules associated to $\bQ$ through the natural
representation of $\Pin(V^{\ast},h^{\ast})$ in $\Sigma$. The
obstructions to existence of such structures were given in
\cite{Lazaroiu:2016vov}; when $p-q\equiv_8 0,2$, they are a slight
modification of those given in \cite{Karoubi} for ordinary
(twisted adjoint-equivariant) $\Pin(V^\ast,h^\ast)$-structures.

\begin{prop}
\label{prop:Ltensor}
Let $(S,\Gamma)$ be a bundle of real Clifford modules on $(M,g)$, $L$
a real line bundle on $M$ and set $S_L\eqdef S\otimes L$. Then there
exists a natural unital morphism of bundles of algebras
$\Gamma_L:\Cl(M,g)\rightarrow End(S\otimes L)$. Hence the {\em
  modification} $(S_L,\Gamma_L)$ of $(S,\Gamma)$ by $L$ is a bundle of
Clifford modules, which is a real spinor bundle iff $(S,L)$ is. In
particular, the real Picard group $\Pic(M)$ acts naturally on the set
of isomorphism classes of bundles of real Clifford modules defined over $(M,g)$.
\end{prop}

\begin{proof}
There exists a unique trivialization $\psi_L:End(L)\simeq \R_M$ of the
line bundle $End(L)$ which is a unital isomorphism of bundles of
$\R$-algebras -- namely that trivialization which sends the identity
endomorphism of $L$ into the unit section of $\R_M$ (which is the
constant function equal to $1$ defined on $M$). This induces a unital
isomorphism of bundles of algebras $\varphi_L:End(S\otimes
L)\stackrel{\sim}{\rightarrow} End(S)$ given by composing the natural
isomorphism of bundles of $\R$-algebras $End(S\otimes
L)\stackrel{\sim}{\rightarrow} End(S)\otimes End(L)$ with
$\Id_{End(S)}\otimes \psi_L$. The conclusion follows by setting
$\Gamma_L\eqdef \varphi_L^{-1}\circ \Gamma$.
\end{proof}

\noindent The map $\Psi_\gamma$ of Section
\ref{sec:SpinorsAsPolyforms} extends to a unital isomorphism of
bundles of algebras:
\be
\Psi_{\Gamma}\eqdef \Gamma\circ \Psi \colon (\Lambda(M),\diamond)
\stackrel{\sim}{\rightarrow} (End(S),\circ)~~,
\ee
which allows us to identify bundles $(S,\Gamma)$ of modules over
$\Cl(T^\ast M, g^\ast)$ with bundles of modules $(S,\Psi_\Gamma)$ over
the \KA algebra. We denote by a dot the external
multiplication\footnote{Through the isomorphisms explained above, this
  corresponds to Clifford multiplication on the vector bundle $S$,
  whose existence amounts to existence of the corresponding real
  Lispchitz structure on $(M,g)$ by the results of
  \cite{Lazaroiu:2016vov}.}  of $(S,\Psi_\Gamma)$, whose action on
global sections is:
\be
\alpha\cdot\epsilon \eqdef \Psi_{\Gamma}(\alpha)(\epsilon) \,\,\,\,\,\, \forall
\alpha \in \Omega(M)\eqdef \Gamma(\Lambda(M))~~\forall \epsilon\in
\Gamma(S)~~.
\ee
Let $\tr:End(S)\rightarrow \R_M$ be the fiberwise trace  morphism, whose map
induced on sections we denote by the same symbol. The results of Section
\ref{sec:SpinorsAsPolyforms} imply:

\begin{prop}
Let $(S,\Gamma)$ be a real spinor bundle. Then:
\be
\cS(\omega)=\tr(\Psi_\Gamma(\omega))\,\,\,\,\,\, \forall \omega\in \Omega(M)~~.
\ee
\end{prop}

\begin{definition}
\label{def:symbol}
Let $(S,\Gamma)$ be a real spinor bundle on $(M,g)$ and $U$ be any
vector bundle on $M$. The {\em symbol} of a section $W\in
\Gamma(End(S)\otimes U)$ is the section ${\hat W}\in \Gamma(\Lambda(M)
\otimes U)$ defined through:
\be
{\hat W}\eqdef (\Psi_\Gamma\otimes \Id_U)^{-1}(T)\in \Gamma(\wedge T^\ast M \otimes U)~~,
\ee
where $\Id_U$ is the identity endomorphism of $U$. 
\end{definition}

\begin{remark}
In particular, the symbol of an endomorphism $\cQ\in \Gamma(End(S))$
is a polyform ${\hat \cQ}\in \Omega(M)$, while the symbol of an
$End(S)$-valued one-form $\cA\in \Gamma(T^\ast M\otimes End(S))$ is an
element ${\hat \cA}\in \Gamma(M,T^\ast M\otimes \wedge T^\ast
M)=\Omega^1(M,\Lambda(M))=\Omega^\ast(M,T^\ast M)$, which can be
viewed as a $T^\ast M$-valued polyform or as a $\Lambda(M)$-valued
1-form.
\end{remark}

\subsection{Paired spinor bundles}

\begin{definition}
Let $(S,\Gamma)$ be a real spinor bundle on $(M,g)$. A
fiberwise-bilinear pairing $\cB$ on $S$ is called {\bf admissible} if
$\cB_m:S_m\times S_m\rightarrow \R$ is an admissible pairing on the
simple Clifford module $(S_m,\Gamma_m)$ for all $m\in M$.  A (real)
{\em paired spinor bundle} on $(M,g)$ is a triplet $\bS=(S,\Gamma,\cB)$,
where $(S,\Gamma)$ is a real spinor bundle on $(M,g)$ and $\cB$ is an
admissible pairing on $S$.
\end{definition}

\noindent Since $M$ is connected, the symmetry and adjoint type
$\sigma,s\in \{-1,1\}$ of the admissible pairings $\cB_m$
(which are non-degenerate by definition) are constant on $M$; they are
called the {\em symmetry type} and {\em adjoint type} of $\cB$ and of
$(S,\Gamma,\cB)$. An admissible pairing on $(S,\Gamma)$ can be viewed
as a morphism of vector bundles $\cB:S\otimes S\rightarrow \R_M$,
where $\R_M$ is the trivial real line bundle on $M$. Since $M$ is
paracompact, the defining algebraic properties of an admissible
pairing can be formulated equivalently as follows using global
sections (see \cite{LazaroiuBC}), when viewing $(S,\Gamma)$ as a
bundle $(S,\Psi_\Gamma)$ of modules over the \KA algebra of $(M,g)$:
\begin{enumerate}[1.]
\itemsep 0.0em
\item $\cB(\xi_1,\xi_2)=\sigma \cB(\xi_2,\xi_2) \,\,\,\,\,\, \forall
  \xi_1,\xi_2\in \Gamma(S)$
\item
$\cB(\Psi_\Gamma(\omega)\xi_1,\xi_2)=\cB(\xi_1,\Psi_\Gamma((\pi^{\frac{1-s}{2}}\circ\tau)(\omega))(\xi_2))~~
  \forall \omega\in \Omega(M) \,\,\,\,\,\, \forall \xi_1,\xi_2\in
  \Gamma(S)$.
\end{enumerate}

\begin{definition}
We say that $(M,g)$ is {\em strongly spin} if it admits a
$\Spin_0(V^{\ast},h^{\ast})$-structure  --- which we call a
{\em strong spin structure}. In this case, a real spinor bundle
$(S,\Gamma)$ on $(M,g)$ is called {\em strong} if it associated to a
strong spin structure.
\end{definition}

\noindent When $(M,g)$ is strongly spin, then it is {\em strongly
  orientable} in the sense that its orthonormal coframe bundle admits a
reduction to an $\SO_0(V^{\ast},h^{\ast})$-bundle. 

\begin{remark}
\label{rem:stronglyspin}
When $pq=0$, the special orthogonal and spin groups are connected
while the pin group has two connected components. In this case,
orientability and strong orientability are equivalent, as are the
properties of being spin and strongly spin. When $pq\neq 0$, the
groups $\SO(V^\ast,h^\ast)$ and $\Spin(V^\ast,h^\ast)$ have two
connected components, while $\Pin(V^\ast,h^\ast)$ has four and we have
$\Pin(V^\ast,h^\ast)/\Spin_0(V^\ast,h^\ast)\simeq \Z_2\times \Z_2$.
In this case, $(M,g)$ is strongly orientable iff it is orientable and
in addition the principal $\Z_2$-bundle associated to its bundle of
oriented coframes through the group morphism
$\SO(V^\ast,h^\ast)\rightarrow
\SO(V^\ast,h^\ast)/\SO_0(V^\ast,h^\ast)$ is trivial, while an
untwisted $\Pin(V^\ast,h^\ast)$-structure $\bQ$ reduces to a
$\Spin_0(V^\ast, h^\ast)$-structure iff the principal $\Z_2\times
\Z_2$-bundle associated to $\bQ$ through the group morphism
$\Pin(V^\ast,h^\ast)\rightarrow
\Pin(V^\ast,h^\ast)/\Spin_0(V^\ast,h^\ast)$ is trivial. When $(M,g)$
is strongly spin, the short exact sequence:
\be
1 \to \Z_2 \hookrightarrow \Spin_0(V^{\ast},h^{\ast}) \rightarrow \SO_0(V^{\ast},h^{\ast})\to 1
\ee
induces a sequence in Cech cohomology which implies that
$\Spin_0(V^\ast,h^\ast)$-structures form a torsor over
$H^1(M,\Z_2)$. A particularly simple case arises when $H^1(M,\Z_2)=0$
(for example, when $M$ is simply-connected). In this situation, $M$ is
strongly orientable and any untwisted $\Pin(V^\ast,h^\ast)$-structure
on $(M,g)$ reduces to a $\Spin_0(V^\ast,h^\ast)$-structure since
$H^1(M,\Z_2\times \Z_2)=H^1(M,\Z_2\oplus \Z_2)=0$. Similarly, any
$\Spin(V^\ast,h^\ast)$-structure on $(M,g)$ reduces to a
$\Spin_0(V^\ast,h^\ast)$-structure. Up to isomorphism, in this special
case there exists at most one $\Spin(V^\ast,h^\ast)$-structure, one
$\Spin_0(V^\ast,h^\ast)$-structure and one real spinor bundle on
$(M,g)$, which is automatically strong.
\end{remark}

\noindent The following gives sufficient conditions for existence of
admissible pairings on real spinor bundles:

\begin{prop}
\label{prop:SpinorialConnection}
Suppose that $(M,g)$ is strongly spin and Let $(S,\Gamma)$ be a strong
real spinor bundle on $(M,g)$. Then every admissible pairing on
$(\Sigma,\gamma)$ extends to an admissible pairing $\cB$ on
$(S,\Gamma)$. Moreover, the Levi-Civita connection $\nabla^g$ of
$(M,g)$ lifts to a unique connection on $S$ (denoted $\nabla^S$ and
called the {\em spinorial connection} of $S$), which acts by module
derivations:
\be
\nabla^S_X(\alpha\cdot \epsilon) = (\nabla^g_X\alpha)\cdot \epsilon +
\alpha\cdot (\nabla^S_X\epsilon) \,\,\,\,\,\, \forall \alpha\in
\Omega(M)~~\forall \epsilon \in \Gamma(S)~~\forall X\in \fX(M)
\ee
and is compatible with $\cB$:
\be
X[\cB(\epsilon_1,\epsilon_2)] = \cB(\nabla^S_X\epsilon_1,\epsilon_2) +
\cB(\epsilon_1,\nabla^S_X\epsilon_2) \,\,\,\,\,\, \forall \epsilon_1,
\epsilon_2 \in \Gamma(S)~~\forall X\in \fX(M)~~.
\ee
\end{prop}

\begin{proof}
The first statement follows from the associated bundle construction
since admissible pairings are $\Spin_0(V^\ast,h^\ast)$-invariant by
Proposition \ref{prop:cBinvar}. The second and third statements are
standard (see \cite[Chapter 3]{Friedrich}). The last statement follows
since the holonomy of $\nabla^S$ is contained in
$\Spin_0(V^{\ast},h^{\ast})$, whose action on $\Sigma$ preserves $\cB$.
\end{proof}

\noindent With the assumptions of the proposition, the spinorial connection
induces a linear connection (denoted $D^S$) on the bundle of
endomorphisms $End(S)=S^\ast \otimes S$.  By definition, we have:
\be
(D^S_X A)(\epsilon) = \nabla^S_X [A(\epsilon)] -
A(\nabla^S_X\epsilon) \,\,\,\,\,\, \forall A\in \Gamma(End(S))~~\forall
\epsilon \in \Gamma(S)~~\forall X\in \fX(M)~~.
\ee

\begin{prop}
Suppose that $(M,g)$ is strongly spin and let $(\Sigma,\Gamma)$ be a
strong real spinor bundle over $(M,g)$. Then $D^S\colon \Gamma(End(S))\to
\Gamma(T^{\ast}M\otimes End(S))$ acts by derivations:
\be
D^S_X(A_1\circ A_2) = D^S_X(A_1)\circ A_2 + A_1\circ D^S_X(A_2)
\,\,\,\,\,\, \forall\,\, A_1, A_2 \in \Gamma(End(S))~~\forall X\in
\fX(M)~~.
\ee
Moreover, $\Psi_{\Gamma}$ induces a unital isomorphism of
algebras $(\Omega(M),\diamond)\simeq (\Gamma(End(S)),\circ)$ which is
compatible with $\nabla^g$ and $D^S$:
\be
D^{S}_X(\Psi_{\Gamma}(\alpha)) = \Psi_{\Gamma}(\nabla^g_X\alpha)
\,\,\,\,\,\, \forall \alpha \in \Omega(M)~~\forall X\in \fX(M)~~.
\ee
\end{prop}

\begin{proof}
That $D^S$ acts by algebra derivations of $\Gamma(End(S))$ is
standard. Proposition \ref{prop:SpinorialConnection} gives:
\be
(D^S_X A)(\epsilon)= \nabla^S_X A(\epsilon) - A(\nabla^S_X\epsilon) =
\nabla^S_X(\Psi_{\Gamma}(\alpha)(\epsilon)) -
\Psi_{\Gamma}(\alpha)(\nabla^S_X\epsilon) =
\Psi_{\Gamma}(\nabla^g_X\alpha)(\epsilon)~~
\ee 
for all $A\in \Gamma(End(S))$, $\epsilon \in \Gamma(S)$ and $X\in
\fX(M)$, where $\alpha\eqdef \Psi_\Gamma^{-1}(A)\in \Omega(M)$.
\end{proof}

\begin{definition}
\label{def:dequantization}
Suppose that $(M,g)$ is strongly spin and let $(S,\Gamma)$ be a
strong real spinor bundle over $(M,g)$. Given a connection
$\cD:\Gamma(S)\rightarrow \Omega^1(M,S)$ on $S$, its {\em
  dequantization} is the connection ${\hat
  \cD}:\Gamma(\Lambda(M))\rightarrow \Omega^1(M,\Lambda(M))$ defined
on $\Lambda(M)$ through:
\be
{\hat \cD}_X\eqdef \Psi_\Gamma^{-1} \circ \cD_X\circ \Psi_\Gamma
\,\,\,\,\,\, \forall X\in \fX(M)~~.
\ee
\end{definition}

\begin{remark}
Writing $\cD=\nabla^S-\cA$ with $\cA\in\Omega^1(End(S))$, we
have:
\be
{\hat \cD}=\nabla^g-{\hat \cA}~~,
\ee
where ${\hat \cA}\in \Omega^1(M,\Lambda(M))$ is the symbol of $\cA$,
which we shall also call the {\em symbol of $\cD$}.
\end{remark}

\subsection{Constrained generalized Killing spinors}

\begin{definition}
\label{def:generalizedKS}
Let $(S,\Gamma)$ be a real spinor bundle on $(M,g)$ and $\cD$ be
an arbitrary connection on $S$. A section $\epsilon \in \Gamma(S)$ is
called {\em generalized Killing spinor with respect to $\cD$} if:
\ben
\label{GKSE}
\cD\epsilon = 0~~.
\een
A {\em linear constraint datum} for $(S,\Gamma)$ is a pair $(\cW,\cQ)$,
where $\cW$ is a real vector bundle over $M$ and $\cQ\in \Gamma(End(S)\otimes
\cW)\simeq \Gamma(Hom(S,S\otimes \cW))$. Given such a datum, the
condition:
\ben
\label{CSE}
\cQ(\epsilon)=0
\een
is called the {\em linear constraint} on $\epsilon$ defined by
$\cQ$. We say that $\epsilon$ is a (real) {\em constrained
  generalized Killing spinor} if it satisfies the system formed by
\eqref{GKSE} and \eqref{CSE}.
\end{definition}

\begin{remark}
Supersymmetric solutions of supergravity theories can often be
characterized as manifolds admitting certain systems of generalized
constrained Killing spinors, see for instance
\cite{LazaroiuB,LazaroiuBII}. This extends the notion of generalized
Killing spinor considered
\cite{BarGM,FriedrichKim,FriedrichKimII,MoroianuSemm}.
\end{remark}

\noindent
Suppose that $(M,g)$ is strongly spin and $(S,\Gamma)$ is a strong
real spinor bundle. Then we can write $\cD=\nabla^S-\cA$ with $\cA\in
\Omega^1(End(S))$, where $\nabla^S$ is the spinorial connection on
$S$. In this case, the equations satisfied by a constrained
generalized Killing spinor can be written as:
\be 
\nabla^S\epsilon = \cA \epsilon ~~ , ~~ \cQ(\epsilon) = 0
\ee
and their solutions are called constrained generalized Killing spinors
{\em relative to $(\cA,\cW,\cQ)$}. When $\cA$ is given, we sometimes
denote $\cD$ by $\cD_\cA$. Using connectedness of $M$ and the parallel
transport of $\cD$, equation \eqref{GKSE} implies that the space of
constrained generalized Killing spinors relative to $(\cA,\cQ,\cW)$ is
finite-dimensional and that a constrained generalized Killing spinor
which is not zero at some point of $M$ is automatically
nowhere-vanishing on $M$; in this case, we say that $\epsilon$ is {\em
  nontrivial}.

\subsection{Spinor squaring maps}

Let $\bS=(S,\Gamma,\cB)$ be a paired spinor bundle on $(M,g)$.
The admissible pairing $\cB$ of $(S,\Gamma)$ allows us to construct
extensions to $M$ of the squaring maps
$\cE_\pm:\Sigma\rightarrow \End(\Sigma)$ of Section
\ref{sec:vectorasendo} and of the spinor squaring maps
$\cE_\bSigma^\pm\colon \Sigma \to \wedge V^{\ast}$ of Section
\ref{sec:SpinorsAsPolyforms}. We denote these by:
\be
\cE_\pm : S\rightarrow End(S)  ~~ \mathrm{and} ~~ \cE^\pm_\bS : S\rightarrow \Lambda(M)~~.
\ee
Although $\cE_\bS^\pm$ preserve fibers, they are not morphisms of
vector bundles since they are fiberwise quadratic. By the results of
Section \ref{sec:SpinorsAsPolyforms}, these maps are two to one away
from the zero section of $S$ (where they branch) and their images --
which we denote by $Z^\pm(M)$ -- are subsets of the total space of
$\Lambda(M)$ which fiber over $M$ with cone fibers $Z^\pm_m(M)$ ($m\in
M$). We have $Z^-(M)=-Z^+(M)$ and $Z^+(M)\cap
Z^-(M)=0_{\Lambda(M)}$. The fiberwise sign action of $\Z_2$ on $S$
permutes the sheets of these covers (fixing the zero
section), hence $\cE_\bS^\pm$ give bijections from $S/\Z_2$ to
$Z^\pm(M)$ as well as a single bijection:
\be
\hcE_\bS:S/\Z_2\stackrel{\sim}{\rightarrow} Z(M)/\Z_2~~,
\ee
where $Z(M)\eqdef Z^+(M)\cup Z^-(M)$ and $\Z_2$ 
acts by sign multiplication. The sets $\dot{Z}^\pm(M)\eqdef Z^\pm(M)\setminus
0_{\Lambda(M)}$ are connected submanifolds of the total space
of $\Lambda(M)$ and the restrictions:
\ben
\label{dcE}
\dot{\cE}_\bS^\pm:\dot{S}\rightarrow \dot{Z}^\pm(M)
\een
of $\cE_\bS^\pm$ away from the zero section are surjective morphisms
of fiber bundles which are two to one.

\begin{definition}
The {\em signed spinor squaring maps} of the paired spinor bundle
$\bS=(S,\Gamma,\cB)$ are the maps $\cE^\pm_\bS\colon \Gamma(S)\to
\Omega(M)$ induced by $\cE_\bS^\pm$ on sections (which we denote by
the same symbols).
\end{definition}

\noindent By the results of Section \ref{sec:SpinorsAsPolyforms},
$\cE_\bS^\pm$ are quadratic maps of $\cC^\infty(M)$-modules and
satisfy:
\be
\supp(\cE_\bS^\pm(\epsilon))=\supp(\epsilon) \,\,\,\,\,\, \forall \epsilon\in \Gamma(S)~~.
\ee
Let $\fZ^\pm(M)\eqdef \cE_\bS^\pm(\Gamma(S))\subset \Omega(M)$ denote
their images and set $\fZ(M)\eqdef \fZ^+(M)\cup
\fZ^-(M)$. Then $\fZ^-(M)=-\fZ^+(M)$ and $\fZ^+(M)\cap
\fZ^-(M)=\{0\}$ and we have strict inclusions $\fZ^\pm(M)\subset
\Gamma(Z^\pm(M))$ and $\fZ(M)\subset \Gamma(Z(M))$ (see equation
\eqref{GammaA} for notation). Moreover, $\cE_\bS^\pm$ induce 
the same bijection:
\be
\hcE_\bS:\Gamma(S)/\Z_2\stackrel{\sim}{\rightarrow} \fZ(M)/\Z_2~~.
\ee
Finally, let $\dGamma(S)=\Gamma(\dot{S})$ be the set of
nowhere-vanishing sections of $S$ and $\dfZ^\pm(M)\eqdef
\dGamma(Z^\pm(M))=\Gamma(\dfZ^\pm(M))\subset \fZ^\pm(M)$ be the set of
those polyforms in $\fZ^\pm(M)$ which are nowhere-vanishing and define
$\dfZ(M)\eqdef \dfZ^+(M)\cup \dfZ^-(M)$. Notice that $\dfZ^+(M)\cap
\dfZ^-(M)=\emptyset$. The signed spinor squaring maps restrict to
two-to one surjections which coincide with the maps induced by
\eqref{dcE} on sections:
\be
\dot{\cE}_\bS^\pm:\dGamma(S)\rightarrow \dfZ^\pm(M)~~.
\ee

\begin{prop}
\label{eq:obstructionliftpoly}
Suppose that $(M,g)$ is strongly spin let $\bS=(S,\Gamma,\cB)$ be a
strong paired spinor bundle associated to a
$\Spin_0(V^{\ast},h^{\ast})$-structure $\bQ$ on $(M,g)$. Then every
nowhere-vanishing polyform $\alpha\in \dfZ(M)$ determines a cohomology
class $c_\bQ(\alpha)\in H^1(M,\Z_2)$ encoding the obstruction to
existence of a globally-defined spinor $\epsilon\in \Gamma(S)$ (which
is necessarily nowhere-vanishing) such that $\alpha\in
\{\cE_\bS^+(\epsilon),\cE_\bS^-(\epsilon)\}$. More precisely, such
$\epsilon$ exists iff $c_\bQ(\alpha) = 0$. In particular, we have:
\be
\dfZ(M)=\{\alpha\in \fZ(M) \,|\,
c_\bQ(\alpha)=0\}~~ \mathrm{and} ~~ \dfZ^\pm(M)=\{\alpha\in \fZ^\pm(M)
\,|\, c_\bQ(\alpha)=0\}~~.
\ee
\end{prop}

\begin{proof}
We have $\alpha\in \fZ^\kappa(M)$ for some $\kappa\in \{-1,1\}$. Let
$L_\alpha$ be the real line sub-bundle of $\Lambda(M)$ determined by
$\alpha$. Since the projective spinor squaring map
$\P\cE_\bS:\P(S)\rightarrow\P(\wedge(M))$ is bijective, $L_\alpha$
determines a real line sub-bundle $L_\bQ(\alpha)\eqdef
(\P\cE_\bS)^{-1}(L_\alpha)$ of $S$.
A section $\epsilon$ of $S$ such that
$\cE^\kappa_\bS(\epsilon)=\alpha$ is a section of
$L_\bQ(\alpha)$. Since such $\epsilon$ must be nowhere-vanishing
(because $\alpha$ is), it exists iff $L_\bQ(\alpha)$ is trivial, which
happens iff its first Stiefel-Whitney class vanishes. The conclusion
follows by setting $c_\bQ(\alpha) \eqdef w_1(L_\bQ(\alpha))\in
H^1(M,\Z_2)$. Notice that $c_\bQ(\alpha)$ depends only on $\alpha$ and
$\bQ$, since the Clifford bundle $(S,\Gamma)$ is associated to $\bQ$
while all admissible pairings of $(S,\Gamma)$ are related to each
other by automorphisms of $S$ (see Remark \ref{rem:cBrelation} in
Section \ref{sec:SpinorsAsPolyforms}).
\end{proof}

\begin{definition}
\label{def:spinorclass}
The cohomology class $c_\bQ(\alpha)\in H^1(M,\Z_2)$ of the previous
proposition is called the {\em spinor class} of the nowhere-vanishing
polyform $\alpha\in \dfZ(M)$.
\end{definition}

\begin{remark} 
$c_\bQ(\alpha)$ is not a characteristic class of $S$,
  since it depends on $\alpha$. 
\end{remark}

\begin{lemma}
\label{lemma:L}
Let $\bS=(S,\Gamma,\cB)$ be a paired real spinor bundle on $(M,g)$,
$(S_L,\Gamma_L)$ be the modification of $(S,\Gamma)$ by a real line
bundle $L$ on $M$ and $q:L^{\otimes 2}\simeq \R_M$ be an isomorphism
of line bundles. Let $\cB_L$ be the bilinear
non-degenerate pairing on $S_L$ whose duality isomorphism
$\ast_L:S_L\rightarrow S_L^\ast$ satisfies:
\ben
\label{Bast}
\ast_L\otimes \Id_{S_L}=\varphi_L^{-1} \circ (\ast\otimes \Id_S)\circ \psi_q~~,
\een
where $\ast:S\rightarrow S^\ast$ is the duality isomorphism of
$\cB$, $\varphi_L:End(S_L)\rightarrow End(S)$ is the natural
isomorphism of bundles of unital algebras and $\psi_q\eqdef
\Id_{S\otimes S}\otimes q:S_L\otimes S_L\rightarrow S\otimes S$ is the
isomorphism of vector bundles induced by $q$. Then $\cB_L$ is an
admissible pairing on $(S_L,\Gamma_L)$ which has the same symmetry and
adjoint type as $\cB$. Hence the triplet $\bS_L\eqdef
(S_L,\Gamma_L,\cB_L)$ is a paired spinor bundle on $(M,g)$ which we
call the {\em modification of $\bS$ by $L$}.
\end{lemma}

\begin{proof}
Recall from Proposition \ref{prop:Ltensor} that
$\Gamma_L=\varphi_L^{-1}\circ \Gamma$. A simple computation gives:
\be
\cB_L(\xi_1\otimes l_1,\xi_2\otimes l_2)=q(l_1\otimes
l_2)\cB(\xi_1,\xi_2)\,\,\,\,\, \forall \xi_1,\xi_2\in
\Gamma(S)~~\forall l_1,l_2\in \Gamma(L)~~,
\ee
which immediately implies the conclusion.
\end{proof}

\noindent The following proposition shows that $c_\bQ(\alpha)$ can be
made to vanish by changing $\bQ$.

\begin{prop}
\label{eq:anotherspin}
Suppose that $(M,g)$ is strongly spin and a let $\bQ$ be a
$\Spin_0(V^{\ast},h^{\ast})$-structure on $(M,g)$. For every
nowhere-vanishing polyform $\alpha\in \dfZ(M)$, there exists a unique
$\Spin_0(V^{\ast},h^{\ast})$-structure $\bQ^{\prime}$ such that
$c_{\bQ^{\prime}}(\alpha) = 0$.
\end{prop}

\begin{proof}
Suppose for definiteness that $\alpha\in \fZ^+(M)$. Let $(S,\Gamma)$
be the strong real spinor bundle associated to $\bQ$ and set $L\eqdef
L_\bQ^+(\alpha)\subset S$. By Remark \ref{rem:stronglyspin},
isomorphism classes of $\Spin_0(V^{\ast},h^{\ast})$-structures on
$(M,g)$ form a torsor over $H^1(M,\Z_2)$. Let $\bQ^{\prime} =
c_\bQ(\alpha)\cdot \bQ$ be the spin structure obtained from $\bQ$ by
acting in this torsor with $c_\bQ(\alpha)$. Then the strong real
spinor bundle associated to $\bQ^{\prime}$ coincides with
$(S_L,\Gamma_L)$.  Pick an isomorphism $q:L^{\otimes 2}\simeq \R_M$
and equip $S_L$ with the admissible pairing $\cB_L$ induced from $\cB$
by $q$ as in Lemma \ref{lemma:L}. Since
$\Psi_{\Gamma_L}=\varphi_L^{-1}\circ \Psi_\Gamma$, relation
\eqref{Bast} implies that the polarizations
$\boldsymbol{\cE}^+_{\bS_L}=\Psi_{\Gamma_L}^{-1}\circ (\ast_L\otimes
\Id_{S_L})$ and $\boldsymbol{\cE}^+_{\bS}=\Psi_{\Gamma}^{-1}\circ
(\ast\otimes \Id_S)$ of the positive spinor squaring maps of $\bS_L$
and $\bS$ are related through:
\be
\boldsymbol{\cE}^+_{\bS_L}=\boldsymbol{\cE}^+_{\bS} \circ \psi_q~~.
\ee
Since $\psi_q(L^{\otimes 2}\otimes L^{\otimes 2})=L^{\otimes 2}$
(where $L^{\otimes 2}$ is viewed as a sub-bundle of $S_L=S\otimes L$),
this gives $\boldsymbol{\cE}^+_{\bS_L}(L^{\otimes 2}\otimes L^{\otimes
  2})=\boldsymbol{\cE}^+_{\bS}(L\otimes L)$, which implies
$\cE^+_{\bS_L}(L^{\otimes 2})=\cE^+_{\bS}(L)=L_\alpha$ Hence the line
sub-bundle of $S_L$ determined by $\alpha$ is the trivializable real
line bundle $L^{\otimes 2}\simeq \R_M$. Thus $c_{\bQ'}(\alpha) = 0$.
\end{proof}

\subsection{Description of constrained generalized Killing spinors as polyforms}

Let $\bS=(S,\Gamma,\cB)$ be a paired spinor bundle (with $\cB$ of
adjoint type $s$) and $(\cW,\cQ)$ be a constraint datum for
$(S,\Gamma)$. Let ${\hat \cQ}\eqdef (\Psi_{\Gamma}\otimes \Id_\cW)\in \Omega^\ast(M,\cW)$ 
be the
symbol of $\cQ$ (see Definition \ref{def:symbol}). Proposition
\ref{prop:constraintendopoly} implies:

\begin{lemma}
A spinor $\epsilon\in \Gamma(S)$ satisfies:
\be
\cQ(\epsilon) = 0
\ee
iff one (and hence both) of the following mutually-equivalent
relations holds:
\be
\hat{\cQ}\diamond \alpha = 0 ~~ , ~~ \alpha \diamond (\pi^{\frac{1-s}{2}}\circ\tau)(\hat{Q}) = 0 ~~ ,
\ee
where $\alpha\eqdef \cE_\bS^+(\epsilon)\in \Omega(M)$ is
the positive polyform square of $\epsilon$.
\end{lemma}

\noindent Now assume that $(M,g)$ is strongly spin and that
$(S,\Gamma,\cB)$ is the paired spinor bundle associated to a
$\Spin_0(V^\ast,h^\ast)$-structure. Set $\cA\eqdef \nabla^S-\cD\in
\Omega^1(M,End(S))$ and let $\hat{\cA} \eqdef (\Psi_{\Gamma}\otimes
\Id_{T^\ast M})^{-1}(\cA)\in \Omega^1(M, \Lambda(M))$ be the symbol of
$\cA$, viewed as a $\Lambda(M)$-valued one-form. In this case, we
have:

\begin{lemma}
\label{lemma:GKSiff}
A nowhere-vanishing spinor $\epsilon \in \Gamma(S)$ satisfies
$\cD\epsilon = 0$ iff:
\ben
\label{eq:GKSI}
\nabla^g\alpha = \hat{\cA}\diamond \alpha + 
\alpha \diamond (\pi^{\frac{1-s}{2}}\circ\tau)(\hat{\cA})~~,
\een
where $\alpha \eqdef \cE_\bS^+(\epsilon)$ is the positive polyform
square of $\epsilon$.
\end{lemma}

\begin{proof}
Assume that $\epsilon$ satisfies $\nabla^S\epsilon = \cA(\epsilon)$.
We have $\alpha \in \Gamma(End(S))$ and:
\begin{eqnarray*}
D^S(\cE_\bS^+(\epsilon))(\chi) = \nabla^S(\cE_\bS^+(\epsilon))(\chi) - \cE_\bS^+(\epsilon)(\nabla^S\chi) = 
\nabla^S(\cB(\chi,\epsilon)\,\epsilon) - \cB(\nabla^S\chi,\epsilon)\,\epsilon \\
= \cB(\chi,\nabla^S\epsilon)\,\epsilon + \cB(\chi,\epsilon)\,\nabla^S\epsilon = 
\cB(\chi,\cA\,\epsilon)\,\epsilon + \cB(\chi,\epsilon)\,\cA\,\epsilon = 
\cE_\bS^+(\epsilon)(\cA^t\,\chi) + \cA(\cE_\bS^+(\epsilon))(\chi)
\end{eqnarray*}
for all $\chi \in \Gamma(S)$, where $\cA^t$ is obtained by fiberwise
application of the $\cB$-transpose of Lemma
\ref{lemma:adjointpoly}. The equation above implies:
\ben
\label{eq:DEepsilon}
D^S(\cE_\bS^+(\epsilon)) =  \cA\circ \cE_\bS^+(\epsilon) + \cE_\bS^+(\epsilon)\circ\cA^t~~.
\een
Applying $\Psi^{-1}_{\Gamma}$ and using Lemma \ref{lemma:adjointpoly} and
Proposition \ref{prop:LCderClifford} gives \eqref{eq:GKSI}.

Conversely, assume that $\alpha$ satisfies \eqref{eq:GKSI}. Applying
$\Psi_{\Gamma}$ gives equation \eqref{eq:DEepsilon}, which reads:
\ben
\label{eq:Depsilonrelation}
\cB(\chi,\cD_{X}\epsilon)\,\epsilon + \cB(\chi,\epsilon)\,\cD_{X}\epsilon = 0 \,\,\,\,\,\, \forall \chi \in \Gamma(S)~~\forall X\in \fX(M)~~.
\een
Hence $\cD_{X}\epsilon = \beta(X) \epsilon$ for some $\beta
\in \Omega^{1}(M)$. Using this in \eqref{eq:Depsilonrelation} gives:
\be
\cB(\chi,\epsilon)\,\beta\otimes \epsilon = 0 \,\,\,\,\,\, \forall \chi \in \Gamma(S)~~.
\ee
This implies $\beta=0$, since $\cB$ is non-degenerate and $\epsilon$
is nowhere-vanishing. Hence $\cD\epsilon = 0$.
\end{proof}

\begin{remark}
If $\cA$ is skew-symmetric with respect to $\cB$, then \eqref{eq:GKSI}
simplifies to:
\ben
\label{eq:GKSII}
\nabla^g\alpha = \hat{\cA}\diamond \alpha - 
\alpha\diamond \hat{\cA}~~.
\een
In applications to supergravity, $\cA$ need {\em not} be
skew-symmetric relative to $\cB$.
\end{remark}

\begin{thm}
\label{thm:GCKS}
Suppose that $(M,g)$ is strongly spin and let $\bS=(S,\Gamma,\cB)$ be
a paired spinor bundle associated to the
$\Spin_0(V^{\ast},h^{\ast})$-structure $\bQ$ and whose admissible form
$\cB$ has adjoint type $s$. Let $\cA\in \Omega^1(M,End(S))$ and 
$(\cW,Q)$ be a linear constraint datum for $(S,\Gamma)$. Then the
following statements are equivalent:
\begin{enumerate}[(a)]
\itemsep 0.0em
\item There exists a nontrivial generalized constrained Killing spinor
  $\epsilon\in \Gamma(S)$ with respect to $(\cA,\cW,\cQ)$.
\item There exists a nowhere-vanishing polyform $\alpha\in \Omega(M)$
  with vanishing cohomology class $c_\bQ(\alpha)$ which satisfies
  the following algebraic and differential equations for every
  polyform $\beta \in\Omega(M)$:
\ben
\label{eq:Fierzglobal}
\alpha\diamond \beta\diamond\alpha = \cS(\alpha\diamond\beta)\, 
\alpha~~,~~ (\pi^{\frac{1-s}{2}}\circ\tau)(\alpha) = \sigma_s\,\alpha~~,  
\een
\ben
\label{eq:GKSeqiff}
\nabla^g\alpha = \hat{\cA}\diamond \alpha + \alpha\diamond
 (\pi^{\frac{1-s}{2}}\circ\tau)(\hat{\cA})~~,~~  \hat{\cQ}\diamond \alpha  = 0~~
\een
or, equivalently, satisfies the equations:
\ben
\alpha\diamond\alpha = \cS(\alpha) \, \alpha~~,
~~(\pi^{\frac{1-s}{2}}\circ\tau)(\alpha) = \sigma_s\,\alpha~~,
~~\alpha\diamond \beta\diamond\alpha =
\cS(\alpha\diamond\beta)\, \alpha~~,
\een
\ben
\nabla^g\alpha = \hat{\cA}\diamond \alpha + \alpha\diamond
(\pi^{\frac{1-s}{2}}\circ\tau)(\hat{\cA})\, ~~,~~
\hat{\cQ}\diamond \alpha = 0~~,
\een
for some fixed polyform $\beta \in\Omega(M)$ such that
$\cS(\alpha\diamond\beta) \neq 0$.
\end{enumerate}
If $\epsilon\in \Gamma(S)$ is chiral of chirality
$\mu\in\left\{-1,1\right\}$, then we have to add the condition:
\be
\ast \, (\pi\circ\tau)(\alpha) = \mu \, \alpha~~.
\ee
The polyform $\alpha$ as above is determined by $\epsilon$ through the
relation:
\be
\alpha=\cE_\bS^\kappa(\epsilon)
\ee
for some $\kappa\in \{-1,1\}$. Moreover, $\alpha$ satisfying the
conditions above determines a nowhere-vanishing real spinor $\epsilon$
satisfying this relation, which is unique up to sign.
\end{thm}

\begin{remark}
\label{remark:obstructionliftingspinor} 
Suppose that $\alpha\in \Omega(M)$ is nowhere-vanishing and satisfies
\eqref{eq:Fierzglobal} and \eqref{eq:GKSeqiff} but we have
$c_\bQ(\alpha)\neq 0$. Then Proposition \ref{eq:anotherspin} implies
that there exists a unique $\Spin_0(V^\ast,h^\ast)$-structure
$\bQ^{\prime}$ such that $c_{\bQ^{\prime}}(\alpha) = 0$. Thus $\alpha$
is the square of a global section of a paired spinor bundle
$(S',\Gamma',\cB')$ associated to $\bQ'$. Hence a nowhere-vanishing
polyform $\alpha$ satisfying \eqref{eq:Fierzglobal} and \eqref{eq:GKSeqiff}
corresponds to the square of a generalized Killing spinor with respect
to a uniquely-determined $\Spin_0(V^\ast,h^\ast)$-structure.
\end{remark}

\begin{proof} 
The algebraic conditions in the Theorem follow from the pointwise
extension of Theorem \ref{thm:reconstruction} and Corollary
\ref{cor:reconstructionchiral}. The differential condition follows
from Lemma \ref{lemma:GKSiff}, which implies that
$\cD_{\cA}\epsilon = 0$ holds iff \eqref{eq:GKSeqiff} does
upon noticing that $\epsilon \in \Gamma(S)$ vanishes at a
point $m\in M$ iff its positive polyform square $\alpha$ satisfies
$\alpha\vert_m = 0$. The condition $c_\bQ(\alpha) = 0$ follows from
Proposition \ref{eq:obstructionliftpoly}.
\end{proof}

\noindent In Sections \ref{sec:RKSpinors} and \ref{sec:Susyheterotic},
we apply this theorem to real Killing spinors on Lorentzian
four-manifolds and to supersymmetric configurations of heterotic
supergravity on principal bundles over such manifolds.


\subsection{Real spinors on Lorentzian four-manifolds}
\label{sec:GCKLorentz4d}


Let $(M,g)$ be a spin Lorentzian four-manifold of ``mostly plus''
signature such that $H^1(M,\Z_2)=0$. By Remark \ref{rem:stronglyspin},
this condition insures that $(M,g)$ is strongly spin, with a
$\Spin(V^\ast,h^\ast)$-structure, $\Spin_0(V^\ast,h^\ast)$-structure
and real spinor bundle $(S,\Gamma)$ which are unique up to
isomorphism. Let $\cS=(S,\Gamma,\cB)$, where the admissible pairing
$\cB$ is skew-symmetric and of negative adjoint type.  The spinor
class $c_\bQ(\alpha)$ (see Definition \ref{def:spinorclass}) of any
nowhere-vanishing polyform $\alpha\in \dfZ(M)$ vanishes since
$H^1(M,\Z_2)=0$. Hence any such $\alpha$ is a signed square of a
nowhere-vanishing spinor. We will characterize real spinors on $(M,g)$
through certain pairs of one-forms.

\begin{definition}
A pair of nowhere-vanishing one-forms $(u,l)\in
\Omega^1(M)\times \Omega^1(M)$ is called {\em parabolic} if:
\be
g^\ast(u,u)=0~~,~~g^\ast(l,l)=1~~,~~g^\ast(u,l)=0~~,
\ee
i.e. $u$ and $l$ are mutually-orthogonal, with $u$ timelike and $l$
spacelike of unit norm. Two parabolic pairs of one-forms $(u,l)$ and
$(u',l')$ are called {\em strongly equivalent} if there exists a sign
factor $\zeta\in \{-1,1\}$ and a real constant $c\in \R$ such that:
\be
u'=\zeta u~~\mathrm{and}~~l'=l+cu~~.
\ee
\end{definition}

\noindent Let $\cP(M,g)$ denote the set of parabolic pairs of one-forms defined
on $(M,g)$.

\begin{definition}
A rank two vector sub-bundle $\Pi$ of $T^\ast M$ is called a {\em
  distribution of parabolic 2-planes} in $T^\ast M$ if, for all $m\in
M$, the fiber $\Pi_m$ is a parabolic 2-plane in the Minkowski space
$(T_m^\ast M, g_m^\ast)$.
\end{definition}

\begin{definition}
Let $\Pi$ is a distribution of parabolic 2-planes in $T^\ast M$. The
real line sub-bundle $\K_h(\Pi)\eqdef \ker(g^\ast_\Pi)$ (where
$g^\ast_\Pi$ is the restriction of $g^\ast$ to $\Pi$) is called
the {\em null line sub-bundle} of $\Pi$ and $\Pi$ is called {\em
  co-orientable} if the quotient line bundle $N_h(\Pi)\eqdef
\Pi/\K_h(\Pi)$ is trivializable. In this case, a {\em co-orientation}
of $\Pi$ is an orientation of $\Pi/\K_h(\Pi)$.
\end{definition}

\noindent A co-orientation of $\Pi$ amounts to the choice of a
sub-bundle of half-planes $\cH\subset \Pi$ such that $\cH_m$ is one of
the two spacelike half-planes of $\Pi_m$ for each $m\in M$. In this
case, the pair $(\Pi,\cH)$ is called a {\em co-oriented} distribution
of parabolic 2-planes in $T^\ast M$.

\begin{definition}
Let $\Pi$ be a distribution of parabolic 2-planes in $T^\ast M$. A
local frame $(u,l)$ of $\Pi$ defined on a non-empty open subset
$U\subset V$ is called a {\em local parabolic frame} if $(u,l)$ is a
parabolic pair of one-forms for the Lorenzian manifold
$(U,g|_U)$. Such a frame is called {\em global} if $U=M$.
\end{definition}

\noindent Local parabolic frames of $\Pi$ defined above $U$ are
determined up to transformations of the form:
\be
u'=b u~~\mathrm{and}~~l'=\zeta l+cu~~,
\ee
where $\zeta\in \{-1,1\}$ and $b,c$ are nowhere-vanishing smooth {\em
  functions} defined on $U$. Notice that $\Pi$ admits a global
parabolic frame iff it is trivializable. Since $H^1(M,\Z_2)=0$, any
smooth section of the projective bundle $\P(S)$ lifts to a
nowhere-vanishing section of $S$ and hence we have:
\be
\Gamma(\P(S))=\Gamma(\dot{S})/\cC^\infty(M)^\times~~,
\ee
where the multiplicative group $\cC^\infty(M)^\times$ of
nowhere-vanishing real-valued functions defined on $M$ acts on
$\Gamma(S)$ through multiplication of sections by the corresponding
function. The results of Section
\ref{sec:4dLorentzexample} imply:

\begin{thm}
\label{thm:flagLorentz4d}
There exists a natural bijection between the set
$\Gamma(\P(S))=\Gamma(\dot{S})/\cC^\infty(M)^\times$ and the set of
trivializable and co-oriented distributions $(\Pi,\cH)$ of parabolic
2-planes in $T^\ast M$. Moreover, there exist natural bijections
between the following two sets:
\begin{enumerate}[(a)]
\itemsep 0.0 em
\item The set $\Gamma(\dot{S})/\Z_2$ of sign-equivalence classes of
  nowhere-vanishing real spinors $\epsilon\in \Gamma(S)$.
\item The set of strong equivalence classes of parabolic pairs of
  one-forms $(u,l)\in \cP(M,g)$.
\end{enumerate}
\end{thm}

\begin{remark}
Let $(u,l)$ be a parabolic pair of one-forms corresponding to a nowhere-vanishing spinor 
$\xi\in
\Gamma(S)$. Then $\alpha\eqdef u+u\wedge l$ is a signed polyform
square of $\xi$ by Section \ref{sec:4dLorentzexample}.
\end{remark}

\subsection{Real spinors on globally hyperbolic Lorentzian four-manifolds}

Let $(M,g)$ be an oriented and spin Lorentzian four-manifold of
``mostly plus'' signature such that $H^1(M,\Z_2)=0$. As before, let
$\bS=(S,\Gamma,\cB)$ be a paired real spinor bundle on $(M,g)$, where
$\cB$ is skew-symmetric and of negative adjoint type.

\begin{prop}
\label{prop:timeoriented}
Suppose that $(M,g)$ is time-orientable and let $v\in \Omega^1(M)$ be
a timelike one-form such that $g^\ast(v,v)=-1$.  Let $P_v:T^\ast
M\rightarrow L_v$ be the orthogonal projection onto the real line
sub-bundle $L_v$ of $T^{\ast}M$ determined by $v$. For any parabolic
pair $(u,l)$ on $(M,g)$, there exists a unique smooth function $f\in
C^{\infty}(M)$ such that:
\ben
\label{gauge}
P_v(l+ f u) = 0~~.
\een
Moreover, there exist exactly two parabolic pairs of one-forms
$(u',l')$ which are strongly-equivalent with $(u,l)$ and satisfy
$P_v(l')=0$, namely:
\be
u'=u~~,~~l'=l+f u~~\mathrm{and}~~u'=-u~~,~l'=l+f u~~
\ee
and every pair $(u'',l'')$ which is equivalent with $(u,l)$ and
satisfies $P_v(l'')=0$ has the form:
\be
u''=b u~~,~~l''=l+ f u
\ee
where $b\in \cC^\infty(M)^\times$ is a nowhere-vanishing smooth function.
\end{prop}

\begin{proof}
We have:
\be
P_v(\alpha)=-g^{\ast}(\alpha,v) v~~,~~\forall \alpha\in \Omega^1(M)~~.
\ee
Notice that $g^\ast(u,v)\neq 0$ since $u\neq 0$ is lightlike and $v$
is timelike.  Condition \eqref{gauge} amounts to:
\be
g^\ast(l,v)+f g^\ast(u,v)=0~~,
\ee
which is solved by $f=-\frac{g^\ast(l,v)}{g^\ast(u,v)}$. The remaining
statements follow immediately from the definition of equivalence and
strong equivalence of parabolic pairs (see Definition
\ref{def:parabolic}).
\end{proof}

\noindent Let $\cP_v(M,g)$ denote the set of parabolic pairs of
one-forms $(u,l)$ on $(M,g)$ which satisfy $P_v(l)=0$. The group
$\Z_2$ acts on this set by changing the sign of $u$ while leaving $l$
unchanged. Proposition \ref{prop:timeoriented} and Theorem
\ref{thm:flagLorentz4d} imply:

\begin{cor}
\label{cor:toriented}
With the assumptions of the previous proposition, there exists a
bijection between the sets $\Gamma(\dot{S})/\Z_2$ and $\cP_v(M,g)/\Z_2$.
\end{cor}

\noindent Assume next that $(M,g)$ is globally hyperbolic. By a
theorem of A. Bernal and M. S\'anchez \cite{Bernal:2003jb}, it follows
that $(M,g)$ is isometric to $\R\times N$ equipped with the warped
product metric $g = -F\,\dd t\otimes \dd t + k(t)$, where $N$ is an
oriented three-manifold, $F\in C^{\infty}(\R\times N)$ is a strictly
positive function and $k(t)$ is a Riemannian metric on $N$ for every
$t\in \R$. Let $V_2(N,k(t))$ be the bundle of
ordered orthonormal pairs of one-forms on $N$. Since $N$ is
oriented, any element of $V_2(N,k(t))$ determines an element of the
principal bundle $P_{\SO(3)}(N,k(t))$ of oriented frames of $(N,k(t))$,
showing that $V_2(N,k(t))$ is a principal $\SO(3)$-bundle.
Let $\cV_2$ be the fiber bundle defined on $M=\R\times N$ whose fiber
at $(t,n)\in \R\times M$ is given by:
\be
\cV_2(t,n)\eqdef V_2(T_n^\ast N, k(t)_n)~~,
\ee
where $V_2(N,k(t))_n=V_2(T_n^\ast N, k(t)_n)$ is the manifold of
$k(t)_n$-orthonormal systems of two elements of $T_n^\ast N$. Consider 
the fiberwise involution $i_1$ of $\cV_2$ defined through:
\be
i_1(e_1,e_2)\eqdef (-e_1,e_2)\,\,\,\,\,\, \forall (e_1,e_2)\in \cV_2~~.
\ee
and let $\Z_2$ act on the set $\cC^\infty(R\times N)^\times\times \Gamma(\cV_2)$
through the involution:
\be
(\f,\fs)\rightarrow (-\f,i_1(\fs))\,\,\,\,\, \forall \f\in
\cC^\infty(R\times N)^\times~~\forall \fs\in \Gamma(\cV_2)~~.
\ee

\begin{prop}
\label{prop:globhyp}
Consider a globally hyperbolic Lorentzian four-manifold:
\be
(M,g) = (\R\times N, -F\, \dd t\otimes \dd t \oplus k(t))
\ee
such that $N$ is oriented and spin and $H^1(N,\Z_2)=0$.  Then there
exists a bijection between the set $\Gamma(\dot{S})/\Z_2$ of
sign-equivalence classes of nowhere-vanishing real spinors defined on
$M$ and the set $[\cC^\infty(\R\times N)^\times \times
  \Gamma(\cV_2)]/\Z_2$. Moreover, there exists a bijection between
the sets $\Gamma(\P(S))$ and $\Gamma(\cV_2)/\Z_2$, where $\Z_2$ acts
on $\Gamma(\cV_2)$ through the involution $i_1$.
\end{prop}

\begin{proof}
Let $v \eqdef F^{\frac{1}{2}} \dd t$ and consider a parabolic pair of
one-forms $(u,l)$ on $M$ such that $P_v(l) = 0$. Since $l$ has unit
norm and is orthogonal to $v$, it can be viewed as a family of
one-forms (parameterized by $t\in \R$) defined on $N$.  We decompose
$u$ orthogonally as:
\be
u = P_v(u)+u^\perp = - g^{\ast}(u,v)\, v \oplus u^{\perp}~~,
\ee
where $u^{\perp}\eqdef u-P_v(u)=u + g^{\ast}(u,v)\, v$ and
$g^{\ast}(u,v)$, $g^{\ast}(u^{\perp},u^{\perp})$ are
nowhere-vanishing. The spacelike 1-form $u^\perp$ satisfies:
\be
g^{\ast}(u^{\perp},u^{\perp})=g^{\ast}(u,v)^2~~.
\ee
Thus $u$ can be written as:
\be
u = \f\, v \oplus \vert \f\vert\, e_u~~, 
\ee
where:
\be
e_u \eqdef
\frac{u^{\perp}}{g^{\ast}(u^{\perp},u^{\perp})^{\frac{1}{2}}}=\frac{u^\perp}{|g^\ast(u,v)|}~~~~\mathrm{and}~~~~\f\eqdef
-g^\ast(u,v)\in \cC^\infty(\R\times N)~~.
\ee
The pair $(e_{u},l)$ determines an orthonormal pair of one-forms
$(e_u(t), l(t))$ on $(N,k(t))$ for all $t\in \R$, which gives a
section $\fs$ of the fiber bundle $\cV_2$. It is clear that the
parabolic pair $(u,l)$ determines and is determined by the pair
$(\f,\fs)$. The conclusion now follows from Corollary
\ref{cor:toriented} by noticing that the transformation $u\rightarrow
b u$ (with $b\in \cC^\infty(M,\R)$) corresponds to $f\rightarrow b\,\f$
and $e_u\rightarrow \sign(b) e_u$.
\end{proof}

\noindent The three-manifold $N$ is parallelizable since it is
connected, oriented and spin. This holds for both compact and open
$M$ by considering the Whitehead tower of $\B\O(3)$ (the classifying
space of $\O(3)$) and using the fact that $\pi_3 (\B\Spin(3)) =
0$. Hence there exists a unique family $\left\{ e(t)\right\}_{t\in\R}$
of one-forms on $N$ such that $(e_u(t), l(t), e(t))$ is
an oriented orthonormal global frame of $(T^\ast N, k_t)$ for all
$t\in\R$. This produces a parallelization of $(M,g)$ given by $(v,
e_u(t), l(t), e(t))_{t\in \R}$. Let:
\be
R\eqdef \diag(-1,0,0) \in \SO(3)~~.
\ee
and let $\Z_2$ act on $\cC^\infty(\R\times N)^\times\times
\cC^\infty(N,\SO(3))$ through the involution:
\be
(\f,\psi)\rightarrow (-\f,\Ad_R\circ \psi) \,\,\,\,\, \forall \f\in
\cC^\infty(\R\times N)^\times~~\forall \psi\in \cC^\infty(N,\SO(3))~~,
\ee
where:
\be
(\Ad_R\circ \psi)(t,n)=R\circ \psi(t,n)\circ R^{-1} \,\,\,\,\, \forall (t,n)\in
\R\times N~~.
\ee
The previous proposition implies:

\begin{cor}
Let $(M,g)$ be as in Proposition \ref{prop:globhyp} and fix a global
oriented orthonormal frame of $T^\ast N$. Then there exists a
bijection between the set $\Gamma(\dot{S})/\Z_2$ of sign equivalence
classes of nowhere-vanishing real spinors defined on $M$ and the set
$[\cC^\infty(\R\times N)^\times\times
  \cC^\infty(N,\SO(3))]/\Z_2$. Moreover, there exists a bijection
between the sets $\Gamma(\P(S))$ and $\cC^\infty(N,\SO(3))/\Z_2$, where
$\Z_2$ acts on $\cC^\infty(N,\SO(3))$ through the involution:
\be
\psi\rightarrow \Ad_R\circ \psi \,\,\,\,\, \forall \psi\in \cC^\infty(N,\SO(3))~~.
\ee
\end{cor}

\noindent
We hope that this characterization can be useful in the
study of globally hyperbolic Lorentzian four-manifolds admitting
spinors satisfying various partial differential equations.


\section{Real Killing spinors on Lorentzian four-manifolds}
\label{sec:RKSpinors}


\begin{definition}
Let $(M,g)$ be a pseudo-Riemannian manifold which is oriented and
strongly spin and $(S,\Gamma)$ be a strong real spinor bundle on
$(M,g)$. Let $\lambda\in \R$ be a real number. A {\em real Killing
  spinor} of Killing constant $\frac{\lambda}{2}$ is a global section
$\epsilon\in \Gamma(S)$ which satisfies:
\be
\nabla^S_X \epsilon = \frac{\lambda}{2} \, X^{\flat}\cdot\epsilon
\,\,\,\,\,\, \forall \,\, X\in \fX(M)~~.
\ee
It is called a {\em parallel spinor} if $\lambda=0$.
\end{definition}

\noindent Real Killing spinors are (unconstrained) generalized Killing
spinors relative to the connection $\cD=\nabla^S-\cA$ defined on $S$,
where $\cA_X=\frac{\lambda}{2}\Psi_\Gamma(X^\flat)\in \Gamma(End(S))$
for all $X\in \fX(M)$. The $End(S)$-valued one-form $\cA$ has symbol
${\hat \cA}\in \Omega^1(M,T^\ast M)$ given by ${\hat
  \cA}_X=\frac{\lambda}{2} X^\flat$. In this section, we study real
Killing spinors when $(M,g)$ is a spin Lorentzian four-manifold of
``mostly plus'' signature $(3,1)$ such that $H^1(M,\Z_2)=0$.

\begin{remark}
When $p-q\equiv_8 0,2$, a real Killing spinor can be viewed as a
complex Killing spinor which is preserved by a
$\Spin_0(p,q)$-invariant real structure on the complex spinor bundle
and which has real (in signature $(p,q)$) or purely imaginary (in
signature $(q,p)$) Killing constant. When comparing signatures, note
that \cite{LawsonMichelsohn} has a sign in the Clifford relation
opposite to our convention \eqref{Crel}. In the conventions of
loc. cit., the real Killing spinors considered below correspond to
special cases of imaginary Killing spinors, which were studied in
\cite{Bohle:2003abk,Leitner}. Reference \cite{Leitner} proves that a
Lorentzian four-manifold admitting a nontrivial imaginary Killing
spinor (which is a real Killing spinor in our convention) with null
Dirac current is locally conformal to a Brinkmann
space-time\footnote{Recall that a Brinkmann space-time is a
  four-dimensional Lorentzian manifold equipped with a non-vanishing
  parallel null vector field.}. In this section, we give a {\em
  global} characterization of Lorentzian four-manifolds admitting real
Killing spinors (see Theorem \ref{thm:equivalencerealkilling}).
\end{remark}

\subsection{Describing real Killing spinors through differential forms}

\noindent For the remainder of this section, let $(M,g)$ be a spin
Lorentzian four-manifold of ``mostly plus'' signature which satisfies
$H^1(M,\Z_2)=0$. Let $(S,\Gamma)$ be a spinor bundle on $(M,g)$.
Since $H^1(M,\Z_2)$ vanishes, the spinor bundle is automatically
strongly spin and unique up to isomorphism. We endow it with an
admissible pairing $\cB$ which is is skew-symmetric and of negative
adjoint type.

\begin{thm}
\label{thm:equivalencerealkilling}
$(M,g)$ admits a nontrivial real Killing spinor with Killing constant
$\frac{\lambda}{2}$ iff it admits a parabolic pair of one-forms
$(u,l)$ which satisfies:
\ben
\label{eq:realkillingequiv}
\nabla^g u = \lambda\, u\wedge l ~~ , ~~ \nabla^g l = \kappa\otimes
u + \lambda(l\otimes l - g)
\een
for some $\kappa\in \Omega^1(M)$. In this case, $u^{\sharp}\in \fX(M)$
is a Killing vector field with geodesic integral curves.
\end{thm}

\begin{remark}
Our conventions for the wedge product of one-forms are
as follows, where $\mathrm{S}_k$ denotes the permutation group on $k$
letters:
\beqa
\theta_1\wedge \ldots \wedge \theta_k &\eqdef& \sum_{\sigma\in
  \mathrm{S}_k}\epsilon(\sigma)\theta_{\sigma(1)}\otimes \ldots \otimes
\theta_{\sigma(k)}~~.
\eeqa
where $\theta_1,\ldots, \theta_k\in \Omega^1(M)$.
\end{remark}

\begin{proof}
Recall from Section \ref{sec:GCKLorentz4d} that a spinor $\xi$
associated to $(u,l)$ has a signed polyform square given by
$\alpha=u+u\wedge l$.  Theorem \ref{thm:GCKS} shows that $\xi$ is a
real Killing spinor iff:
\be
\nabla^g_X (u + u\wedge l) = \hat{\cA}_X\diamond (u + u\wedge l)  + 
(u + u\wedge l) \diamond (\pi\circ\tau)(\hat{\cA}_X) \,\,\,\,\,\, \forall X\in \fX(M)~~,
\ee
where $\hat{\cA}_X =\frac{\lambda}{2} X^{\flat}$. Expanding the
geometric product and isolating degrees, this equation gives:
\be
\nabla^g_X u = \lambda\, \big{(} u(X) l - l(X) u \big{)}~~,~~
\nabla^g_X (u\wedge l) = \lambda\, X^{\flat}\wedge u~~,
\ee
which in turn amounts to \eqref{eq:realkillingequiv} for some
$\kappa\in \Omega^1(M)$. The vector field $u^\sharp$ is Killing since
$\nabla^g u$ is an antisymmetric covariant 2-tensor by the first
equation in \eqref{eq:realkillingequiv}. Since $u$ is null and
orthogonal to $l$, the same equation gives
$\nabla^g_{u^\sharp}u=0$. Hence $\nabla^g_{u^\sharp} u^\sharp=0$,
i.e. $u^\sharp$ is a geodesic vector field.
\end{proof}

\begin{remark} 
The first equation in \eqref{eq:realkillingequiv} gives:
\be
\nabla^g_X u^\sharp = \lambda\, (u(X)\, l^\sharp - l(X)\, u^\sharp)\,\,\,\,\, \forall X\in \fX(M)~~.
\ee
Hence the null vector field $u^{\sharp}\in \fX(M)$ is not recurrent,
i.e. $\nabla^g$ does not preserve the rank one distribution spanned by
$u^{\sharp}$. Lorentzian manifolds admitting recurrent vector fields
are called \emph{almost decent} and were studied extensively (see
\cite{Galaev:2009ie,Galaev:2010jg,WalkerManifolds} and references
therein).
\end{remark}

\noindent Taking $\lambda = 0$ in Theorem \ref{thm:equivalencerealkilling} gives:

\begin{cor}
\label{cor:parallel}
$(M,g)$ admits a nontrivial parallel real spinor iff it admits a
  parabolic pair of one-forms $(u,l)$ which satisfies the following
  conditions for some one-form $\kappa \in \Omega^1(M)$:
\ben
\label{eq:realparallelequiv}
\nabla^g u = 0 ~~ , ~~ \nabla^g l = \kappa\otimes u~~.
\een
\end{cor}

\noindent Although Lorentzian manifolds admitting parallel spinors
were studied extensively in the literature (see \cite{Bryant,Leistner}
and references therein), Corollary \ref{cor:parallel} seems to be new. Recall
that $u$ coincides up to sign with the Dirac current of any of the
spinors $\xi, -\xi$ determined by the parabolic pair $(u,v)$ (see
Remark \ref{rem:Dirac}).  Reference \cite{Leitner} shows that a
Lorentzian four-manifold admitting an (imaginary, in the conventions
of loc. cit) Killing spinor with null Dirac current is locally
conformally Brinkmann. The following proposition recovers this result
in our approach.

\begin{prop}
\label{prop:Brinkmann} 
Suppose that $(M,g)$ admits a nontrivial real Killing spinor with
nonzero Killing constant $\frac{\lambda}{2}\neq 0$ and let $(u,l)$ be
a corresponding parabolic pair of one-forms. Then $u$ is locally
conformally parallel iff $l$ is locally equivalent to a {\em closed}
one-form $l'$ by transformations of the form \eqref{lprime}. In this
case, $(M,g)$ is locally conformal to a Brinkmann space-time.
\end{prop}

\begin{proof}
The one-form $u$ is locally conformally parallel iff for sufficiently
small non-empty open subsets $U\subset M$ there exists $f\in
\cC^\infty(U)$ such that the metric $\hat{g} = e^{f} g$ satisfies
$\nabla^{\hat{g}} u = 0$ on $U$. This amounts to:
\beqan
\label{nablaequ}
0 &=& \nabla^{\hat{g}}_X u = \nabla^{g}_X u + \dd f(X) u + \dd
f(u^{\sharp}) X^{\flat} - u(X)\dd f\nn\\
& =& \lambda (u(X) l - l(X) u) + \dd f(X) u + \dd f(u^{\sharp})
X^{\flat} - u(X)\dd f \,\,\,\,\, \forall X\in \fX(U)~~,
\eeqan
where in the last equality we used the first equation in
\eqref{eq:realkillingequiv}. Taking $X=u^\sharp$ and using the fact
that $u$ is nowhere-vanishing, null and orthogonal to $l$ gives $(\dd
f)(u^\sharp)=0$, whence \eqref{nablaequ} reduces to:
\be
u\otimes (\dd f-\lambda l)=(\dd f-\lambda l)\otimes u~~,
\ee
which amounts to the condition $\dd f=\lambda (l+ c u)$ for some $c\in
\cC^\infty(U)$. This has local solutions $f$ iff $l+ c u$ is
closed for some locally-defined function $c$. In this case, the
nowhere-vanishing null one form $u$ is $\nabla^{\hat g}$-parallel and
hence $(M,g)$ is locally conformally Brinkmann.
\end{proof}

\subsection{The Pfaffian system and its consequences}

Antisymmetrizing the two equations in \eqref{eq:realkillingequiv}
gives the Pfaffian system:
\ben
\label{eq:external}
\dd u =  2\lambda\, u\wedge l~~,~~\dd l  = \kappa \wedge u~~,
\een
which implies:

\begin{lemma}
\label{lemma:Brinkmann}
Let $(u,l)$ be a parabolic pair of one-forms which satisfies equations
\eqref{eq:realkillingequiv} for some $\kappa\in \Omega^1(M)$ and
$\lambda\in \R$ and let $C_u\subset TM$ be the rank one distribution
spanned by $u^{\sharp}$. Then $l$ is closed if and only iff $\kappa
\in \Gamma(C_u)$. Moreover, $u$ is closed iff $\lambda=0$.
\end{lemma}

\begin{remark}
Let:
\ben
\label{lprime}
l'=l+c u~~,
\een
where $c\in \cC^\infty(M)$. Then the parabolic pair $(u,l)$ satisfies
\eqref{eq:realkillingequiv} for the one-form $\kappa\in \Omega^1(M)$
iff the parabolic pair $(u,l')$ satisfies them with
$\kappa$ replaced by:
\ben
\label{kappaprime}
\kappa'=\kappa+\dd c-\lambda (c^2 u + 2 c l) ~~.
\een
Similarly $(-u,l)$ satisfies them with $\kappa$ replaced by $-\kappa$.
In particular, a nontrivial real Killing spinor on $(M,g)$ determines
$\kappa$ up to transformations of the form:
\be
\kappa\rightarrow \zeta \kappa+\dd c-\lambda (c^2 u + 2 c l)
~~,~~\mathrm{with}~~\zeta\in \{-1,1\}~~\mathrm{and}~~c\in \cC^\infty(M)~~.
\ee
Moreover, $(u,l)$ satisfies \eqref{eq:external} for $\kappa$ iff
$(u,l')$ satisfies them with $\kappa$ replaced by
$\kappa'$. Similarly, $(-u,l)$ satisfies \eqref{eq:external} with
$\kappa$ replaced by $-\kappa$.
\end{remark}

By the Frobenius theorem, the first equation in \eqref{eq:external}
implies that the distribution $\ker(u)=C^{\perp}_u \subset TM$
integrates to a codimension one foliation of $M$ which is
transversally-orientable since $u$ is nowhere-vanishing. Since $C_u$
is contained in $C_u^{\perp}$, this foliation is degenerate in the
sense that the restriction of $g$ to $C_u^\perp$ is a degenerate
vector bundle metric. In particular, the three-dimensional vector
space $C^\perp_{u,m}$ is tangent to the causal cone $\cL_m\subset T_m
M$ along the null line $C_{u,m}$ at any point $m\in M$ (see Appendix
\ref{app:flags}) and the complement $C_u^\perp\setminus C_u$ consists
of spacelike vectors. Since $l$ is orthogonal to $u$, we have
$l^{\sharp}\in \Gamma(C^{\perp}_u)$. The vector fields $u^\sharp$ and
$l^\sharp$ span a topologically trivial distribution $\Pi^\sharp$ of
parabolic 2-planes contained in $C_u$.

Let $S(C^{\perp}_u)$ be any complement of $C_u$ in $C^{\perp}_u$:
\be
C_u^{\perp} = C_u \oplus S(C^{\perp}_u)~~.
\ee
Such a complement is known as a \emph{screen
  bundle} of $C^{\perp}_u$ (see \cite{DuggalSahin} and references
therein); in our situation, it can be chosen such that $l^\sharp\in
\Gamma(S(C_u^\perp))$, in which case we can further decompose
$S(C_u^\perp)=C_l\oplus L$, where $C_l$ is the rank one distribution
spanned by $l^\sharp$ and $L$ is any rank one distribution 
complementary to $\Pi^\sharp$ in $C_u^\perp$.
For any choice of the screen bundle, the restriction of $g$ to
$S(C^{\perp}_u)$ is non-degenerate and positive-definite and hence
admits an orthogonal complement in $TM$ which has the form $C_u\oplus
C_v$, where $C_v\subset TM$ is the rank one distribution spanned by
the unique null vector field $v^{\sharp} \in \fX(M)$ which is
orthogonal to $S(C^{\perp}_u)$ and satisfies $g(u^{\sharp},v^{\sharp})
= 1$. This gives:
\ben
\label{eq:screensplittingorthogonal}
TM = (C_v \oplus C_u) \oplus S(C^{\perp}_u)~~,
\een
which allows us to write the metric as:
\ben
\label{eq:metricscreensplit}
g = u \otimes v + v \otimes u + q~~,
\een
where $q = g\vert_{S(C^{\perp}_u)}$ and $v$ is the one-form dual to
$v^\sharp$.

\begin{lemma}
\label{lemma:walkerRKS} 
Suppose that $(M,g)$ admits a nontrivial real Killing spinor $\epsilon
\in \Gamma(S)$ with Killing constant $\frac{\lambda}{2}\neq 0$ and let
$(u,l)$ be a corresponding parabolic pair of one-forms. Around every
point in $M$, there exist local {\em Walker-like coordinates}
$(x^v,x^u,x^1,x^2)$ with $u^{\sharp} = \partial_{x^u}$ in which the
metric takes the form:
\ben
\label{WalkerMetric}
\dd s^2_g = \cF\, (\dd x^v)^2  + 2 \cK \dd x^v \dd x^u + 
\omega_i \dd x^v \dd x^i +  q_{ij}\, \dd x^i \dd x^j~~,
\een
where $\cF$, $\cK$, $\omega_i$ and $q_{ij}$ are locally-defined
functions which do not depend on $x^u$ and such that $\cK$ is
nowhere-vanishing. In these coordinates, the one-forms $u$ and $l$ can
be written as:
\be
u=\cK \dd x^v~~,~~l = -\frac{1}{2\lambda} \dd\log(\cK) + \fs\, \dd x^v=-\frac{1}{2\lambda} 
\dd\log(\cK) + \frac{\fs}{\cK}\, u~~,
\ee
for some locally-defined function $\fs$.
\end{lemma}

\begin{proof} 
Since $C^{\perp}_u$ is integrable and of corank one and $C_u\subset
C_u^\perp$ has rank one, there exist local coordinates
$(x^v,x^u,x^1,x^2)$ on $M$ such that $\partial_{x^u}=u^\sharp$, the
vector fields $\partial_{x^u}, \partial_{x^1}, \partial_{x^2}$ span
$C_u^\perp$ and $\partial_{x_v}$ is null. In such
local coordinates, we have:
\be
u = g_{uv}\, \dd x^v~~,~~\,\,\,\, \mathrm{with}~~g_{uv} =u(\partial_{x^v})=
g(\partial_{x^u},\partial_{x^v})~~.
\ee
Notice that $g_{uv}$ is nowhere-vanishing since $u$ is.
The vector fields $\partial_{x^1}$ and $\partial_{x^2}$ span an
integrable local screen $S(C_u^\perp)$ for $C_u^\perp$.  Let $v$ be
the unique null one form which vanishes along $S(C_u^\perp)$ and
satisfies $v(\partial_{x^u})=1$. Then the vector field $v^\sharp$
satisfies the assumptions which allow us to write the metric in the
form \eqref{eq:metricscreensplit}.  Writing $v = v_v\, \dd x^v + v_u\,
\dd x^u + v_i\, \dd x^i$, the condition $v(\partial_{x^u})=1$ implies $v_u =
1$ and \eqref{eq:metricscreensplit} gives:
\be
g = 2\,g_{u v}\, v_v\, \dd x^v\otimes \dd x^v + g_{u v}\,(\dd
x^v\otimes \dd x^u +\dd x^u\otimes \dd x^v )+g_{u v} \, v_i\, (\dd
x^v \otimes \dd x^i+\dd x^i\otimes \dd x^v) + q_{ij}\, \dd x^i\otimes \dd x^j~~.
\ee
Relabeling coefficients gives \eqref{WalkerMetric} with $\cF= 2 g_{uv}
v_v$, $\cK=g_{uv}$ and $\omega_i = 2 g_{uv} v_i$. The coefficients of
$g$ do not depend on $x^u$ since $u^{\sharp} = \partial_{x^u}$ is a
Killing vector field. In these coordinates we have $u = \cK \,\dd x^v$
(hence $\cK$ is nowhere-vanishing) and the first equation of the
Pfaffian system \eqref{eq:external} becomes:
\be
\left(\dd\cK + 2\lambda\, \cK\,l\right)\wedge \dd x^v = 0~,
\ee
showing that:
\be
l = -\frac{1}{2\lambda} \dd\log(\cK) + \fs\, \dd x^v
\ee
for some locally-defined function $\fs$.  
\end{proof}

\noindent Lemma \ref{lemma:walkerRKS} gives existence of Walker-like
coordinates on Lorentzian four-manifolds admitting real Killing
spinors. These generalize the classical Walker coordinates
\cite{Walker} of Lorentzian manifolds which admit a parallel null line
\cite{Galaev:2009ie,WalkerManifolds,Walker}. The main difference is
that our $u$ is not recurrent. On the other hand, our $u$ is Killing
--- a condition which may not hold on generic Walker manifolds.

\begin{example} 
The simply-connected four-dimensional anti-de Sitter
space $\AdS_4$ admits Walker-like coordinates $(x^v,x^u,x,y)$ in which
the Anti-de Sitter metric $g$ reads:
\be
\dd s^2_{\AdS_4} = \frac{1}{c \, y^2} \left[ \dd x^v \dd x^u + (\dd x)^2 + (\dd y)^2\right]~~,
\ee
where $c$ is a positive constant equal to minus the curvature. It is well-known
\cite{AlonsoAlberca:2002gh,Gibbons:1986uv,Lu:1998nu} that $\AdS_4$
admits a four-dimensional space of real Killing spinors.
\end{example}

\subsection{The locally stationary and locally integrable case}

\begin{definition}
\label{def:localtypesRKS}
Suppose that $(M,g)$ admits a nontrivial real Killing spinor
$\epsilon\in \Gamma(S)$ with nonzero Killing constant. We say that
$(M,g,\epsilon)$ is:
\begin{itemize}
\itemsep 0.0em
\item {\em locally stationary} if, around every point, the Walker-like
coordinates induced by $\epsilon$ are such that $\partial_{x^v}$ is
Killing.
\item {\em locally integrable} if, around every point, the Walker-like
coordinates coordinates induced by $\epsilon$ are such that $\omega_1 =\omega_2=
0$.
\item {\em locally static} if, around every point, the Walker-like
coordinates induced by $\epsilon$ are such that $\partial_{x^v}$ is
Killing, $\omega_1=\omega_2 = 0$ and $\cF = 0$.
\end{itemize}
\end{definition}

\begin{remark}
Notice that the locally defined rank two distribution $\Delta$ spanned
by $\partial_{x^v}$ and $\partial_{x^u}$ is nondegenerate (and hence
admits a non-vanishing timelike section) since $\partial_{x^u}$ is
null and $\cK=g(\partial_{x^u},\partial_{x^v})$ is locally
nowhere-vanishing. If $(M,g,\epsilon)$ is locally stationary in the
sense above then some linear combination of $\partial_{x^v}$ and
$\partial_{x^u}$ is a timelike Killing vector field, whence $(M,g)$ is
locally stationary in the standard sense. If $(M,g,\epsilon)$ is
locally integrable, then the orthogonal complement of $\Delta$ is an
integrable spacelike distribution of rank two. If $(M,g,\epsilon)$ is
locally static then $\partial_{x^v}$ is null and $(M,g)$ admits a
local hypersurface orthogonal to a time-like Killing vector field $X$
(namely $X=\partial_{x^v}+\partial_{x^u}$ or
$X=\partial_{x^v}-\partial_{x^u}$), hence $(M,g)$ is locally static in
the standard sense.
\end{remark}

\begin{thm}
\label{thm:WalkerRKS}
The following statements are equivalent:
\begin{enumerate}[(a)]
\itemsep 0.0em
\item There exists a nontrivial real Killing spinor $\epsilon \in
  \Gamma(S)$ with Killing constant $\frac{\lambda}{2} \neq 0$ on
  $(M,g)$ such that $(M,g,\epsilon)$ is locally stationary and locally
  integrable.
\item $(M,g)$ is locally isometric to a Lorentzian four-manifold of the
form:
\be
(\hat{M},\dd s^2_{\hat{g}}) = \left(\R^2\times X , \,\cF\, (\dd x^v)^2 +
2\,\cK\, \dd x^v \dd x^u + q\right)~~,
\ee
where $(x^v,x^u)$ are Cartesian coordinates on $\R^2$, $X$ is a
non-compact, oriented and simply-connected surface endowed with the
Riemannian metric $q$ and $\cF , \cK \in C^{\infty}(X)$ are functions
on $X$ (with $\cK$ nowhere-vanishing) which satisfy:
\begin{eqnarray}
\label{eq:iffsusyconfRKS}
\nabla^q\dd\cK - \frac{\dd\cK\otimes\dd\cK}{2\cK} =
2\lambda^2\,\cK\,q~~,~~ \Delta_q \cK = 6 \lambda^2 \cK~~,~~
\partial_{x^u} \fs = \kappa(\partial_{x^u})\, \cK\,
,\nonumber \\ \frac{q^\ast(\dd\cK,\dd\cF)}{4\,\lambda\,\cK} -
\partial_{x^v}\fs = \lambda (\cF - \fs^2) -
\kappa(\partial_{x^v})\, \cK~~,~~ \partial_{x^i} \fs =
\kappa(\partial_{x^i})\, \cK\, 
\end{eqnarray}
for some function $\fs\in C^{\infty}(M)$ and some one-form $\kappa\in
\Omega^1(M)$, where $x^1,x^2$ are local coordinates on $X$.
\end{enumerate}
In this case, the formulas:
\ben
\label{ulK}
u = \cK\,\dd x^v ~~,~~ l = -\frac{1}{2\lambda} \dd\log(\cK) +
\fs\, \dd x^v~~.
\een
give a parabolic pair of one-forms $(u,l)$ corresponding to the real
Killing spinor $\epsilon$, which satisfy equations
\eqref{eq:realkillingequiv} with respect to the one-form
$\kappa$. Moreover, $(M,g)$ is Einstein with Einstein constant
$\Lambda$ iff $\Lambda = -3\lambda^2$ and the following equations are
satisfied, where $\Ric^q$ is the Ricci tensor of $q$:
\ben
\label{eq:einsteinRKS}
\Delta_q\cF - \frac{q^\ast(\dd\cK,\dd \cF)}{\cK} = 2\,\lambda^2 \cF~~,
~~ \Ric^q = -\lambda^2 q~~,
\een
in which situation $(X,q)$ is a hyperbolic Riemann surface.
\end{thm}

\begin{remark}
Here, the Laplacian $\Delta_q$ is defined through:
\be
\Delta_q(f) \eqdef \tr (\nabla^q\grad_q f) \,\,\,\,\, \forall f\in C^{\infty}(X)~~.
\ee
The second equation in \eqref{eq:iffsusyconfRKS}:
\ben
\label{eigv}
\Delta_q \cK = 6 \lambda^2 \cK~~,
\een
is a ``wrong sign'' eigenvalue problem for the Laplacian on $X$.
Since $\Delta_q$ is negative semidefinite, no nontrivial solutions to
\eqref{eigv} exists unless $X$ is non-compact. The function $\fs$ can
be chosen at will as long as equations \eqref{eq:iffsusyconfRKS} hold,
since different choices produce the same signed polyform square
$\alpha = u + u\wedge l$ of $\epsilon$, namely:
\ben
\label{eq:sqaurecaseL}
\alpha = \cK\,\dd x^v + \frac{1}{2\lambda} \dd\cK \wedge \dd x^v~~.
\een
We can exploit this freedom to choose:
\ben
\label{sgauge}
\fs^2 =  \cF - \frac{q^\ast(\dd\cK ,\dd\cF)}{4\lambda^2 \cK}~~,
\een
which is independent of $x^v$ and $x^u$. For this choice of $\fs$,
equations \eqref{eq:iffsusyconfRKS} reduce to:
\ben
\label{eq:iffsusyconfRKSII}
\kappa=\frac{\dd \fs}{\cK}~~,~~\nabla^q\dd\cK
- \frac{\dd\cK\otimes\dd\cK}{2\cK} = 2\lambda^2\,\cK\,q~~,~~ \Delta_q
\cK = 6 \lambda^2 \cK
\een
and hence $\kappa$ is a one-form defined on $X$ which is completely
determined by $\lambda$, $\cF$, $\cK$ and $q$. In particular, the
condition that $(M,g)$ admits a nontrivial real Killing spinor with
Killing constant $\frac{\lambda}{2}$ such that $(M,g,\epsilon)$ is
locally stationary and locally integrable reduces to the last two
equations in \eqref{eq:iffsusyconfRKSII}, which involve only
$\lambda$, $\cK$ and $q$ but do not involve $\cF$. This generalizes a
statement made in \cite[page 391]{Gibbons:1986uv}. On the other hand,
the Einstein condition \eqref{eq:einsteinRKS} involves both $\cF$ and
$\cK$. Strictly speaking, the ``gauge choice'' \eqref{sgauge} requires:
\be
\cF - \frac{q^\ast(\dd\cK , \dd\cF)}{4\,\lambda^2\, \cK} \geq 0
\ee
if $l$ is to be well-defined, since the formula for $l$ involves
$\fs$. However, the real Killing spinor associated to $(u , l)$ is
well-defined and satisfies the Killing spinor equations even when
$\fs^2$ is formally negative somewhere on $M$, because $\epsilon$ is
determined by the polyform \eqref{eq:sqaurecaseL}, which is
independent of $\fs$.
\end{remark}

\begin{proof}
By Lemma \ref{lemma:walkerRKS} and Theorem
\ref{thm:equivalencerealkilling}, we must solve equations
\eqref{eq:realkillingequiv} for $u$ and $l$ of the form \eqref{ulK}
, which automatically
satisfy the first equation of the Pfaffian system \eqref{eq:external}.
The first equation in \eqref{eq:realkillingequiv} is equivalent to the
condition that the vector field $u^{\sharp} = \partial_{x^u}$ is
Killing, together with the condition the first equation in
\eqref{eq:external} holds. Thus it suffices to consider the second
equation in \eqref{eq:realkillingequiv}. Evaluating this equation on
$\partial_{x^v}$ and $\partial_{x^u}$ gives the system:
\beqa
\nabla^g_{\partial_{x^v}} l &=& \kappa(\partial_{x^v})\, u + \lambda\,
l(\partial_{x^v})\, l - \lambda\, g(\partial_{x^v})~~\nn\\
\nabla^g_{\partial_{x^u}} l &=& \kappa(\partial_{x^u})\, u + \lambda\,
l(\partial_{x^u})\, l - \lambda\, g(\partial_{x^u})~~,
\eeqa
which reduce to the following equations for $u$ and $l$ as in
\eqref{ulK}:
\ben
\label{eq1}
\frac{q^\ast(\dd\cK , \dd\cF)}{4\lambda\,\cK} - \partial_{x^v} \fs
= \lambda\,\cF - \lambda\,\fs^2 -
\kappa(\partial_{x^v})\,\cK~~,~~ \frac{q^\ast(\dd\cK ,
\dd\cK)}{4\lambda\,\cK} = \lambda \, \cK~~,~~ \partial_{x^u}
\fs = \kappa(\partial_{x^u})\, \cK~~.
\een
On the other hand, restricting the second equation in
\eqref{eq:realkillingequiv} to $X$ and using \eqref{ulK} gives:
\ben
\label{eq2}
\nabla^q\dd\cK -  \frac{\dd\cK\otimes\dd\cK}{2\cK} = 2\lambda^2\,\cK\,q~~,~~
\partial_{x^i}\fs = \kappa(\partial_{x^i}) \cK~~.
\een
Furthermore, taking the trace of \eqref{eq2} and combining it with the
third equation in \eqref{eq1} we obtain $\Delta_q \cK = 6 \lambda^2
\cK$. Together with relations \eqref{eq1} and \eqref{eq2}, this establishes the system
\eqref{eq:iffsusyconfRKS}. Consider now the Einstein equation $\Ric^g
= \Lambda\, g$ on $(M,g)$. The only non-trivial components are:
\be
\Ric^g(\partial_{x^v},\ \partial_{x^v}) = \Lambda\, \cF~~,~~ 
\Ric^g(\partial_{x^u},\partial_{x^v}) = 
\Lambda\, \cK~~,~~ \Ric^g\vert_{TX} = \Lambda\, q~~.
\ee
Direct computation gives:
\begin{eqnarray*}
\Ric^g(\partial_{x^v},\ \partial_{x^v}) &=&
 - \frac{1}{2} \Delta_q \cF + \frac{1}{2\cK}\, q^\ast(\dd\cK , \dd \cF) - 
\frac{\cF}{2\cK^2}\, q^\ast(\dd\cK , \dd \cK)~~, \\
\Ric^g(\partial_{x^v},\ \partial_{x^u}) &=& - \frac{1}{2}
 \Delta_q \cK~~\,\, ,\,\,~~ \Ric^g\vert_{TX} = \Ric^q - 
\frac{\nabla^q\dd\cK}{\dd\cK} + \frac{\dd\cK\otimes \dd\cK}{2\cK^2}~~.
\end{eqnarray*}
Combining these relations with \eqref{eq:iffsusyconfRKS}, we conclude.
\end{proof}

\begin{remark}
Lorentzian four-manifolds admitting nontrivial real
Killing spinors are supersymmetric configurations of four-dimensional
$\cN=1$ minimal $\AdS$ supergravity
\cite{FreedmanProeyen,Ortin}. Supersymmetric {\em solutions} of that
theory are Lorentzian four-manifolds admitting nontrivial real Killing
spinors which satisfy the Einstein equation with negative cosmological
constant. Hence Theorem \ref{thm:WalkerRKS} characterizes all locally
integrable and locally stationary supersymmetric solutions of this
theory. To our best knowledge, the classification of four-dimensional
Lorentzian manifolds admitting real Killing spinors is currently
open. Theorem \ref{thm:equivalencerealkilling} and
\cite{Galaev:2009ie} could be used to attack this problem
in full generality.
\end{remark}

\noindent
Theorem \ref{thm:WalkerRKS} suggests a strategy to construct
Lorentzian four-manifolds admitting real Killing spinors. Fix a
simply-connected (generally incomplete) hyperbolic Riemann surface
$(X,q)$ and consider the eigenspace of the Laplacian on $(X,q)$ with
eigenvalue $6\lambda^2$. In this space, look for a function $\cK$
which satisfies the first equation in \eqref{eq:iffsusyconfRKS}. If
such exists, it gives a real Killing spinor for any $\cF\in
C^{\infty}(X)$. Below, we give special classes of solutions
when $(X,q)$ is the Poincar\'e half-plane.


\subsection{Special solutions from the Poincar\'e half-plane}
\label{subsec:SpecialPoincare}


Take:
\ben
\label{eq:upperhalfform}
(M,\dd s^2_g) = \left( \R^2\times \H,\, \cF\, 
(\dd x^v)^2  + 2\,\cK\, \dd x^v \dd x^u   + 
c\,\frac{(\dd x)^2 + (\dd y)^2}{y^2} \right)~~,
\een
where $x,y$ are global coordinates on the Poincar\'e half-plane $\H =
\left\{ (x,y) \in \R^2 \,\, \vert\,\, y> 0\right\}$ and $\cF,\cK$ are
real-valued functions defined on $\H$ (with $\cK$ nowhere-vanishing),
while $c>0$ is a constant. Let $q$ denote the metric $c\,\frac{\dd
  x\otimes \dd x + \dd y\otimes \dd y}{y^2}$ on $\H$. Theorem
\ref{thm:WalkerRKS} shows that such $(M,g)$ admits a real Killing
spinor with Killing constant $\frac{\lambda}{2}\neq 0$ iff:
\ben
\label{eq:upper1}
\nabla^q\dd\cK - \frac{\dd\cK\otimes\dd\cK}{2\cK} =
2\lambda^2\,\cK\,q~~,~~ \Delta_q \cK = 6 \lambda^2 \cK~~,
\een
in which case $(M,g)$ is Einstein iff:
\ben
\label{eq:upper2}
\Delta_q\cF -  \frac{q^\ast(\dd\cK,\dd \cF)}{\cK} = 2\,\lambda^2 \cF~~.
\een
Direct computation shows that equations \eqref{eq:upper1} are
equivalent with:
\begin{eqnarray*}
& \partial^2_{x}\cK -\frac{\partial_y \cK}{y} =
  \frac{(\partial_{x}\cK)^2}{2\cK} + \frac{2\lambda^2 c}{y^2} \cK~~,
  ~~ \partial^2_{y}\cK +\frac{\partial_y \cK}{y} =
  \frac{(\partial_{y}\cK)^2}{2\cK} + \frac{2\lambda^2 c}{y^2} \cK~~,
  \\
& \partial^2_{xy} \cK + \frac{\partial_x \cK}{y} = \frac{\partial_x
    \cK \partial_y \cK}{2\cK} ~~, ~~ y^2 \left( \partial^2_{x} \cK
  + \partial^2_{y} \cK \right) = 6\, \lambda^2 c\,\cK~~.
\end{eqnarray*}
This gives the following:

\begin{cor}
The Lorentzian four-manifold \eqref{eq:upperhalfform} admits a
nontrivial real Killing spinor with Killing constant
$\frac{\lambda}{2}\neq 0$ iff:
\begin{eqnarray}
\label{eq:cKequationsII}
& \partial^2_{y}\cK +\frac{\partial_y \cK}{y} = \frac{(\partial_{y}\cK)^2}{2\cK} + 
\frac{2\lambda^2 c}{y^2} \cK~~,~~ \partial^2_{xy} \cK + \frac{\partial_x \cK}{y} = 
\frac{\partial_x \cK \partial_y \cK}{2\cK} ~~, \nn\\
&  (\partial_{x} \cK)^2 +  (\partial_{y} \cK)^2 = \frac{4\, \lambda^2 c}{y^2}\,\cK^2\, 
,~~   y^2 \left( \partial^2_{x} \cK +  \partial^2_{y} \cK \right) = 6\, \lambda^2~~.
c\,\cK~~.
\end{eqnarray}
In this case, it is Einstein iff $\cF$ satisfies \eqref{eq:upper2}.
\end{cor}

\noindent Choosing $\cF$ to not satisfy
\eqref{eq:upper2} produces large families of non-Einstein Lorentzian
four-manifolds admitting real Killing spinors.

\begin{example}
\label{ep:AdS}
Taking $\cK = \cF = \frac{c}{y^2}$ gives a solution of
\eqref{eq:cKequationsII} iff $c\, \lambda^2 = 1$. Hence
the Lorentzian four-manifold:
\be
(M,\dd s^2_g) = \Big{(}\R^2\times \H,\, 
\frac{1}{\lambda^2 y^2} \left[(\dd x^v)^2  + 2\,\dd x^v \dd x^u  
 + (\dd x)^2 + (\dd y)^2 \right] \Big{)}\, 
\ee
admits a real Killing spinor. This is the $\AdS_4$ space with metric
written in horospheric coordinates \cite{Gibbons:2011sg}, which is
well-known to admit the maximal number (namely four) of real Killing
spinors \cite{Bellorin:2005hy}.
\end{example}

\noindent More examples can be constructed by solving in more
generality the eigenvector problem for the Laplace operator of the
Poincar\'e half plane and checking which solutions satisfy the first
equation in \eqref{eq:upper1}.  We illustrate this by constructing
solutions obtained through separation of variables. Set:
\be
\cK = k_x\, k_y~~,
\ee
where $k_x\in \cC^\infty(\H)$ depends only on $x$ and $k_y\in
\cC^\infty(\H)$ depends only on $y$. The second
equation in \eqref{eq:cKequationsII} gives:
\ben
\label{eq:k}
\dot{k}_x \left( \dot{k}_y + \frac{2 k_y}{y}\right) = 0~~,
\een
where the dot denotes derivation with respect to the corresponding
variable. When $\dot{k}_x = 0$, equations \eqref{eq:cKequationsII}
reduce to:
\beqa
& \partial_y \cK + \frac{2}{y} \cK = 0~~,~~ \lambda^2 c = 1~~,
\eeqa
with general solution $\cK= c_0 y^{-2}$ (where $c_0\neq 0$ is a
constant). If $\dot{k}_x \neq 0$, then \eqref{eq:k} gives $\dot{k}_y +
\frac{2 k_y}{y} = 0$, so $k_y = c_0 y^{-2}$ for a non-zero constant
$c_0$. Using this in \eqref{eq:cKequationsII} gives $\dot{k}_x = 0$, a
contradiction. Hence every Lorentzian four-manifold of the form:
\ben
\label{eq:upperhalfformfamily}
(M,\dd s^2_g) = \left( \R^2\times \H,\, \cF\, 
(\dd x^v)^2  + 2\, c_0\, \frac{\dd x^v \dd x^u}{y^{2}}   + 
\frac{(\dd x)^2 + (\dd y)^2}{\lambda^2 y^{2}} \right)\, 
\een
with $\cF$ a smooth function admits nontrivial real Killing
spinors. This gives large families of non-Einstein Lorentzian
four-manifolds carrying real Killing spinors by taking $\cF$ to be
generic. The Lorentzian manifold \eqref{eq:upperhalfformfamily} is
Einstein when equation \eqref{eq:upper2} is satisfied, which for $\cK=
c_0 y^{-2}$ reads:
\ben
\label{eq:upper3}
y^2 \left( \partial^2_{x} \cF +  \partial^2_{y} \cF \right) + 
2\,y \,\partial_y \cF = 2\,\cF~~.
\een 
To study \eqref{eq:upper3}, we try the separated Ansatz:
\be
\cF = f_x\, f_y~~,
\ee
where $f_x\in \cC^\infty(\H)$ depends only on $x$ and $f_y\in \cC^\infty(\H)$
depends only on $y$. Equation \eqref{eq:upper3} gives:
\ben
\label{feq}
\frac{\ddot{f}_x}{f_x} = c^2 = \frac{\ddot{f}_y}{f_y} +
\frac{2\,\dot{f}_y}{y\, f_y} - \frac{2}{y^2}
\een
for some $c\in \R$. If $c = 0$, this is solved by:
\be
f_x = a_1 + a_2 x~~,~~ f_y = a_3 y + \frac{a_4}{y^{2}}~~,  
\ee
where $\vec{a}\eqdef (a_1 , \hdots , a_4) \in \R^4$. This gives the
following family of Einstein Lorentzian metrics on $\R^2\times \H$
admitting real Killing spinors, where we eliminated $c_0$ by rescaling
$x^u$:
\ben
\label{eq:AdSdeformed}
\dd s^g = (a_1 + a_2 x) \left(a_3 y + \frac{a_4}{y^{2}}\right)\, 
(\dd x^v)^2  +\frac{\dd x^v\dd x^u}{y^{2}}   +  
\frac{(\dd x)^2 + (\dd y)^2}{ \lambda^2 y^2}~~.
\een
For $a_1 = a_2 = a_3 = a_4 =0$ we
recover the $\AdS_4$ metric written in horospheric coordinates. Thus
\eqref{eq:AdSdeformed} gives a four-parameter deformation of
$\AdS_4$. Every choice $\vec{a}\in \R^{4}$ produces an Einstein
metric on $\R^2\times \H$ with Einstein constant $\Lambda = -3
\lambda^2$ admitting real Killing spinors.

If $c\neq 0$, the first equation in \eqref{feq} gives:
\be
f_x = a_1 e^{c x} + a_2 e^{-c x}
\ee
with $a_1 , a_2 \in \R$. On the other hand, the 
equation for $f_y$ can be written as:
\be
y^2 \ddot{f}_y   + 2 y \dot{f}_y + (c^2  y^2 - 2) f_y = 0~~,
\ee
being the radial part of the Helmholtz equation in spherical
coordinates. Its general solution is a linear combination of the
spherical Bessel functions $\BY$ and $\BJ$:
\be
f_y = a_3 \BY(c y) + a_4 \BJ(c y)~~, 
\ee
where $a_3, a_4 \in \R$. We have:
\be
\BJ(c y) =  \frac{\sin(c y)}{c^2 y^2} - 
\frac{\cos(c y)}{c y} ~~,~~ \BY(c y) =   
-\frac{\cos(c y)}{c^2 y^2} - 
\frac{\sin(c y)}{c y} ~~,
\ee
whence:
\be
\dd s^2_g = (a_1 e^{c x} + a_2 e^{-c x}) \left[a_3
\BY(c y) + a_4 \BJ(c y)\right]\, (\dd x^v)^2
+\frac{\dd x^v\dd x^u}{y^{2}} + \frac{(\dd x)^2 + (\dd
y)^2}{ \lambda^2 y^2}~~.
\ee
This gives a four-parameter family (parameterized by
$(a_1,a_2,a_3,a_4)\in \R^4$) of Lorentzian Einstein metrics on
$\R^2\times \H$ admitting real Killing spinors.

\begin{remark}
When $a_1a_2\neq 0$ and $a_3\neq 0$, the Lorentzian four-manifolds
constructed above are not isometric to $\AdS_4$, since their
Weyl tensor is non-zero and their Riemann tensor is not parallel.
\end{remark}
 

\section{Supersymmetric heterotic configurations}
\label{sec:Susyheterotic}


In this section we consider generalized constrained Killing spinors in
an abstract form of heterotic supergravity (inspired by
\cite{Tipler}), which is parameterized by a triplet $(M,P,\fc)$, where
$M$ is a spin open four-manifold, $P$ is a principal bundle over $M$
with compact semi-simple Lie structure group $\G$ and $\fc$ is an
$\Ad_G$-invariant, symmetric and non-degenerate inner product on the
Lie algebra $\fg$ of $G$. Let $\fg_P \eqdef P\times_{\Ad_G}\fg$ be the
adjoint bundle of $P$. The Killing spinor equations of heterotic
supergravity couple a strongly spinnable Lorentzian metric $g$ on $M$
(taken to be of ``mostly plus'' signature), a closed one-form
$\varphi\in \Omega^1(M)$, a three-form $H\in \Omega^3(M)$ and a
connection $A$ on $P$. This system of partial differential equations
characterizes supersymmetric configurations of the theory defined by
$(M,P,\fc)$. For simplicity, we assume $H^1(M,\Z_2) = 0$, although
this assumption can be relaxed. We refer the reader to Appendix
\ref{sec:Susyheterotic} for certain details.

\subsection{Supersymmetric heterotic configurations}

Let us fix a triple $(M,P,\fc)$ as above, where $H^1(M,\Z_2)=0$. For
every strongly-spinnable metric $g$ of signature $(3,1)$ on $M$, let
$\bS_g = (S,\Gamma,\cB)$ be a paired real spinor bundle on $(M,g)$,
where the admissible pairing $\cB$ is skew-symmetric and of negative
adjoint type. With our assumptions, $\bS_g$ is unique up to
isomorphism of paired spinor bundles, combined with a rescaling of
$\cB$ by a non-zero constant. Let $F_A \in \Omega^2(\fg_P)$ denote the
curvature form of a connection $A$ on $P$. As explained in Appendix
\ref{sec:Susyheterotic}, $g$ and $\fc$ induce a symmetric morphism of
vector bundles $\fc(-\wedge -):\Lambda(M,\fg_p)\otimes
\Lambda(M,\fg_p)\rightarrow \Lambda(M)$, where $\Lambda(M,\fg_p)\eqdef
\Lambda(M)\otimes \fg_P$. For any 3-form $H\in \Omega^3(M)$, let
${\hat \nabla}^H$ be the natural lift to $S_g$ of the unique
metric connection on $(M,g)$ with totally skew-symmetric torsion given by
$-H$.  

\begin{definition}
\label{def:susyconfHet}
A {\em heterotic configuration} for $(M,P,\fc)$ is an ordered
quadruplet $(g,\varphi,H,A)$, where $g$ is a strongly-spinnable
Lorentzian metric on $M$, $\varphi\in \Omega^1(M)$ is a closed
one-form, $H\in \Omega^3(M)$ is a three-form and $A\in \cA_P$ is a
connection on $P$ such that the {\em modified Bianchi identity} holds:
\ben
\label{Bianchi}
\dd H =\fc(F_A\wedge F_A)~~.
\een
The configuration is called {\em supersymmetric}
if there exists a nontrivial spinor $\epsilon\in
\Gamma(S_g)$ such that:
\ben
\label{eq:Hetconf}
\hat{\nabla}^H \epsilon = 0~~,~~ \varphi\cdot \epsilon =  H\cdot \varepsilon ~~,~~
F_A\cdot \epsilon = 0~~.
\een
\end{definition}

\begin{remark}
Equations \eqref{eq:Hetconf} encode vanishing of the gravitino,
dilatino and gaugino supersymmetry variations. Since we work in
Lorentzian signature, supersymmetric configurations need not solve the
equations of motion (which are given in Appendix
\ref{sec:Susyheterotic}). However, the study of supersymmetric
configurations is a first step toward classifying supersymmetric
solutions. The study of supersymmetric solutions this theory in the
physical case of ten Lorentzian dimensions was pioneered in
\cite{Gran:2005wf,Gran:2007fu,Gran:2007kh}, where their local
structure was characterized. The last equation in \eqref{eq:Hetconf}
is formally identical to the spinorial characterization of instantons
in Riemannian signature and dimensions from four to eight.
\end{remark}

\subsection{Characterizing supersymmetric heterotic configurations through differential forms}

The metric connection $\nabla^H$ is given by (see Appendix
\ref{sec:Susyheterotic} for notation):
\be
\nabla^H_Y X = \nabla^g_Y X - \frac{1}{2} H^{\sharp}(X,Y) \,\,\,\,\,\, \forall X, Y \in \fX(M)~~,
\ee
where $\nabla^g$ is the Levi-Civita connection of $(M,g)$ and
$H^{\sharp}$ is $H$ viewed as a $TM$-valued two-form. Hence the first equation in
\eqref{eq:Hetconf} can be written as:
\be
\nabla^g_X \epsilon - \frac{1}{4} H(X)\cdot \epsilon = 0 \,\,\,\,\,\, \forall\,\, X\in \fX(M)~~,
\ee
where $H(X)\eqdef \iota_X H\in \Omega^2(M)$. This shows that $\epsilon
$ is a generalized Killing spinor relative to the connection
$\cD=\nabla^S-\cA$ on $S$, where:
\be
\hat{\cA}_X\eqdef \Psi_\Gamma^{-1}(\cA) =\frac{1}{4} H(X)=\frac{1}{4}
(\ast\rho)(X) = \frac{1}{4} \ast (\rho\wedge X^\flat) \,\,\,\,\,\,
\forall X\in \fX(M)
\ee
and we defined $\rho \eqdef \ast H\in \Omega^1(M)$. The second and third
conditions in \eqref{eq:Hetconf} are linear algebraic constraints on
$\epsilon$. Hence the first three equations of \eqref{eq:Hetconf}
state that $\epsilon$ is a constrained generalized Killing spinor (see
Definition \ref{def:generalizedKS}). As in Sections
\ref{sec:4dLorentzexample} and \ref{sec:GCKLorentz4d}, consider a
signed polyform square of $\epsilon$:
\be
\alpha = u + u\wedge l~~,
\ee
where $(u,l)$ is a parabolic pairs of one-forms.

\begin{lemma}
\label{lemma:dilatinoeq}
$(g,\varphi,H,A,\epsilon)$ satisfies the second equation
in \eqref{eq:Hetconf} (the dilatino equation) iff:
\beqa
& \varphi\wedge u = - \ast (\rho\wedge u)~~,~~ \varphi\wedge u\wedge l
= g^{\ast}(\rho,l) \ast u~~,~~ g^{\ast}(\varphi,l)\, u = \ast (l\wedge
u\wedge \rho)~~,\\ & g^{\ast}(u,\varphi) = 0~~,~~ g^{\ast}(u,\rho) =
0~~,~~ g^{\ast}(\rho,\varphi) = 0~~,
\eeqa
where $\rho \eqdef \ast H$.
\end{lemma}

\begin{proof}
By Proposition \ref{prop:constraintendopoly}, the dilatino
equation holds iff:
\ben
\label{dilatino}
\varphi\diamond \alpha = H\diamond \alpha~~,
\een
where $\alpha = u + u\wedge l$ is a signed polyform square of
$\epsilon$. We compute:
\beqa
\varphi\diamond \alpha &=& g^{\ast}(\varphi,u) + \varphi\wedge u +
\varphi\wedge u \wedge l + g^{\ast}(\varphi,u)\, l -
g^{\ast}(\varphi,l)\, u~~,\\
H\diamond \alpha &=& \nu\diamond\rho\diamond \alpha = 
g^{\ast}(\rho,u)\,\nu - \ast (\rho\wedge u) - \ast(l \wedge u \wedge \rho) + 
g^{\ast}(\rho,u)\, \ast l + g^{\ast}(\rho,l)\, \ast u~~,
\eeqa
where in the second equation we used \eqref{eq:nuaction} to rewrite
the geometric product in terms of $\rho$. Separating degrees in
\eqref{dilatino} and using these relations gives the conclusion.
\end{proof}

\begin{lemma}
\label{lemma:gauginoeq}
$(g,\varphi,H,A,\epsilon)$ satisfies the third equation
in \eqref{eq:Hetconf} (the gaugino equation) iff:
\ben
\label{Gaugino}
F_A = u\wedge \chi_A~~,
\een
where $\chi_A \in \Gamma(u^{\perp}\otimes \fg_P)$ is a
$\fg_P$-valued one-form orthogonal to $u$.
\end{lemma}

\begin{proof}
By Proposition \ref{prop:constraintendopoly}, the gaugino equation
holds iff:
\ben
\label{gaugino}
\alpha\diamond F_A = u\diamond F_A + u\diamond l\diamond F_A = 0~~.
\een
Expanding the geometric product gives:
\be
\alpha\diamond F_A = u\wedge F_A + \iota_u F_A + u\wedge l\wedge F_A -
l\wedge \iota_u F_A + u \wedge \iota_l F_A + \iota_u\iota_l F_A = 0~~.
\ee
Hence separating degrees in \eqref{gaugino} gives the system:
\be
u\wedge F_A = 0~~,~~ \iota_u F_A = 0~~, 
\ee
which is solved by \eqref{Gaugino}.
\end{proof}

\begin{lemma}
\label{lemma:gravitinoeq}
$(g,\varphi,H,A,\epsilon)$ satisfies the first equation in
\eqref{eq:Hetconf} (the gravitino equation) iff:
\ben
\label{Gravitino}
\nabla^g u = \frac{1}{2} \ast (\rho\wedge u)~~,~~ \nabla^g l =
\frac{1}{2} \ast (\rho \wedge l) + \kappa\otimes u~~.
\een
where $\kappa\in \Omega^1(M)$ and $\rho \eqdef \ast H$. In this case,
$u^{\sharp} \in \fX(M)$ is a Killing vector field.
\end{lemma}

\begin{proof}
By Theorem \ref{thm:GCKS}, the gravitino equation holds iff:
\be
\nabla^g_v (u + u\wedge l) = \frac{1}{4} \left[ H(v) \diamond (u +
u\wedge l) - (u + u\wedge l)\diamond H(v) \right] \,\,\,\,\,\,
\forall\,\, v\in \fX(M)~~,
\ee
which reduces to the following upon expanding the geometric
product and separating degrees:
\be
2 \nabla^g_v u + H(v,u) = 0~~,~~ 2\nabla^g_v
 (u\wedge l) + H(v,u)\wedge l + u\wedge H(v,l) = 0~~.
\ee 
This system is equivalent with:
\be
2 \nabla^g_v u + H(v,u) = 0~~,~~ 2\nabla^g_v   
l +  H(v,l) - 2\kappa(v) u = 0 \,\,\,\,\,\, \forall\,\, v\in \fX(M)~~
\ee 
for some one-form $\kappa\in \Omega^1(M)$. In turn, this is equivalent
with \eqref{Gravitino}.
\end{proof}

\begin{thm}
\label{thm:susyheteroticconf}
A quadruplet $(g,\varphi,H,A)$ is a supersymmetric heterotic
configuration for $(M,P,\fc)$ iff there exists a parabolic pair of
one-forms $(u,l)$ such that the following equations (where $\rho
\eqdef \ast H\in \Omega^1(M)$) are satisfied:
\begin{eqnarray}
\label{eq:susyheteroticconf}
& \varphi\wedge u = \ast (\rho\wedge u)~~,~~ \varphi\wedge
u\wedge l = -g^{\ast}(\rho,l) \ast u~~,~~ -g^{\ast}(\varphi,l)\,
u = \ast (l\wedge u\wedge \rho)~~,\nonumber\\ & g^{\ast}(u,\varphi) =
0~~,~~ g^{\ast}(u,\rho) = 0~~,~~ g^{\ast}(\rho,\varphi) =
0~~,~~ F_A = u\wedge \chi_A~~,\\ & \nabla^g u = \frac{1}{2}
u\wedge\varphi ~~,~~ \nabla^g l = \frac{1}{2} \ast (\rho\wedge
l) + \kappa\otimes u~~,~~ \dd^{\ast} \rho = 0~~, \nonumber
\end{eqnarray}
for some one-form $\kappa\in \Omega^1(M)$ and some $\fg_P$-valued
one-form $\chi_A \in \Omega^1(M, \fg_P)$ which is orthogonal to
$u$. In this case, $u^{\sharp} \in \fX(M)$ is a Killing vector field
and the distribution $\ker u\subset TM$ integrates to a
transversely-orientable codimension one foliation of $M$.
\end{thm}

\begin{proof}
By Lemmas \ref{lemma:dilatinoeq}, \ref{lemma:gauginoeq} and
\ref{lemma:gravitinoeq}, it suffices to prove the last equation in the
third line of \eqref{eq:susyheteroticconf}. The modified Bianchi
identity can be written as:
\be
\dd \ast \rho = \ast^2 \fc(F_A\wedge F_A)~~,
\ee
and gives:
\be
\dd^{\ast} \rho = \fc(\iota_u \chi_A , \iota_u \chi_A ) = \fc(\chi_A(u), \chi_A(u))= 0~~,
\ee
because $\chi_A$ is orthogonal to $u$. The first equation in the the third line of
\eqref{eq:susyheteroticconf} amounts to:
\be
\cL_{u^{\sharp}}g = 0~~,~~ \dd u = \frac{1}{8} u\wedge \varphi~~,
\ee
showing that $u^{\sharp}$ is Killing and the distribution $\ker
u\subset TM$ is integrable, giving a transversely-orientable foliation
$\cF_u \subset M$ of codimension one.
\end{proof}

\subsection{Some examples}

\begin{example}
Let:
\be
(M,\dd s^2_g) = \left(\R^2\times X, 2\, \dd x^v \dd x^u + q(x^v)\right)~~, 
\ee
where $q(x^v)$ is a flat Riemannian metric on
$X$ for all $x^v\in \R$ and take $P$ to be the unit principal bundle over
$M$. Consider the spinor $\epsilon$ corresponding to the pair $(u,l)$,
where $u = \dd x^v$ and $l=l(x^v)$ depends on $x^v$. Finally, take:
\be
\varphi = \Omega\, u~~,~~ \rho = 0~~, 
\ee
where $\Omega \in C^{\infty}(\R^2\times X)$. A short computation shows 
that equations \eqref{eq:susyheteroticconf} reduce to:
\be
\nabla^g l = \kappa\otimes \dd x^v~~.
\ee
Applying this to $\partial_{x^v}$, $\partial_{x^u}$ and restricting to $TX$ gives:
\begin{eqnarray*}
& \nabla^g_{x^v} l = \partial_{x^v} l - \frac{1}{2} l^{\sharp} \lrcorner
\partial_{x^v} q(x^v) = 0~~,~~  \kappa(\partial_{x^v}) = 0~~, 
~~ \kappa(\partial_{x^u}) = 0~~, \\
& \nabla^g l\vert_{T^{\ast}X} = \nabla^q l - \frac{1}{2} l^{\sharp}\lrcorner 
\partial_{x^v}q(x^v)  \dd x^v = (\kappa_X - \ast_{q(x^v)} l)\otimes \dd x^v~~,
\end{eqnarray*}
where $\kappa_X\eqdef \kappa|_{TX}$. This implies:
\be
\nabla^q l = 0~~,~~ \kappa_X  = \ast_{q(x^v)} l - \frac{1}{2} l^{\sharp}\lrcorner \partial_{x^v}q(x^v)~~,
\ee
showing that $(X,q(x^v))$ is flat for all $x^v\in
\R$. The only remaining non-trivial condition is:
\ben
\label{eq:evolutionexample}
\partial_{x^v} l = \frac{1}{2} l^{\sharp} \lrcorner\partial_{x^v} q(x^v)~~.
\een
This is a linear first order ordinary differential equation for the
function $x^v\rightarrow l(x^v)$. For every choice of parallel vector
field on $(X,q(x^v_0))$ with fixed $(x^v_0,x^u_0)\in \R^2$, its
solution with the corresponding initial condition determines a
one-parameter family of one forms $\left\{ l(x^v)\right\}_{x^v\in \R}$
on $(X,q(x^v))$. Assuming for instance that $X$ is simply connected
and that $q(x^v)$ satisfies:
\be
\partial_{x^v}q(x^v) = 2 F(x^v)\, q(x^v)~~,
\ee
for some function $F(x^v)$ depending only on $x^v$, then the explicit
solution is:
\be
l(x^v) = e^{\int F(x^v)} l_0~~,
\ee
where $l_0$ is parallel vector field on $(X,q(x^v_0))$ and we 
canonically identify the tangent vector spaces of $\left\{
(x^v,x^u) \right\}\times X$ at different points $(x^v,x^u)$.
\end{example}

\begin{remark}
Heterotic solutions with exact null dilaton were considered before
(see \cite{Bergshoeff:1992cw}).  As remarked earlier, a supersymmetric
heterotic configuration need not be a solution of the equation of
motion given in Appendix \ref{app:HetSugra}. The classification of
(geodesically complete) supersymmetric heterotic solutions on a
Lorentzian four-manifold and the diffeomorphism type of four-manifolds
admitting such solutions for fixed principal bundle topology is an
open problem. Appendix \ref{app:HetSugra} gives a brief formulation of
abstract bosonic heterotic supergravity and its Killing spinor
equations.
\end{remark}


\appendix


\section{Parabolic 2-planes and degenerate complete flags in $\R^{3,1}$}
\label{app:flags}

\noindent Let $(V,h)$ be a four-dimensional Minkowski space of
``mostly plus'' signature. A non-zero subspace $W\subset V^\ast$ is
called degenerate or nondegenerate according to whether the
restriction $h^\ast_W$ of $h^\ast$ to $W$ is a degenerate or
non-degenerate quadratic form. A non-degenerate subspace $W$ is
called:
\begin{itemize}
\itemsep 0.0em
\item positive or negative {\em definite}, if the restriction
  $h^\ast_W$ to $W$ is positive or negative negative definite,
  respectively
\item {\em hyperbolic}, if the restriction of $h^\ast$ to $W$ is not
  positive or negative definite.
\end{itemize}
Notice that $W$ is partially isotropic (i.e. contains nonzero null
vectors) iff it is degenerate or hyperbolic. Let $\cL$ denote the cone
of causal (i.e. non-spacelike) vectors in $(V^\ast, h^\ast)$. A
non-zero subspace $W\subset V^\ast$ is:
\begin{itemize}
\itemsep 0.0em
\item hyperbolic iff $\dim(W\cap \cL)>1$, i.e. iff $W$ meets $\cL$
  along a sub-cone of the latter which has dimension at least two.
\item degenerate iff $\dim(W\cap \cL)=1$, i.e. iff $W$ is tangent to
  $\cL$ along a null line
\item non-degenerate iff $W\cap \cL=\{0\}$, in which case $W$ is
  spacelike (i.e. positive definite).
\end{itemize}
A degenerate subspace $W$ of $V^\ast$ contains no timelike vectors and
the set of its null vectors coincides with the kernel $\K_h(W)\eqdef
\ker(h^\ast_W)$; accordingly, $W$ decomposes as:
\be
W=\K_h(W)\oplus U~~,
\ee
where $\K_h(W)=\ker(h^\ast_W)$ coincides with the unique null line
contained in $W$ and $U$ is a spacelike subspace of $V^\ast$ which is
orthogonal to $\K_h(W)$. In particular, we have $W\subset \K_h(W)^\perp$.
For example, a 2-plane $\Pi\subset V^\ast$ can be spacelike,
hyperbolic or degenerate, according to whether 
$h^\ast_\Pi$ is positive-definite, non-degenerate of signature
$(1,1)$ or degenerate. In the latter case, $\Pi$ is called a
{\em parabolic 2-plane}.

\begin{definition}
A complete flag $0\subset W_{(1)} \subset W_{(2)}\subset W_{(3)}
\subset V^\ast$ is called {\em degenerate} if $W_{(i)}$ is a
degenerate subspace of $(V^\ast, h^\ast)$ for each $i=1,2,3$.
\end{definition}

\noindent Notice that a complete flag $0\subset W_{(1)} \subset
W_{(2)}\subset W_{(3)} \subset V^\ast$ is determined by the
increasing sequence of vector spaces $W_\bullet\eqdef
(W_{(1)},W_{(2)},W_{(3)})$.  Such a flag is degenerate iff $W_{(2)}$
and $W_{(3)}$ are tangent to the light cone of $(V^\ast, h^\ast)$
along the line $W_{(1)}$ (whence $W_{(1)}$ is a null line).

\begin{prop}
\label{prop:hypflag}
There exists a natural bijection between the set of parabolic 2-planes
and the set of degenerate complete flags in $(W^\ast,h^\ast)$. This
associates to each parabolic 2-plane $\Pi\subset V^\ast$ the unique
degenerate complete flag $0\subset W_{(1)} \subset W_{(2)}\subset
W_{(3)} \subset V^\ast$ with $W_{(2)}=\Pi$, which is given by:
\ben
\label{eq:degflag}
W_{(1)}=\K_h(\Pi)~~,~~W_{(2)}=\Pi~~,~~W_{(3)}=\K_h(\Pi)^\perp~~.
\een
\end{prop}

\begin{proof}
Let $0\subset W_{(1)} \subset W_{(2)}\subset W_{(3)} \subset V^\ast$
be a degenerate complete flag. Then the 2-plane $\Pi\eqdef W_{(2)}$ is
degenerate, i.e. parabolic. Since the subspace $W_{(1)}$ of
$W_{(2)}=\Pi$ is degenerate and one-dimensional, it is a null line and
hence coincides with $\K_h(\Pi)$. Since $W_{(3)}$ is degenerate, the
one-dimensional subspace $\K_h(W_{(3)})$ is the unique null line
contained in $W_{(3)}$ and hence coincides with $W_{(1)}$. We thus
have $W_{(3)}\subset W_{(1)}^\perp$. This inclusion is an equality
because $\dim W_{(1)}^\perp=\dim V^\ast -1=3=\dim W_{(3)}$. Hence any
degenerate complete flag has the form \eqref{eq:degflag} for a unique
parabolic 2-plane $\Pi$, namely $\Pi=W_{(2)}$. It is clear that the
correspondence thus defined is a bijection.
\end{proof}

\begin{definition}
\label{def:framedflag}
A {\em co-oriented} degenerate complete flag in $(V^\ast, h^\ast)$ is a
pair $(W_\bullet, L)$, where $W_\bullet=(W_{(1)},W_{(2)},W_{(3)})$ is
a degenerate complete flag in $(V^\ast, h^\ast)$ and $L$ is a co-orientation
of the parabolic two-plane $W_{(2)}$.
\end{definition}

\noindent Corollary \ref{cor:MinkHyp} and Proposition \ref{prop:hypflag}
imply the following reformulation of Theorem \ref{thm:squarespinorMink}:

\begin{thm}
\label{thm:squarespinorMinkflag}
There exists a natural bijection between $\P(\Sigma)$ and the set of
all {\em co-oriented} degenerate complete flags in $(V^\ast, h^\ast)$.
\end{thm}

\section{Heterotic supergravity in four Lorentzian dimensions}
\label{app:HetSugra}


\noindent Let $G$ be a compact semisimple real Lie group whose Lie
algebra we denote by $\fg$ and whose adjoint representation we denote
by $\Ad_G:G\rightarrow \GL(\fg)$. Let $\fg=\fg_1\oplus \ldots \oplus
\fg_k$ be the decomposition of $\fg$ into simple Lie algebras. Then
any $\Ad_G$-invariant non-degenerate symmetric pairing $\fc$ on $\fg$
can be written as:
\be
\fc=c_1 B_1\oplus \ldots \oplus c_k B_k~~,
\ee
where $B_j$ is the Killing form of $\fg_j$ and $c_j$ are non-zero
constants.

\begin{definition}
A {\em four-dimensional heterotic datum} of type $G$ on $M$ is a
triplet $(M,P,\fc)$, where $M$ is an oriented open four-manifold, $P$
is a principal bundle over $M$ and $\fc$ is a non-degenerate symmetric
and $\Ad_G$-invariant bilinear pairing on $\fg$.
\end{definition}

\noindent Let $(M,P,\fc)$ be a four-dimensional heterotic datum. Let
$\fg_P\eqdef P\times_{\Ad_G}\fg$ be the adjoint bundle of Lie algebras
of $P$ and $\cA_P$ be the affine space of connections on $P$. For any
connection $A\in \cA_P$, let $F_A \in \Omega^2(M,\fg_P)$ denote the
curvature form of $A$. Let $\fc_P$ be the pairing induced by $\fc$ on
the adjoint bundle $\fg_P$. Since $\fc$ is $\Ad_G$-invariant, the
latter can be viewed as a morphism of vector bundles:
\be
\fc_P:\fg_P\otimes \fg_P\rightarrow \R_M~~,
\ee
where $\R_M$ is the trivial real line bundle on $M$. Let
$\Lambda(M,\fg_P)\eqdef \Lambda(M)\otimes \fg_P$ and consider the
vector bundle morphism:
\be
\fc(- \wedge -)\eqdef \wedge \otimes \fc: \Lambda(M,\fg_P)\otimes
\Lambda(M,\fg_P)\to \Lambda(M)~~,
\ee
where $\wedge:\Lambda(M)\otimes \Lambda(M)\rightarrow \Lambda(M)$ is
the wedge product. Let $\Met_{3,1}(M)$ be the set of metrics of
signature $(3,1)$ defined on $M$. Let:
\be
\langle~,~\rangle_g:\Lambda(M)\times \Lambda(M)\rightarrow
\R_M~~\mathrm{and}~~\langle~,~\rangle_{g,\fc}=\langle~,~\rangle_g\otimes \fc_P:\Lambda(M,
\fg_P)\times (\Lambda(M,\fg_P)\rightarrow \R_M
\ee
be the metrics induced by $g\in \Met_{3,1}(M)$ and $\fc_P$ on the vector
bundles $\Lambda(M)$ and $\Lambda(M,\fg_P)$. Consider the symmetric
bilinear maps corresponding to the vector bundle morphisms:
\be
\circ: \Lambda(M)\otimes \Lambda(M)\rightarrow T^\ast M\otimes T^\ast
M~~\mathrm{and}~~\fc(- \circ -)\colon \Lambda(M,\fg_P)\otimes
\Lambda(M,\fg_P)\to T^\ast M\otimes T^\ast M
\ee
whose action on global sections is given by:
\beqa
(\omega\circ \eta)(X,Y) &\eqdef& \, \langle \iota_X\omega,
\iota_Y\eta\rangle_g \,\,\,\,\,\, \forall \omega,\eta\in
\Omega(M)~~\forall X,Y\in \fX(M)~~\nn\\ \fc(\alpha\circ \beta)(X,Y)
&\eqdef & \, \langle \iota_X\alpha, \iota_Y\beta\rangle_{g,\fc}
\,\,\,\,\,\, \forall \alpha,\beta \in \Omega(M,\fg_P)~~\forall
X,Y\in \fX(M)~~.
\eeqa
For a three-form $H\in\Omega^3(M)$ and the curvature $F_A\in
\Omega^2(\fg_P)$ of a connection $A\in\cA_P$, we have:
\beqa
(H\circ H)(X,Y) &=& \langle \iota_X H, \iota_Y H\rangle_g \,\,\,\,\,\,
\forall X, Y \in \fX(M)~~,\nn\\ \fc(F_A\circ F_A)(X,Y) &=& \langle
\iota_X F_A, \iota_Y F_A\rangle_{g,\fc} \,\,\,\,\,\, \forall X, Y \in
\fX(M)~~.
\eeqa

\begin{remark}
Pick a nonempty open subset $U\subset M$ which supports local
coordinates $\{x^i\}_{i=1,\ldots, 4}$ of $M$ as well as a local
$g$-orthonormal frame $\{e_i\}_{i=1,\ldots, 4}$ of $TM$ and a
$\fc$-orthogonal local frame $\{T_a\}_{a=1,\ldots, \rk g}$ of $\fg_P$
such that $\fc(T_a,T_b)=\delta_{ab}\, \varepsilon_a$, where
$\varepsilon_a\in \{-1,1\}$. Then the following relations hold on $U$:
\be
\fc(F_A\wedge F_A)=\varepsilon_a F_A^a\wedge
F_A^a~~,~~\fc(F_A\circ F_A)_{ij}= \varepsilon_a (F_A^a)_{im}\,
(F_A^a)_{jk} \, g^{mk}~~,~~~~(H\circ H)_{ij} = H_{ilm} H_{j}^{\,\,\,
  lm}~~,
\ee
where $F_A=F^a_A T_a$ and we use Einstein summation over $i,j,l,m=1, \ldots, 4$ and
over $a=1,\ldots \rk \, \fg$.
\end{remark}

\noindent Let $\Omega^{1}_\cl(M)$ be the space of closed one-forms
defined on $M$.

\begin{definition}
A {\em bosonic heterotic configuration} of $(M,P,\fc)$ is a quadruplet
$(g,\varphi,H,A)\in \Met_{3,1}(M)\times \Omega^1_\cl(M)\times
\Omega^3(M)\times \cA_P$ which satisfies the {\em modified Bianchi
  identity}:
\ben
\label{eq:BianchiT}
\dd H = \fc\left(F_A \wedge F_A\right)~~.
\een
\end{definition}

\noindent The one-form $\varphi\in \Omega^1_\cl(M)$ is called the {\em
  dilatonic one-form} of the bosonic heterotic configuration. On
sufficiently small non-empty open subsets $U\subset M$, it can be written as
$\varphi=\dd \phi$, where the locally defined function $\phi\in
\cC^\infty(U)$ corresponds to the {\em dilaton} of the physics
literature.

\begin{definition}
A bosonic heterotic configuration $(g,\varphi,H,A)$ of $(M,P,\fc)$ is
called {\em bosonic heterotic solution} if it satisfies the {\bf
  equations of motion}:
\beqa
\Ric^{g} + \nabla^{g}\varphi - \frac{1}{4} H \circ H -
\fc(F_A \circ F_A) & = & 0~~,\\ \dd^{\ast} H + \iota_{\varphi}
H & = & 0~~,\\ \dd_A \ast F_A - \varphi\wedge \ast F_A + F_A \wedge
\ast H & = & 0~~,\\ \dd^{\ast}\varphi + \langle \varphi,\varphi\rangle_g - \langle H,H\rangle_g -
\langle F_A,F_A\rangle_{g,\fc} & = & 0~~.
\eeqa
\end{definition}

\noindent Let $\Conf(M,P,\fc)$ and $\Sol(M,P,\fc)\subset
\Conf(M,P,\fc)$ be the sets of bosonic configurations
and solutions of the heterotic datum $(M,P,\fc)$.

\begin{remark}
If $\G$ is the trivial group, then $P$ is the unit principal bundle
over $M$, which has total space $M$ and projection given by the
identity map $\id_M$. In this case, $\fc$ vanishes (as does any
connection on $P$) and the set of bosonic configurations reduces to:
\be
\Conf_0(M)=\Met_{3,1}(M)\times \Omega^1_\cl(M)\times \Omega^3_\cl(M)~~,
\ee
since the modified Bianchi identity requires $\dd H=0$. Moreover,
the equations of motion reduce to:
\ben
\label{eq:motionHetsugraNSNS}
\begin{split}
\operatorname{Ric}^{g} + \nabla^{g}\varphi - \frac{1}{4} H \circ H =
0~~,~~ \dd^{\ast} H + \iota_{\varphi} H = 0~~,~~
\dd^{\ast}\varphi + \langle \varphi,\varphi\rangle_g - \langle H,H\rangle_g = 0~~.
\end{split}
\een
This particular case is known as {\em NS-NS supergravity}.
\end{remark}

For any three-form $H$ on $M$, let $\nabla^H$ be the unique
$g$-compatible connection on $TM$ with totally skew-symmetric torsion
given by $T=H^\sharp$, where $H^\sharp\in \Gamma(T^\ast M\otimes T^\ast
M\otimes TM)$ is defined through: 
\be
H^\sharp(X,Y)=(\iota_Y\iota_X H)^\sharp\in \fX(M)\,\,\,\,\,\, \forall X,Y\in \fX(M)~~.
\ee
Here $~^\sharp$ denotes raising of indices with respect to $g$. This
connection is given by:
\be
\nabla^H = \nabla^g  - \frac{1}{2} H^\sharp~~.
\ee
Assume that $M$ admits strongly spinnable Lorentzian metrics, whose
space we denote by $\Met_{3,1}^\ss(M)$.  We shall assume for
simplicity that $H^1(M,\Z_2)=0$, although this can be relaxed. For any
$g\in \Met_{3,1}^\ss(M)$, let $(S_g,\Gamma_g)$ be a spinor bundle on
$(M,g)$. The assumption $H^1(M,\Z_2)=0$ implies that this spinor
bundle is unique up to isomorphism.  

\begin{definition}
A bosonic heterotic configuration $(g,\varphi,H,A)\in\Conf(M,P,\fc)$
is called {\bf supersymmetric} if $g\in \Met_{3,1}^\ss(M)$ and there
exists a nontrivial spinor $\epsilon\in \Gamma(S_g)$ which satisfies the
{\em Killing spinor equations}:
\ben
\label{eq:susytransheterotic}
{\hat \nabla}^{H}\epsilon = 0~~,~~ \varphi\cdot \epsilon = H\cdot\epsilon = 0~~,~~
F_A\cdot \epsilon = 0~~.
\een
\end{definition}

\

\noindent A similar formulation can be given for heterotic supergravity
on a ten-dimensional open manifold.

\phantomsection
\bibliographystyle{JHEP}

\end{document}